\documentclass[11pt,reqno]{amsart}
\usepackage{etex}
\usepackage{amsmath,amsthm,amsfonts,amssymb,amscd,flafter,pinlabel, booktabs}
\usepackage[mathscr]{eucal}
\usepackage{graphics}
 \usepackage[all,knot,arc]{xy}
 \usepackage[normalem]{ulem}

\usepackage[usenames,dvipsnames]{xcolor}
\usepackage{subfigure}
\usepackage{graphicx}
 \usepackage{tabu}

\usepackage{bbm}
\usepackage{mathtools}

\usepackage{pgf}
\usepackage{tikz}
\usepackage{tikz-cd}
\usetikzlibrary{arrows,automata}
\usepackage[latin1]{inputenc}

\allowdisplaybreaks

\usepackage[colorlinks=true,linkcolor=blue,citecolor=blue,urlcolor=blue]{hyperref}

\tikzset{cdlabel/.style={above,sloped,
    execute at begin node=$\scriptstyle,execute at end node=$}}    
\tikzset{al/.style={->, bend right=45, thick}}
\tikzset{ar/.style={->, bend left=45, thick}}

\normalfont\upshape

\usepackage[lmargin=1in,rmargin=1in,tmargin=1in,bmargin=1in]{geometry}

\theoremstyle{definition}
\newtheorem{definition}[equation]{Definition}
\newtheorem*{remark*}{Remark}
\newtheorem*{example*}{Example}
\newtheorem{example}[equation]{Example}
\newtheorem{remark}[equation]{Remark}

\theoremstyle{plain}
\newtheorem{theorem}[equation]{Theorem}

\newtheorem{proposition}[equation]{Proposition}
\newtheorem{lemma}[equation]{Lemma}

\numberwithin{equation}{section}
\numberwithin{figure}{section}
\numberwithin{figure}{section}

\newtheorem{ourtheorem}{Theorem} 
\newtheorem{ourcorollary}[ourtheorem]{Corollary}

\newcommand\fol{\mathcal{F}}

\newcommand{\R}{\ensuremath{\mathbb{R}}}

\newcommand\CFt{{\rm {\widetilde{CF}}}}
\newcommand\HFt{{\rm {\widetilde{HF}}}}

\newcommand\SFC{{\rm {SFC}}}

\newcommand\alphas{\mbox{\boldmath$\alpha$}}
\newcommand\betas{\mbox{\boldmath$\beta$}}

\newcommand\F{\mathbb F}

\newcommand\HD{\mathcal H}

\newcommand\pit{{\widetilde{\pi}}}
\newcommand\bsGamma{\mathsf{\Gamma}}

\def\x{\mathbf{x}}

\def\y{\mathbf{y}}
\def\z{\mathbf{z}}

\def\ip#1{{\textcolor{teal}{#1}}}

\def\co{\colon\thinspace} 

\newcommand{\bdy}{\partial}

\newcommand{\sqtens}{\boxtimes} 
\newcommand{\ol}{\overline}

\newcommand{\xx}{{\bf x}}

\DeclareMathOperator{\Int}{Int}

\newcommand{\HH}{\mathcal{H}}

\newcommand{\zz}{\mathcal Z}

\newcommand{\xxx}{\mathbf{x}}

\newcommand{\cf}{\mathit{CF}}
\newcommand{\hf}{\mathit{HF}}

\newcommand{\bsd}{\mathit{BSD}}
\newcommand{\bsa}{\mathit{BSA}}

\newcommand{\sfh}{\mathit{SFH}}
\newcommand{\sfc}{\mathit{SFC}}
\newcommand{\cfhat}{\widehat{\cf}}
\newcommand{\hfhat}{\widehat{\hf}}

\newcommand{\bsdhat}{\widehat{\bsd}}
\newcommand{\bsahat}{\widehat{\bsa}}

\newcommand{\brho}{\boldsymbol\rho}

\begin{document}

\title[{Bordered Floer homology and contact structures}]{Bordered Floer homology and contact structures}

\author[A. Alishahi]{Akram Alishahi}
\address {Department of Mathematics, University of Georgia\\ Athens, GA 30606}
\email {akram.alishahi@uga.edu}
\thanks{AA was supported by NSF Grant DMS-2019396.}
\author[V. F\"oldv\'ari]{Vikt\'oria F\"oldv\'ari}
\address {Alfr\'ed R\'enyi Institute of Mathematics, Budapest}
\email {foldvari@renyi.hu}
\thanks{VF was supported by the
	NKFIH \'Elvonal (Frontier) project KKP126683.}
\urladdr{\href{http://fviktoria.web.elte.hu/}{http://fviktoria.web.elte.hu/}}
\author[K. Hendricks]{Kristen Hendricks}
\address {Department of Mathematics, Rutgers University \\ New Brunswick, NJ 08901}
\email{kristen.hendricks@rutgers.edu}
\thanks{KH was supported by NSF CAREER Grant DMS-2019396 and a Sloan Research Fellowship.}
\urladdr{\href{http://math.rutgers.edu/~kh754}{www.math.rutgers.edu/~kh754}}
\author[J. Licata]{Joan Licata}
\address {Mathematical Sciences Institute, The Australian National University\\ Canberra, Australia}
\email {joan.licata@anu.edu.au}
\urladdr{\href{https://sites.google.com/view/joanlicata/home}{https://sites.google.com/view/joanlicata/home}}
\author[I. Petkova]{Ina Petkova}
\address {Department of Mathematics, Dartmouth College\\ Hanover, NH 03755}
\email {ina.petkova@dartmouth.edu}
\thanks{IP was supported by NSF Grant DMS-1711100.}
\urladdr{\href{http://math.dartmouth.edu/~ina}{http://math.dartmouth.edu/~ina}}
\author[V. V\'ertesi]{Vera V\'ertesi}
\address {Faculty of Mathematics, University of Vienna}
\email {vera.vertesi@univie.ac.at}
\urladdr{\href{http://mat.univie.ac.at/~vertesi}{www.mat.univie.ac.at/~vertesi}}
\thanks{VV was supported by the ANR grant ``Quantum topology and contact geometry''}


\begin{abstract}
We introduce a contact invariant in the bordered sutured Heegaard Floer homology of a three-manifold with boundary.  The input for the invariant is a contact manifold $(M, \xi, \mathcal{F})$ whose convex boundary is equipped with a signed singular foliation $\mathcal{F}$ closely related to the characteristic foliation.  Such a manifold admits a family of foliated open book decompositions classified by a Giroux Correspondence, as described in \cite{LV}.  We use a special class of foliated open books to construct admissible bordered sutured Heegaard diagrams and identify well-defined classes $c_D$ and $c_A$ in the corresponding bordered sutured modules. 

Foliated open books exhibit user-friendly gluing behavior, and we show that the pairing on invariants induced by gluing compatible foliated open books recovers the Heegaard Floer contact invariant for closed contact manifolds.  We also consider a natural map associated to forgetting the  foliation $\mathcal{F}$ in favor of the dividing set, and show that it maps the bordered sutured invariant to the contact invariant of a sutured manifold defined by Honda-Kazez-Mati\'{c}.

\end{abstract}

\maketitle

\tableofcontents


\section{Introduction}
\label{sec:intro}

A great deal of interesting and beautiful mathematics has been devoted to understanding the fundamental dichotomy in three-dimensional contact geometry: the subdivision of contact structures into tight and overtwisted. Overtwisted structures are determined by homotopical data and thus may be addressed by tools from algebraic topology. In contrast,  tight contact structures do not satisfy an $h$-principle, and many existence and classification questions for tight contact structures are still open.  Tight structures are nevertheless extremely natural objects of study, as they include the contact structures arising as the boundary of a complex or symplectic manifold. 

Many of the recent advances in classifying tight contact structures were made possible by the advent of Heegaard Floer homology in the early 2000s and the subsequent development of Floer-theoretic contact invariants.  Using open books, Ozsv\'ath and Szab\'o defined an invariant of closed contact three-manifolds \cite{oszc}. Given a closed, contact manifold $(M,\xi)$, this invariant is a class $c(\xi)$ in the Heegaard Floer homology $\hfhat(-M)$. 
In \cite{HKM09_HF}, Honda, Kazez, and Mati\'c gave an alternative description of $c(\xi)$, again using open books. This ``contact class" was used to show that knot Floer homology detects both genus \cite{hfkg} and  fiberedness \cite{pgf, ynf}. It gives information about overtwistedness:    if $\xi$ is overtwisted, then $c(\xi) = 0$, whereas  if $\xi$ is Stein fillable, then $c(\xi)\neq 0$ \cite{oszc}.  The contact class was also used to distinguish notions of fillability: Ghiggini used it to construct examples of strongly symplectically fillable contact three-manifolds which do not have Stein fillings \cite{ghfill}.

Contact manifolds with convex boundary can be  partitioned  by the equivalence class of the dividing set on their boundary.  Given a contact manifold $M$ with convex boundary and dividing set $\Gamma$, Honda, Kazez, and Mati\'c used partial open books to define an invariant $c(M, \Gamma, \xi)$ of the equivalence class in the sutured Floer homology $\sfh(-M, -\Gamma)$ \cite{hkm09}. 
They also defined a gluing map for sutured Floer homology that respects the contact invariants \cite{hkm08}. This map requires the 
Heegaard diagrams to satisfy a number of technical conditions, collectively referred to as ``contact-compatibility.'' Establishing contact-compatibility for specific examples is difficult, so in practice,   computations with the gluing map are rarely possible. As a result, most applications of this gluing map have relied only on its formal properties. 

Gluing techniques -- and Heegaard Floer theory more broadly --   benefitted soon after from the introduction of \emph{bordered Floer homology}, a new theory for manifolds with boundary defined by Lipshitz, Ozsv\'ath, and Thurston in \cite{LOTbook}.   Although bordered Floer theory has produced a wide variety of new results,  (e.g. \cite{hrrw, lsp, lev9, lev14, hom-tau, cables, AL-bs}), applications to contact topology remain mostly uncharted territory. 

To a three-manifold with parametrized boundary, bordered Floer homology associates an $\mathcal A_{\infty}$-module (or \emph{type A structure}), or equivalently, a \emph{type D structure} over a differential graded algebra associated to the parametrization.  When manifolds are glued along compatible parameterized boundaries, the derived tensor product of their bordered invariants recovers the Heegaard Floer homology of the closed three-manifold, up to homotopy equivalence. Zarev introduced a generalization, \emph{bordered sutured Floer homology}, which is an invariant of three-manifolds whose boundary is ``part sutured, part parametrized" \cite{bs}. Bordered sutured Floer homology similarly associates to a \emph{bordered sutured manifold} a type $A$ structure and a type $D$ structure such that taking derived tensor products recovers the sutured Floer homology of the manifold formed by gluing along the parameterized parts of the boundary.

\subsection{Results}
In this paper, we define a contact invariant in the bordered sutured Heegaard Floer homology of a contact three-manifold with boundary.  We consider contact manifolds whose convex boundary is equipped with a certain type of singular foliation.  As shown in \cite{LV}, to any such contact manifold with foliated boundary, one may associate a topological decomposition known  as a \emph{foliated open book}. Intuitively, foliated open books are constructed by cutting ordinary open books along separating convex surfaces. 
The pages and binding of the resulting pair of foliated open books are simply the restrictions of the pages and binding of the original open book to the corresponding piece.  The intersection of the cutting surface with the pages determines an ordered signed singular foliation, here called the boundary foliation. This induced ``open book foliation''  is closely related to the characteristic foliation of a supported contact structure and has been extensively studied in the work of Ito--Kawamuro, e.g., \cite{IK1, IK2}. Under mild technical hypotheses, the topological data of the resulting foliated open book  uniquely determines the restriction of the original contact structure to each piece, up to isotopy.

We associate to a foliated  contact three-manifold  $(M,\xi, \fol)$  a bordered sutured manifold $(M, \bsGamma, \mathcal{Z})$.  We then show  that the data of a sorted foliated open book for $(M,\xi, \fol)$ gives rise to an admissible bordered sutured Heegaard diagram for the manifold $(M, \bsGamma, \mathcal{Z})$, and so via dualizing, an admissible bordered sutured Heegaard diagram for $(-M,\bsGamma,\overline{\zz})$, which we will call $\HD$. Moreover, we define preferred generators $\x_D$ and $\x_A$ in $\bsdhat(\HD)$ and $\bsahat(\HD)$, respectively. The structures $\bsdhat(\HD)$ and $\bsahat(\HD)$ give invariants of $(-M,\bsGamma,\overline{\zz})$ up to homotopy equivalence. (Some further algebraic properties of these elements may be found in Proposition~\ref{prop:xd-def}.) We further define an \emph{equivalence} between elements of two homotopy equivalent type $A$ structures or between elements of two homotopy equivalent type $D$ structures in Section \ref{sssec:bs}.  Our main result is that these preferred generators are invariants of the contact structure up to this equivalence.

  \begin{ourtheorem}\label{thm:ca-cd}
Let  $(M,\xi, \fol)$  be a foliated contact three-manifold with associated bordered sutured manifold $(M, \bsGamma, \mathcal{Z})$. Given admissible bordered sutured Heegaard diagrams $\HD$ and $\HD'$ for $(-M,\bsGamma,\overline{\zz})$ with 
preferred generators $\x_A\in \bsahat(\HD)$ and $\x'_A\in \bsahat(\HD')$, there is a homotopy equivalence between   $\bsahat(\HD)$ and $\bsahat(\HD')$ induced by Heegaard moves, under which $\x_A$ and $\x'_A$ are equivalent; an analogous statement holds for the preferred generators $\x_D\in \bsdhat(\HD)$ and  $\x'_D\in \bsdhat(\HD')$. We refer to these type $A$ and type $D$ equivalence classes as $c_A(M,\xi, \fol)$ and $c_D(M,\xi, \fol)$, respectively.
 \end{ourtheorem}

 More precisely, we show  that varying the choices made in our construction induces  homotopy equivalences between the bordered sutured Floer homologies associated to the resulting Heegaard diagrams; these maps carry the preferred generator of one module to the preferred generator of the other module.  


An ordered signed singular foliation on the convex boundary of a contact manifold determines a dividing set, but such foliations also induce a finer partition on the set of contact manifolds with convex boundary. This additional data improves the compatibility with cut-and-paste operations,  as the ordered signed singular foliation carries data about both the contact structure and the foliated open book near the boundary.  Given a pair of foliated contact three-manifolds $(M^L, \xi^L, \fol^L)$ and $(M^R, \xi^R, \fol^R)$  whose  foliated boundaries agree in an appropriate sense, there is a canonical perturbation of the contact structure near the boundary so that the pieces   glue  to  a closed contact three-manifold $(M, \xi)$.  In fact, the foliated open books supporting these manifolds also glue to an open book for the resulting closed contact manifold; see Section~\ref{ssec:gluing}.  Because our bordered sutured contact invariant is sensitive not only to the dividing set of a convex boundary but also to the singular foliation, it behaves nicely with respect to these cut-and-paste operations.

We prove that  the contact invariants of the two foliated contact three-manifolds pair  to recover the contact invariant  of $(M, \xi)$.

\begin{ourtheorem}
\label{thm:glue}
The tensor product $c_D(M^L,\xi^L,\fol^L)\boxtimes c_A(M^R,\xi^R,\fol^R)$ recovers the contact invariant $c(M,\xi)$.
 \end{ourtheorem} 
A more precise version of this statement on the level of generators is given in Theorem~\ref{thm:gluing}. 

 One may also choose to forget the singular foliation and retain only the data of the dividing set on the convex boundary; this is captured by a natural map from a foliated open book to a partial open book.
On the level of Heegaard diagrams, this corresponds to converting a bordered sutured Heegaard diagram to a sutured Heegaard diagram; the procedure to do so was described by Zarev in \cite{bs-JG} and induces an isomorphism 
\[\sfh(-M,-\Gamma(\fol))\cong H_\ast\left(\bsahat(-M,\bsGamma,\overline{\zz})\right)\cdot \iota_+,\]
where  $\iota_+$ is an idempotent naturally determined by the foliation data. We show that under Zarev's isomorphism the bordered sutured invariant associated to a foliated open book  maps to the contact invariant in the sutured Floer homology associated to the corresponding partial open book.

\begin{ourtheorem}
\label{thm:hkm} 
Under the above isomorphism, $c_A(M,\xi, \fol)\cdot \iota_+$ is identified with the contact invariant $\mathrm{EH}(M,\Gamma(\fol), \xi)$ from  \cite{hkm09}. 
\end{ourtheorem}

In particular $EH(M,\Gamma(\fol), \xi)$ vanishes in $\sfh(-M,-\Gamma(\fol))$ if and only if $c_A(M,\xi, \fol)\cdot \iota_+$ is zero in  $H_\ast\left(\bsahat(-M,\bsGamma,\overline{\zz})\right)\cdot \iota_+$. Since  
$c_A(M,\xi,\fol)=c_A(M,\xi, \fol)\cdot \iota_+$ and the differential of $\bsahat(-M,\bsGamma,\overline{\zz})$ respects the splitting by the idempotents, the class $c_A(M,\xi, \fol)\cdot \iota_+$ being zero in  $H_\ast\left(\bsahat(-M,\bsGamma,\overline{\zz})\right)\cdot \iota_+$ is equivalent to the class $c_A(M,\xi, \fol)$ being zero in $H_\ast\left(\bsahat(-M,\bsGamma,\overline{\zz})\right)$, which together with \cite[Theorem~1]{GHVHM} and \cite[Corollary~4.3 and Theorem~4.9]{hkm09} implies the following two corollaries.

\begin{ourcorollary} If $(M,\xi,\fol)$ is overtwisted or has positive $2\pi$-torsion, then the class $c_A(M,\xi, \fol)$ is zero in  $H_\ast\left(\bsahat(-M,\bsGamma,\overline{\zz})\right)$.
\end{ourcorollary}

\begin{ourcorollary} If $c_A(M,\xi, \fol)$ is zero in  $H_\ast\left(\bsahat(-M,\bsGamma,\overline{\zz})\right)$, then $(M,\xi,\fol)$ does not embed into any closed contact manifold $(N,\xi')$ with nonvanishing contact invariant. 
\end{ourcorollary}

\subsection{Further directions}

The results in this paper establish a framework for studying foliated open books via Heegaard Floer homology in concert with other combinatorial representations of contact manifolds.  This is an essential first step towards developing cut-and-paste technology for the Heegaard Floer contact invariant, and we briefly note further avenues for developing this theory.  

The defining data of a contact manifold with foliated boundary includes a choice of  a distinguished leaf in the foliation, which plays an essential role in constructing the associated Heegaard diagram.  It is natural to ask how the invariant depends on this choice; accordingly, this dependence is the subject of planned future work.   We anticipate that the  foliation on $\partial M$ may be reparameterized by the addition of a suitable foliated open book for $\partial M\times I$.  This is a special case of the more general process of  gluing a boundary-parallel layer onto $\partial M$. We  hope to understand the maps induced by such gluings in general, and to  compare our gluing operation to the sutured Floer homology gluing map from  \cite{hkm09}.

Theorems~\ref{thm:ca-cd} and \ref{thm:glue} are phrased terms of equivalences of elements, rather than elements. This subtlety arises because as of this writing there is not a naturality result for the bordered variants of Heegaard Floer homology; hence, the type $D$ bordered sutured invariant associated to $(M,\bsGamma,\zz)$ is the type $D$ homotopy equivalence class of $\bsdhat(M,\bsGamma,\zz)$, as defined in Section \ref{sssec:bs}. An analogue of Juh\'asz, Thurston, and Zemke's proof of the naturality of the Heegaard Floer homology of closed three-manifolds \cite{JTZNaturality} for bordered sutured Floer homology would immediately upgrade our invariants.

\subsection*{Organization}
Section~\ref{sec:prelim} reviews some necessary background in contact geometry and Heegaard Floer homology. Section~\ref{ssec:contact} discusses assumed background in contact geometry while Section~\ref{ssec:hf-background} contains a rapid review of Heegaard Floer homology and bordered sutured Heegaard Floer homology; in particular, Section~\ref{sssec:bs} introduces a notion of equivalence between elements in bordered sutured Heegaard Floer homology under homotopy equivalences of type $A$ and type $D$ structures. Finally Section~\ref{ssec:afob} summarizes the relevant material from \cite{LV} concerning foliated open books. In Section~\ref{sec:construction} we associate a bordered sutured manifold $(M, \bsGamma, \mathcal{Z})$ to a foliated  contact three-manifold.  We then show how the data of a sorted foliated open book gives rise to an admissible bordered Heegaard diagram $\mathcal{H}$ for $(-M, \bsGamma, \overline{\mathcal{Z}})$ and we identify preferred generators $\mathbf{x}_A$ and $\mathbf{x}_D$ in the associated bordered Floer homology modules $\bsahat(\mathcal{H})$ and $ \bsdhat(\mathcal{H})$.  Section~\ref{sec:invariance} proves Theorem~\ref{thm:ca-cd}, namely invariance of $\x_A$ and $\x_D$ up to the choices made in their definitions.  Section~\ref{sec:gluing} proves Theorem~\ref{thm:glue}, showing that we recover the ordinary contact invariants after gluing. Finally, Section~\ref{sec:relation} discusses the relationship of our invariants to the invariant in sutured Floer homology, proving Theorem~\ref{thm:hkm}.

\subsection*{Acknowledgements} 
This material is partially based upon work supported by the National Science Foundation under Grant  DMS-1439786 while the authors were in residence at the Institute for Computational and Experimental Research in Mathematics in Providence, RI, during the Women in Symplectic and Contact Geometry and Topology workshop. 

JEL would like to express her appreciation to the Australian Mathematical Society for their support in the form of an Anne Penfold Street Award. During the first six months of working on this paper VV had a position at the University of Strasbourg with the CNRS; VV is grateful to both for providing a calm yet active research environment.  

We thank Robert Lipshitz and Paolo Ghiggini for helpful conversations. We are also grateful to the anonymous referees for their careful reading and thoughtful comments.


\section{Preliminaries}\label{sec:prelim}
This section provides the background required to read the rest of the paper.  We provide references for various classical objects in contact geometry in Section~\ref{ssec:contact} and Heegaard Floer theory in Section \ref{ssec:hf-background}, along with more in-depth summaries of partial open books and the contact invariants in various flavors of Floer homology.  Because we are concerned with the relationships between various theories, we pay particular attention to the conventions for bordered, sutured, and bordered sutured  versions.  Finally, Section~\ref{ssec:afob} gives an efficient introduction to foliated open books.

\subsection{Assumed background in contact geometry} \label{ssec:contact} Throughout this article, we assume familiarity with many standard definitions in three-dimensional contact geometry, including contact structures; characteristic foliations;  convex surfaces and dividing sets; open book decompositions for closed three-manifolds, as in e.g., \cite{Geiges} and other standard references.  Although we will introduce foliated open books carefully in Section~\ref{ssec:afob} below, we briefly first recall the definition of a partial open book  supporting a contact structure on a manifold with boundary.
\\

\begin{definition}\label{def:pob}
	A \emph{partial open book} is a triple $(S, P, h)$ where 
	\begin{enumerate}
		\item $S$ is a compact, oriented, connected surface with boundary;
		\item $P=\cup P_i$ is a subsurface of $S$ such that the surface $S$ is obtained from $\overline{S\setminus P}$ by successively attaching 1-handles $P_i$; and		
		\item $h \co P\rightarrow S$ is an embedding which is the identity along $\partial P\cap \partial S$.
	\end{enumerate}
\end{definition}

To a partial open book, we can associate a sutured manifold $(M, \Gamma)$, as follows. (See \cite{hkm09} or \cite{EO11} for more details.) Let $H=S\times [-1,0]$ with the identification $(x,t)\sim (x,t')$ for $x\in \partial S$ and $t, t'\in [-1,0]$. Similarly, let $N=P\times [0,1]$ with the identification $(x,t)\sim (x,t')$ for $x\in \partial P\cap \partial S$ and $t, t'\in [0,1]$. Then $M=H\cup N$ where we identify $P\times \{0\}\subset \partial H$ with $P\times \{0\}\subset \partial N$ and $h(P)\times \{-1\}\subset \partial H$ with $P\times \{1\}\subset \partial N$. 
	The suture $\Gamma$ on $\partial M$ can be given as the union of oriented closed curves obtained by gluing the following arcs, modulo identifications:
	$$\Gamma = \overline{\partial S \setminus \partial P}\times \{0\}\cup - \overline{\partial P\setminus \partial S}\times \{1/2\}.$$

\begin{definition}
A contact structure $\xi$ is \emph{compatible} with the partial open book $(S,P,h)$ if for the corresponding sutured manifold $(M=H\cup N, \Gamma)$, the following hold: 

	\begin{enumerate}
		\item $\xi$ is tight on $H$ and $N$;
		\item $\partial H$ is a convex surface in $(M, \xi)$ with dividing set $\partial S\times \{0\}$; and
		\item $\partial N$ is a convex surface in $(M, \xi)$ with dividing set $\partial P\times \{1/2\}$.
	\end{enumerate}
\end{definition}

\subsection{Heegaard Floer theories} \label{ssec:hf-background} We assume familiarity with the various Heegaard Floer theories and provide only a brief review to establish notation and review details related to the contact invariants defined in \cite{oszc, hkm09, HKM09_HF}. 

\subsubsection{Earlier Heegaard Floer theoretic contact invariants}\label{sssec:hfc}

Using open books, Ozsv\'ath and Szab\'o defined a Heegaard Floer  invariant of a closed contact three-manifold \cite{oszc}. For a contact manifold $(M,\xi)$, this invariant is a class $c(\xi)$ in the Heegaard Floer homology $\hfhat(-M)$. 
In \cite{HKM09_HF}, Honda, Kazez, and Mati\'c gave an alternative description of $c(\xi)$. Their construction again uses open books and goes roughly as follows. An open book $(S,h)$ for $(M,\xi)$ induces a Heegaard splitting of $M$ into $U_1\coloneqq S\times \left[0, 1/2\right]$ and $U_2\coloneqq S\times \left[1/2, 1\right]$ with Heegaard surface $\Sigma=(S\times \{1/2\})\cup_B-(S\times\{0\})$. Let $\{a_i\}$ be a collection of properly embedded arcs on $S$ that cut $S$ into a disk.  For all $i$, let $b_i$ be a small perturbation of $a_i$ that moves the endpoints in the positive direction along $\partial S$, so that  $b_i$ intersects $a_i$ in one point.  Fix a basepoint $z$ on $\bdy S$ away from the thin strips cobounded by the $\{a_i\}$ and $\{b_i\}$. It is clear that $\Sigma = \bdy U_1 = -\bdy U_2$,  and that the $\alpha_i := \bdy(a_i\times \left[0,1/2\right])$ bound compressing disks for $U_1$ and the $\beta_i := \bdy(b_i\times \left[1/2,1\right])$ bound compressing disks for $U_2$. Thus, $(\Sigma, \alphas, \betas, z)$ is a Heegaard diagram for $M$. One can see the curves as 
  \begin{align*}
 \alpha_i &= a_i\cup -a_i \subset (S\times \{1/2\})\cup_B-(S\times\{0\}),\\
\beta_i &= b_i\cup -h(b_i) \subset (S\times \{1/2\})\cup_B-(S\times\{0\}).
 \end{align*}
Then 
$c(\xi)\in \hfhat(-M)$
 is defined as the homotopy equivalence class of the unique generator of $\cfhat(\Sigma, \betas, \alphas, z)$ fully supported on $S\times \{1/2\}$.

 Using partial open books, Honda, Kazez, and Mati\'c then extended the above construction to define an invariant $EH(M, \Gamma, \xi)$ of contact three-manifolds with convex boundary. In this construction, one begins with a partial open book $(S,P, h)$ for $(M, \Gamma, \xi)$, where $S$ is built up from $\overline{S\setminus P}$ by the addition of 1-handles $P_i$, as in Definition~\ref{def:pob}. The roles of $U_1$ and $U_2$ are now played by $\big(P\times[-1/2,0]\big)\cup \big(S\times[0,1/2]\big)/\sim$ and $\big(S\times[1/2,1]\big)\cup \big(P\times [-1,-1/2]\big)/\sim$, respectively, and the curves $\{a_i\}$ are the co-cores of the 1-handles $P_i$. The rest of the construction is as above, except that the final generator supported on $P\times \{-1/2\}$ defines an element in the sutured Floer homology $\sfh(-M, -\Gamma)$. Under certain technical conditions, the authors also defined a gluing map for sutured Floer homology that respects the contact invariants  \cite{hkm08}.

Before briefly reviewing bordered sutured Floer homology, we discuss a straightforward generalization of the contact invariant to multipointed Heegaard diagrams. We begin by constructing a Heegaard diagram analogous to the one above, except that we allow additional arcs $\{a_i\}$ that cut $S$ into $n$ disks. We place a basepoint in each of the disks thus obtained. This results in a multipointed Heegaard diagram $(\Sigma, \alphas, \betas, \mathbf{z})$ for $M$. The unique generator $\mathbf{x}$ on $S\times \{1/2\}$ defines an element in $\HFt(\Sigma, \betas,  \alphas, \mathbf{z}) \cong \hfhat(-M)\otimes H_*(T^{n-1})$. Here $\HFt$ is the homology with respect to the differential that avoids all basepoints, and  $H_*(T^{n-1}) \cong H_*(S^1)^{\otimes(n-1)}$ is the ordinary singular homology of $T^{n-1}$.

An adaptation of the argument of  part (5) of \cite[Theorem 3.1]{BVV} shows the following:
\begin{proposition}\label{lem:multicontact}
Let $(S,h)$ and $(S',h')$ be two open book decompositions compatible with $(M,\xi)$. Let   $\{a_i\}$ and $\{a_i'\}$ be sets of cutting arcs that cut up $S$ and $S'$ into $n$ and $n'$ disks, respectively, with $n<n'$. The graded isomorphism between the Heegaard Floer homologies induced by Heegaard moves, including index $0$ and $3$ stabilizations,  \[
\HFt(\Sigma', \betas',  \alphas', \mathbf{z}') \otimes H_*(T^{n'-n}) \to
\HFt(\Sigma, \betas,  \alphas, \mathbf{z}) \]
 maps the homology class to $[\mathbf{x}'] \otimes \theta^{\otimes{(n'-n)}}$ to $[\mathbf{x}]$, where  $\theta$ corresponds to the lower-degree generator of  $H_*(S^1)$. \qed
\end{proposition}
This means that up to homotopy equivalence, the multipointed contact invariant is simply $c(\xi)\otimes \theta^{\otimes{(n-1)}}$, where $n$ is the number of basepoints.

We will also make use of the fact that the multipointed Heegaard Floer homology for a closed three-manifold $M$ can be interpreted as the sutured Floer homology of $M$ with balls removed, as follows.
Let $(\Sigma, \alphas, \betas, \mathbf{z})$ be a multipointed Heegaard diagram for  $M$ with $n$ basepoints, and for each basepoint $z\in \mathbf{z}$, let $D^2_z$ be a disk neighborhood of $z$. 
Then $(\Sigma\setminus \cup_\mathbf{z} D^2_z, \alphas, \betas)$ is a sutured Heegaard diagram for  $M(n)  = (M\setminus \cup_{z\in \mathbf{z}}B^3_{z}, \cup_{z\in \mathbf{z}}\partial D^2_z)$. As the two Heegaard diagrams are identical outside the basepointed/sutured regions, the chain complex  $\CFt(\Sigma, \alphas, \betas, \mathbf{z})$ is isomorphic to the chain complex $\SFC(\Sigma\setminus \cup_\mathbf{z} D^2_z, \alphas, \betas)$. Thus, we can compute the multipointed Heegaard Floer homology $\HFt(\Sigma, \betas,  \alphas, \mathbf{z})$ as the sutured Floer homology $\sfh(\Sigma\setminus \cup_\mathbf{z} D^2_z, \alphas, \betas)$. See also \cite[Proposition 9.14]{sfh}.

\subsubsection{Bordered sutured Floer homology}\label{sssec:bs}

Lipshitz, Oszv{\'a}th, and Thurston refine Heegaard Floer homology to a bordered variant associated to a three-manifold with parametrized boundary \cite{LOTbook, hfmor}, and Zarev \cite{bs, bs-JG} further refines the invariant to an invariant of sutured manifolds with partially parameterized boundary. We briefly discuss Zarev's constructions \cite[Section 3]{bs-JG}. Note that the following description is indicative, rather than complete; for a review of the algebraic definitions involved, see \cite[Section 2]{LOTbook}.

Recall that an \emph{arc diagram} $\mathcal Z = (Z,a,m)$ consists of a finite collection of oriented line segments $Z$, commonly called arcs in the literature, a collection of points $a=(a_1,\cdots, a_{2k})$ on $Z$, a matching $m$ of the points in $a$ into pairs, and a ``type,'' $\alpha$ or $\beta$. To every arc diagram $\zz$ one associates an $\mathcal A_{\infty}$-algebra $\mathcal A(\zz)$ generated by tuples of  oriented arcs in $[0,1]\times Z$  such that each arc connects some $(0,a_i)$ to some $(1, a_j)$ with $a_j\geq a_i$, up to an equivalence relation imposed by the matching. The ground ring of idempotents $\mathcal I(\zz)$ of this algebra consists of elements $\iota$ corresponding to tuples of horizontal strands $[0,1]\times \{a_k\}$ in $[0,1]\times Z$. One further associates to $\zz$ a graph $G(\zz)$ and a surface $F(\zz)$.  The graph is constructed by beginning with $Z$  and adding an edge between each pair of matched points, while  $F(\zz)$ is constructed from $Z \times \left[0,1\right]$ by attaching one-handles to neighborhoods of the matched points on $Z\times \{0\}$. 

A bordered sutured manifold $(M,\bsGamma,\zz)$    consists of the following:
\begin{itemize}
\item a three-manifold $M$ whose boundary decomposes as a \emph{bordered} part $F$ and a \emph{sutured} part $T$;
\item an arc diagram $\mathcal{Z}=(Z,a,m)$ and an identification of $F$ with $F(\zz)$;  
\item a \emph{dividing set} $\bsGamma$ which is a properly embedded, oriented one-manifold in $T$ with boundary $\bdy \bsGamma = -\bdy(Z\times \{\frac 1 2 \})$, so that $\bsGamma$ decomposes $T$ into  the union of $R_+(\bsGamma)$ and  $R_-(\bsGamma)$ with $\partial R_{\pm}(\bsGamma)\setminus\bdy F=\pm \bsGamma$. 
\end{itemize}
The components of the dividing set are also referred to as \emph{sutures}. 

A $\beta$-bordered Heegaard diagram $\HH=(\Sigma, \alphas, \betas, \zz)$ consists of a compact surface $\Sigma$ with no closed components; a collection of pairwise disjoint, properly embedded circles $\alphas$; a collection of pairwise disjoint, properly embedded circles $\betas^{c}$ and of properly embedded arcs $\betas^{a}$, with $\betas = \betas^a \cup \betas^c$; and an arc diagram $\zz$ of $\beta$-type, together with an embedding $G(\zz)\rightarrow \Sigma$ which maps $Z$ to a subset of $\partial \Sigma$ and the edges of $G(\zz)$ connecting matched points to the arcs $\betas^c$. One also requires that the maps $\pi_0(\bdy\Sigma \setminus Z)\to \pi_0(\Sigma\setminus \alphas)$ and $\pi_0(\bdy\Sigma \setminus Z)\to \pi_0(\Sigma\setminus \betas)$ be surjective. (An $\alpha$-bordered diagram is defined similarly, mutatis mutandis.) One constructs a bordered sutured manifold  $(M,\bsGamma,\zz)$ from $\HH$ as follows. The three-manifold $M$ is obtained by attaching two-handles to $\Sigma\times[0,1]$ along $\alphas \times \{1\}$ and $\betas^c\times \{0\}$ circles; the dividing set $\bsGamma$ appears as $\bsGamma = (\partial \Sigma \setminus Z) \times \left\{\frac{1}{2}\right\}$, and $F(\zz)$ is a neighborhood of $(Z \times [0,1])\cup (\betas^a\times \{0\})$.  Another way to view $(M,\bsGamma,\zz)$ as coming from $\HH$, which fits better with Morse theory, is to also attach ``halves of two-handles" along $\betas^a \times \{0\}$ and see $F(\zz)$ as $Z\times [0,1]$ together with the intersection of the thickened cores of the partial two-handles with $\bdy M$.

Let ${\bf x}$ and ${\bf y}$ denote tuples of intersection points between the $\alphas$ and $\betas$ such that each $\alpha$-circle is used exactly once, each $\beta$-circle is used exactly once, and each $\beta$-arc is used no more than once; these will ultimately be the generators of the bordered sutured modules. We recall that the set of homology classes $\pi_2({\bf x}, {\bf y})$ connecting ${\bf x}$ to ${\bf y}$ is defined as follows. We let $I_s =[0,1]$ and $I_t = [-\infty, \infty]$ be intervals, and consider the relative homology group
\[H_2(\Sigma \times I_s \times I_t, ((\alphas \times \{1\}) \cup (\betas \times \{0\}) \cup (Z \times I_s))\times I_t), ({\bf x}\times I_s\times \{-\infty\}) \cup ({\bf y} \times I_s \times \{\infty\})).\]
The set $\pi_2({\bf x}, {\bf y})$ denotes elements of this group which are sent to the fundamental class of $({\bf x}\times I_s)\cup ({\bf y}\times I_s)$ by the map which applies the boundary homomorphism and then collapses the remainder of the boundary. Any homology class has a unique corresponding \emph{domain}, that is, a linear combination of the components of $\Sigma \setminus (\alphas \cup \betas)$ obtained by projection. A domain is \emph{provincial} if its boundary has no intersection with $Z$. 

The set $\pi_2({\bf x}, {\bf x})$ is the set of $\emph{periodic domains}$. We recall that $\HH$ is said to be \emph{provincially admissible} if every provincial periodic domain has both positive and negative coefficients and \emph{admissible} if every periodic domain has both positive and negative coefficients.

To a provincially admissible $\beta$-bordered Heegaard diagram $\HH=(\Sigma, \alphas, \betas, \zz)$ equipped with an admissible almost complex structure, Zarev associates a left type $D$ module $\bsdhat(\HH)$ over $\mathcal A(-\zz)$ and right type $A$ module $\bsahat(\HH)$ over $\mathcal A(\zz)$. In both cases the generators are tuples of intersection points as described above.\footnote{We assume we are in the nondegenerate case of having a nonempty set of generators. In particular, $|\alphas|\geq |\betas^c|$.}  The type $A$ chain homotopy equivalence class of $\bsahat(\HH)$ is an invariant of the associated bordered sutured manifold $(M, \bsGamma, \zz)$, written $\bsahat(M, \bsGamma, \zz)$. More precisely, given bordered sutured Heegaard diagrams $\HH$ and $\HH'$ for $(M, \bsGamma, \zz)$, there is a sequence of Heegaard moves connecting them as in \cite[Proposition 4.5]{bs}, which induce a type $A$ homotopy equivalence $f$ from $\bsahat(\HH)$ to $\bsahat(\HH')$. This type $A$ homotopy equivalence is a collection of maps
\[f_i:\bsahat(\HH)\otimes \mathcal A(\zz)^{\otimes(i-1)}\to \bsahat(\HH')\]
indexed by $i\ge 1$, satisfying certain conditions (see \cite[Section 2]{LOTbook}). Consider $\x\in\bsahat(\HH)$ and $\x'\in\bsahat(\HH')$ so that 
\[m_{i+1}(\x,a_1,\cdots,a_i)=0\quad\quad\text{and}\quad\quad m_{i+1}(\x',a_1,\cdots,a_i)=0\]
for all $i\ge 0$ and $a_i\in\mathcal{A}(\zz)$. We say $\x$ and $\x'$ are \emph{equivalent} if there exists a homotopy equivalence $f$ from $\bsahat(\HH)$ to $\bsahat(\HH')$ so that $f_1(\x)=\x'$ and $f_{i+1}(\x,a_1,\dots,a_i)=0$ for all $i>0$ and $a_i\in\mathcal{A}(\zz)$. Similarly, there is a type $D$ homotopy equivalence from $\bsdhat(\HH)$ to $\bsdhat(\HH')$. If $\x\in \bsdhat(\HH)$ and $\x'\in\bsdhat(\HH')$ are both closed elements (meaning that $\delta^1(\x)=\delta^1(\x')=0$), we say $\x$ and $\x'$ are \emph{equivalent} if there is a homotopy equivalence from  $\bsdhat(\HH)$ to $\bsdhat(\HH')$ that maps $\x$ to $\x'$.

Given two bordered sutured manifolds $(M_1, \bsGamma_1, \zz)$ and $(M_2, \bsGamma_2, -\zz)$, we may glue along $\zz$ to obtain a sutured manifold $(M, \Gamma)$. Given $\beta$-bordered sutured Heegaard diagrams $\HH_1 = (\Sigma_1, \alphas_1, \betas_1, \zz)$ for $(M_1,  \bsGamma_1, \zz)$ and $\HH_2 = (\Sigma_2, \alphas_2, \betas_2, -\zz)$ for $(M_1,  \bsGamma_1, -\zz)$, there is a glued Heegaard diagram $\HH = (\Sigma, \alphas, \betas)$ where $\Sigma = \Sigma_1 \bigcup_{Z} \Sigma_2$; the set $\alphas$ is the union of $\alphas_1$ and $\alphas_2$; and the set $\betas$ is the union of $\betas_1$ and $\betas_2$, together with the circles formed by gluing the arcs in $\betas_1^a$ and $\betas_2^a$ along their endpoints on $Z$.

When at least one of the two diagrams is admissible, there is a gluing map
\[
\bsahat(\HH_1) \boxtimes_{\mathcal{A}(\zz)} \bsdhat(\HH_2) \rightarrow \sfc(\HH)
\]
which induces a chain homotopy equivalence of vector spaces. Chain homotopy equivalences of type $A$ and type $D$ modules induce chain homotopy equivalences of the box tensor product, so there is a well-defined equivalence class of $\F_2$ vector spaces $\bsahat(M_1, \bsGamma_1, \zz) \boxtimes_{\mathcal{A}(\zz)} \bsdhat(M_2, \bsGamma_2, \ip{-}\zz)$, and an equivalence of chain homotopy equivalence classes of vector spaces
\[
\bsahat(M_1, \bsGamma_1, \zz) \boxtimes_{\mathcal{A}(\zz)} \bsdhat(M_2, \bsGamma_2, \ip{-}\zz) \rightarrow \sfc(M, \Gamma).
\]

\subsection{Abstract foliated open books}\label{ssec:afob}

In this section, we provide an  overview of the essential definitions and properties of foliated open books that we will rely on in the remainder of this article.   Readers are encouraged to see \cite{LV} for more detail.

\begin{definition}\cite[Definition 3.14]{LV}\label{def:afob} An \emph{abstract foliated open book} is a tuple $(\{S_i\}_{i=0}^{2k}, h)$ where $S_i$ is a surface with 
boundary $\partial S_i=B\cup A_i$ \footnote{By a slight abuse of notation we denote the ``constant" part of the boundary of $S_i$ by $B$ for all $i$.} and corners at $E=B\cap A_i$ such that 
\begin{enumerate}
\item for all $i$, $A_i$ is a union of intervals;
\item with the boundary orientation, each component $I$ of $A_i$ is oriented from a corner labeled $e_+ = e_+(I)\in E_+$ to a corner labeled $e_- = e_-(I)\in E_-$, and $E=E_+\cup E_-$;

\item the surface $S_{i}$ is obtained from $S_{i-1}$ by either 
\begin{itemize}
\item[-]\textbf{(add):} attaching a 1-handle along two points $\{p_{i-1}, q_{i-1}\}\in A_{i-1}$, or 
\item[-]\textbf{(cut):} cutting $S_{i-1}$ along a properly embedded arc $\gamma_{{i}}$ with endpoints in $A_{{i-1}}$ and then smoothing.\footnote{The indices of $\gamma_i$ in this paper are shifted compared to \cite{LV}, where the cutting arcs were denoted by $\gamma_{i-1}$.}
\end{itemize}
\end{enumerate}

Furthermore,  $h \colon S_{2k}\rightarrow S_0$ is a diffeomorphism between cornered surfaces that
preserves $B$ pointwise. 
\end{definition}
We denote by $H_+$ (resp.~$H_-$) the set of indices $i$ for which $S_{i}$ is obtained from $S_{i-1}$ by cutting  (resp.~adding), so that we have a partition $[2k]=H_+\cup H_-$. (Here $[n]$ denotes the set $\{1,\dots, n\}$.)

Note that the operations \textbf{(add)} and \textbf{(cut)} are  inverses  of each other.  Let $\gamma$ denote the cocore of a handle attached along points $p$ and $q$.  Then cutting along $\gamma$ cancels the handle attachment, and vice versa. We will use the following notation to describe this:
\[S\xrightarrow[{p,q}]{\text{\textbf{add}}} S' \text{    and    } S'\xrightarrow[{\gamma}]{\text{\textbf{cut}}}S.\] 

\begin{example}\label{ex:torus}
Figure~\ref{fig:afobtorus} shows a first example of a foliated open book.  The complete set of labels is shown for the indicative page $S_0$, while the attaching spheres and cutting arcs are shown on all pages to the right.  The binding is shown in bold. The monodromy $h \colon S_4\rightarrow S_0$ is the identity, and we have the partition $H_-=\{2,4\}$, $H_+=\{1, 3\}$. 
\begin{figure}[h]
	\labellist
	\pinlabel {$S_0$} [ ] at 23 -6
	\pinlabel {$S_0$} [ ] at 140 -6
	\pinlabel {$S_1$} [ ] at 178 -6
	\pinlabel {$S_2$} [ ] at 216 -6
	\pinlabel {$S_3$} [ ] at 258 -6
	\pinlabel {$S_4$} [ ] at 297 -6
	\pinlabel {$A_0$} [ ] at 5 74
	\pinlabel {$\gamma_1$} [ ] at 11 35
	\pinlabel {$A_0$} [ ] at 30 48
	\pinlabel {$e_+$} [l] at 37 80
	\pinlabel {$e_-$} [l] at 37 58
	\pinlabel {$e_+$} [l] at 37 25
	\pinlabel {$e_-$} [l] at 37 2
	\pinlabel {$B$} [l] at 37 70
	\pinlabel {$B$} [l] at 37 14
	\endlabellist
\begin{center}
\includegraphics[scale=1]{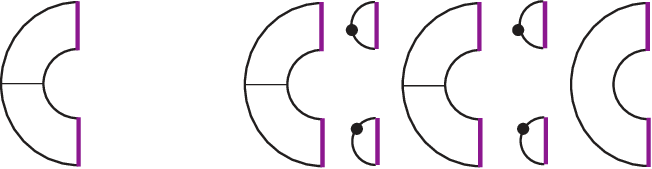}
\end{center}
\caption{A foliated open book with $k=2$.}\label{fig:afobtorus}
\end{figure}
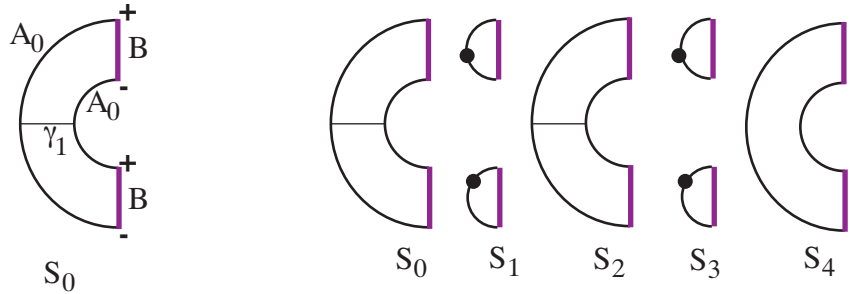
\end{example}

\subsubsection{Supported contact structures}
\label{sec:supp-cs}

The construction of a compact manifold is natural from the data of a foliated open book: pairs of successive pages define cornered cobordisms.  As described in more detail below, we concatenate these to form a manifold with boundary and collapse the resulting components of $B\times I$ to circles and intervals again labeled $B$.  The final page glues to the initial page via $h$. The resulting  manifold $M$ associated to the foliated open book $(\{S_i\}_{i=0}^{2k}, h)$ retains partial information about the sequence of cuts and adds in the abstract data. This information is encoded in the form of boundary decorations on $\partial M$, and we introduce the kind of foliation we will consider before explaining how it arises:

\begin{definition}  A \emph{signed singular foliation} is an equivalence class of smooth vector fields (up to multiplication by smooth positive functions) that vanish at only finitely many isolated points, called \emph{singular points}. The  complement of the singular points has an open cover such that in each ball, the integral curves of the vector field are a product of oriented intervals, while elliptic and four-pronged hyperbolic singularities patch these charts together. Elliptic singularities may be classified as positive (sources) or negative (sinks), and the hyperbolic singularities have signs determined by input external to the defining vector field.  
\end{definition}
 
 See Figure \ref{fig:sing} for local models of the integral curves.

As is  standard, leaves that enter (or exit) the hyperbolic singularities are called \emph{stable (or unstable) separatrices}.   

\begin{figure}[h]
\begin{center}
\includegraphics[scale=.8]{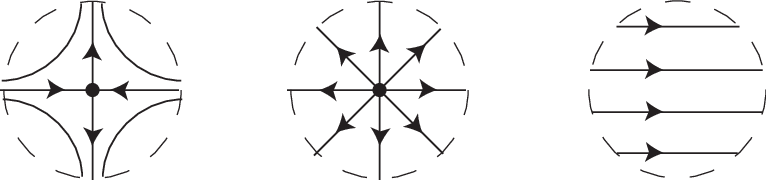}
\caption{Left: a hyperbolic singularity.  Center: a positive elliptic singularity. Right: a regular foliation on an open set.}\label{fig:sing}
\end{center}
\end{figure}

Each elliptic point $e$ induces a cyclic order on the subset of hyperbolic points with separatrices terminating at $e$.  If $e$ is positive (respectively, negative), the order increases as the separatrices are encountered in a counterclockwise (clockwise) path around $e$.  
\begin{definition}
A signed singular foliation is \emph{ordered} if there is a cyclic order on the set of all hyperbolic points which is compatible with the cyclic orders associated to each of its elliptic points.
\end{definition}
  
Beginning with an abstract foliated open book, we now build a manifold whose boundary is naturally equipped with an ordered signed singular foliation. Each pair of successive pages defines an elementary cobordism $M_i$
from $S_{i-1}$ to $S_{i}$ with vertical boundary $(B\times I)\cup V_i$, where $V_i$ is the union of a single saddle and of the collection of products $A_{i-1}\times I$ for any components $A_{i-1}$ that are left unchanged.  More specifically, the saddle connects either the components of $A_{i-1}$ containing the endpoints of $\gamma_{i}$  or the component(s) containing $p_{i-1}$ and $q_{i-1}$ with the  obtained component(s) of $A_i$ that have the same endpoints. See Figure~\ref{fig:cobord}. 

\begin{figure}[h]
	\labellist
	\pinlabel {$B$} [ ] at 33 122
	\pinlabel {$B$} [ ] at 33 23
	\pinlabel {$B$} [ ] at 285 23
	\pinlabel {$B$} [ ] at 285 122
	\pinlabel {$A_i$} [ ] at 110 147
	\pinlabel {$A_i$} [ ] at 203 147
	\pinlabel {$A_{i-1}$} [ ] at 127 -3
	\pinlabel {{\color{lightgray}$A_{i-1}$}} [ ] at 96 49
	\pinlabel {$\gamma_i$} [ ] at 151 20
	\pinlabel {$S_{i-1}$} [ ] at 228 24
	\pinlabel {$S_i$} [ ] at 227 121
	\endlabellist
\begin{center}
\includegraphics[scale=.6]{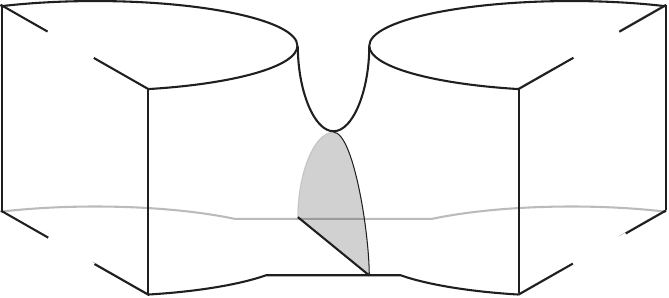}
\caption{The elementary cobordism associated to cutting $S_{i-1}$ along $\gamma_i$, shown before collapsing $B\times I$.}\label{fig:cobord}
\end{center}
\end{figure}
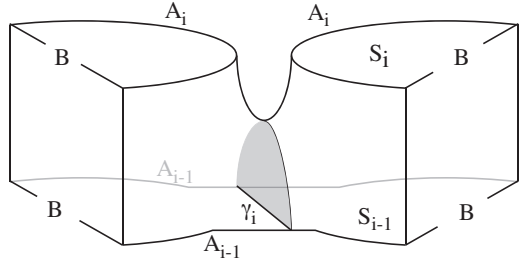

Concatenating these $M_i$ and gluing the final $S_{2k}$ to the initial $S_0$ via $h$ yields a manifold with boundary $(B\times S^1) \cup V$, where $V$ is the (circular) union of the $V_i$.  The singular foliation on $V$ has $2k$ singular points  associated to the transitions between topological types of the pages, while the regular leaves may be identified with curves of the form $A_i\times \{t\}$. 
If $S_{i-1}\xrightarrow{\text{\textbf{add}}} S_{i}$, then the corresponding hyperbolic point $h_i$ is negative; otherwise, it is positive. We denote the set of hyperbolic points with respect to this 
partition $H=H_+\cup H_-$. The signs match the partition of $[2k]$ introduced after Definition~\ref{def:afob}.
 Collapsing the components $B\times S^1$ to a single copy of $B$ yields a manifold with foliated boundary $(M, \partial M, \mathcal{F})$. (See Definition~\ref{def:mfb}, below, for a precise definition of this term.) Each endpoint  of $B$ labeled $e_+$ becomes a positive elliptic point of the foliation, which we again denote by $e_+$; likewise, an endpoint $e_-$ becomes a negative elliptic point $e_-$.  
 
 \begin{example} We return to the foliated open book introduced in the first example. The associated smooth manifold is the solid torus shown in Figure~\ref{fig:afobtorusmfd}. The images of $A_1$ and $A_2$ are shown in green and blue, respectively, on both the manifold on the left and in the associated boundary foliation on the right.

\begin{figure}[h]
	\labellist
	\pinlabel {$e_+$} [bl] at 37 144
	\pinlabel {$e_-$} [l] at 37 99
	\pinlabel {$e_+$} [l] at 37 68
	\pinlabel {$e_-$} [tl] at 37 22
	\pinlabel {$e_+$} [ ] at 225.5 71.5
	\pinlabel {$e_+$} [B] at 140 154
	\pinlabel {$e_-$} [B] at 216 154
	\pinlabel {$e_+$} [B] at 294 154
	\pinlabel {$e_-$} [ ] at 294 80
	\pinlabel {$e_+$} [B] at 294 3
	\pinlabel {$e_-$} [B] at 216 3
	\pinlabel {$e_+$} [B] at 140 3
	\pinlabel {$e_-$} [] at 140 80
	\pinlabel {$h_+$} [B] at 184 33
	\pinlabel {$h_-$} [B] at 254 33
	\pinlabel {$h_-$} [B] at 184 104
	\pinlabel {$h_+$} [B] at 254 104
	\endlabellist
\begin{center}
\includegraphics[scale=1]{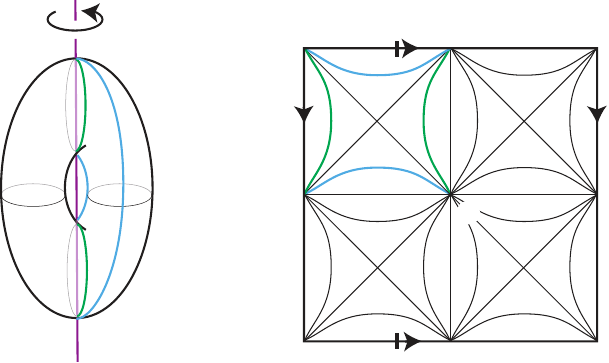}
\caption{The smooth manifold (a solid torus) associated to the  foliated open book from Figure~\ref{fig:afobtorus}  together with the foliation $\mathcal{F}$ on $\partial M$.}\label{fig:afobtorusmfd}
\end{center}
\end{figure}
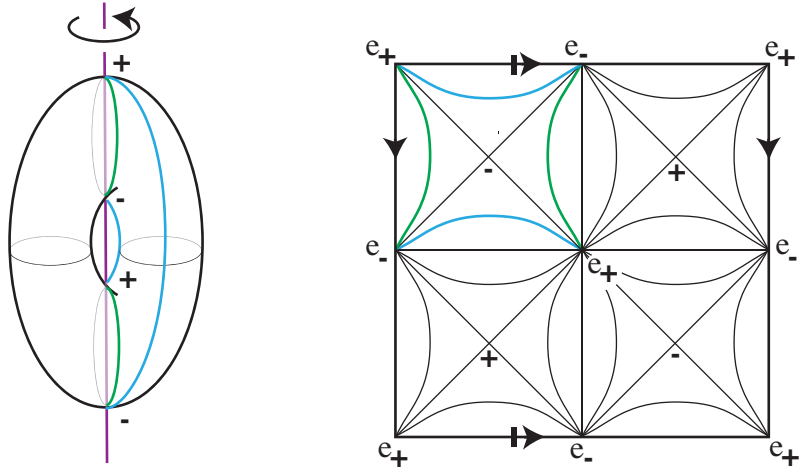
\end{example}

 In order to define the compatibility between contact structures and foliated open books, we  examine the boundary foliation $\mathcal{F}$ more closely.  To any smooth manifold constructed  as above, we may associate  an  $S^1$-valued Morse function $\pi \colon  M \setminus B \rightarrow S^1$  whose level sets are the pages of the open book and which restricts to the boundary as an $S^1$-valued Morse function $\pit:=\pi|_{\partial M}\colon \partial M\setminus E\to S^1$ with the same critical points as $\pi$. The function $\pi$ has only index 1 and 2 critical points, and all of these are index 1 critical points for $\pit$. The index with respect to $\pi$ is visible in the sign of these hyperbolic points: index 2 critical points of $\pi$ give  positive hyperbolic points of the foliation, while index 1 critical points of $\pi$ give negative hyperbolic points of the foliation. 
Heuristically, this means that the interesting features of $\pi$ may all be seen from the foliation on the boundary.
One may build a signed singular foliation simply by patching together the local models of Figure~\ref{fig:sing}, but in the context of this paper we will only encounter  ordered signed singular foliations;  equivalently, these are signed singular foliations that  can be induced by  an $S^1$-valued Morse function.  We therefore consider the map $\pit$ to be an essential part of the data and we write  $\fol=(\pit, H=H_-\cup H_+, E=E_-\cup E_+)$.  We may also assume that $0$ is a regular value of $\pit$ and we further require that $\pit=\pm\phi$ on an $(r, \phi)$-disk neighborhood of each elliptic point $e\in E$, where the sign depends on the sign of the elliptic point. Here $(r,\phi)$ is the standard polar coordinate system on $D^2$.   By construction, the foliation on a manifold constructed from an abstract foliated open book has only elliptic and four-pronged hyperbolic singularities, with  signs inherited from the  partitions of $E$ and $H$.

 \begin{definition}\label{def:mfb}
 A \emph{manifold with foliated boundary}  is a compact oriented three-manifold equipped with an ordered signed singular foliation $\fol=(\pit, H=H_-\cup H_+, E=E_-\cup E_+)$ on its boundary. We denote this collection of data by  $(M, \partial M, \mathcal{F})$.  \end{definition}
  
   We recall  that this data includes the $S^1$-valued Morse function $\pit$, and we    require that a diffeomorphism $f:(M_1, \partial M_1, \fol_1) \rightarrow (M_2, \partial M_2, \fol_2)$  between manifolds with foliated boundary must satisfy $\pit_2\circ f= \pit_1$.
   
 We have seen that an abstract foliated open book gives rise to a diffeomorphism class of manifolds with foliated boundary; we say a manifold with foliated boundary is \emph{compatible} with $(\{S_i\}, h)$ if it is diffeomorphic to an output of this construction. A fixed manifold with foliated boundary will admit many different functions $\pi$ extending the Morse function $\pit$ on the boundary, but compatibility identifies an equivalence class of such functions whose level sets are diffeomorphic to the pages $\{S_i\}$. 
 
  \begin{definition}\label{def:div} Suppose that $\fol=(\pit, H=H_-\cup H_+, E=E_-\cup E_+)$ is a signed singular foliation with no $S^1$ leaves.  We define $R_+(\fol)$ to be a closed neighborhood of the union of the stable separatrices of hyperbolic points in $H_+$, and let $R_-(\fol)=\overline{M\setminus R_+(\fol)}$. 
  The \emph{dividing curve} $\Gamma(\fol)$ of $\fol$ is  $\partial R_+(\fol)=-\partial R_-(\fol)$.
  \end{definition} 
  
Note that the serifed symbol $\Gamma$ used here is distinct from the sans-serif symbol $\bsGamma$ introduced in Section~\ref{sssec:bs} as part of the defining data of a bordered sutured manifold. The above construction of $\Gamma(\fol)$ follows the construction of Giroux \cite{Gi-convex} for dividing curves of characteristic foliations of convex surfaces and in our case too $\Gamma(\fol)$ indeed divides $\fol$: $\Gamma(\fol)$ is positively transverse to the leaves of $\fol$ and it separates $\partial M\setminus \Gamma(\fol)$  into two parts, one of which contains all the positive singular points, while the other part contains all the negative singular points. \footnote{Readers concerned that we have not addressed a third condition for dividing curves, regarding the divergence of the foliation, are referred to Lemma 2.6 of \cite{LV}.} These  properties in fact specify $\Gamma(\fol)$ up to isotopy transverse to $\fol$, and we will often work with its equivalence class, only specifying the representative when needed. The choice of the representative then changes $R_+(\fol)$ and $R_-(\fol)$ accordingly, while still preserving the condition that $\Gamma(\fol)=\partial R_+(\fol)=-\partial R_-(\fol)$.   Note that distinct foliations on $\partial M$ may induce isotopic dividing sets; see Figure~\ref{fig:folball}.

   \begin{definition}
   Two ordered signed singular foliations $\fol_1$ and $\fol_2$ on a surface $\Sigma$ are \emph{strongly topologically conjugate} if 
	there is an isotopy taking $\fol_1$ to $\fol_2$ which respects the cyclic orders on the hyperbolic points.
   \end{definition}
   
 The term ``topological'' refers to the fact that the above isotopy may not be smooth. In fact, in the next definition, the isotopy cannot be chosen to be smooth.

\begin{definition} \cite[Definition 3.8, 3.10]{LV} A contact structure $\xi_0$ on a manifold with foliated boundary  $(M, \partial M, \mathcal{F})$ is \emph{strictly supported} by  the abstract foliated open book $(\{S_i\}, h)$    if  there is a diffeomorphism $f\colon (M', \partial M', \mathcal{F'})\to (M, \partial M, \mathcal{F})$, where $(M', \partial M', \mathcal{F'})$ is a manifold constructed from $(\{S_i\}, h)$ as above, such that the pullback $f^{\ast}\xi$ is the kernel of some contact one-form $\alpha$ on $M'$ satisfying the following conditions:
 \begin{enumerate}
 \item $\alpha(TB)>0$;
 \item for all $t$, $d\alpha|_{\pi^{-1}(t)}$ is an area form; and
 \item for each  pair of consecutive hyperbolic points at times $t_1$ and $t_2$, there is a time $t_*$ such that $t_1<t_*<t_2$ with the property that the regular leaf $\pit^{-1}(t_*)$ of $\mathcal{F}$  is Legendrian.
 \end{enumerate}
 A contact structure $\xi_1$ is \emph{supported} by $(\{S_i\}, h)$ if there exists a path of contact structures $\xi_t$ such that  $\xi_0$  is strictly supported by $(\{S_i\}, h)$ and for all $t$, the characteristic foliation $\mathcal{F}_{\xi_t}$ is strongly topologically conjugate to the characteristic foliation $\mathcal{F}_{\xi_0}$.  
  \end{definition}

\begin{definition}\label{def:mxif} A \emph{foliated contact three-manifold} $(M, \xi, \fol)$ is a manifold with foliated boundary together with a contact structure $\xi$ on $M$ such that $\fol$ is strongly topologically conjugate to $\mathcal{F}_{\xi}$.  \end{definition}

\begin{theorem}\cite[Theorems 3.12, 7.1, 7.2]{LV} \label{thm:support}
Any foliated open book with circle-free boundary foliation supports a unique isotopy class of contact structures, and any foliated contact three-manifold $(M,\xi,\mathcal{F})$ admits a supporting foliated open book.
\end{theorem}

\begin{remark}\label{rmk:t} Because the  signed singular foliation $\fol$ data includes a function $\widetilde{\pi}$, a foliated contact three-manifold has a distinguished $t=0$ leaf, which we always assume to be regular. 
\end{remark}

Distinct foliations may  induce isotopic dividing sets, as illustrated in the following example. 

\begin{example} Let $(B^3, \xi_1, \fol_1)$ and $(B^3, \xi_2, \fol_2)$ be two foliated contact three-balls with $\xi_1$ and $\xi_2$ tight, and with the foliations $\fol_1$ and $\fol_2$ on the boundaries as in Figure~\ref{fig:folball}.  The two three-balls are equivalent, in the sense that there is a contact embedding of one into the other with the property that the image of the embedded boundary is convexly isotopic to the boundary of the target manifold. This notion is an equivalence relation on contact manifolds with boundary, yet the given balls are distinct as foliated contact manifolds.

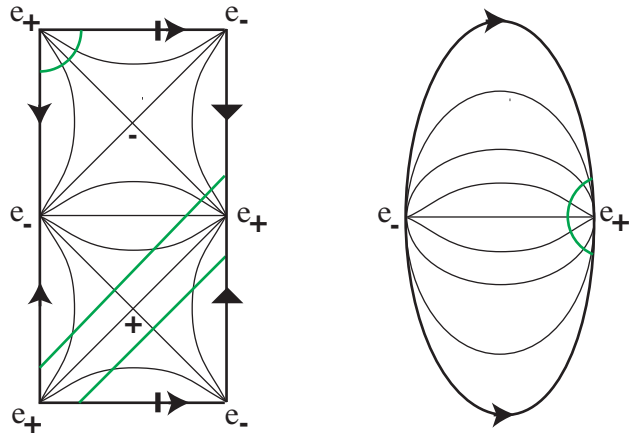
\begin{figure}[h]
	\labellist
	\pinlabel {{\tiny$h_+$}} [ ] at 41 38
	\pinlabel {{\tiny$h_-$}} [ ] at 41 108
	\pinlabel {$e_+$} [B] at -1 6
	\pinlabel {$e_-$} [B] at -1 80
	\pinlabel {$e_+$} [B] at -1 154
	\pinlabel {$e_-$} [B] at 83 154
	\pinlabel {$e_+$} [B] at 83 80
	\pinlabel {$e_-$} [B] at 83 6
	\pinlabel {$e_-$} [B] at 136 80
	\pinlabel {$e_+$} [B] at 222 80
	\endlabellist
\begin{center}
\includegraphics[scale=1]{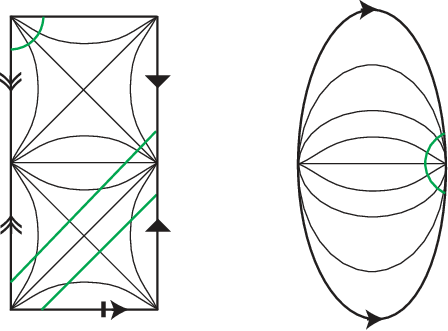}
\caption{Left: A foliation on $S^2$ with four elliptic points and two hyperbolic points. Right: A foliation on $S^2$ with two elliptic points. The dividing set for each foliation is a circle.}\label{fig:folball}
\end{center}
\end{figure}
\end{example}

\subsubsection{Stabilization}

Although each foliated open book supports a unique isotopy class of contact structures, there are many foliated open books which support the same class.  Next, we define an operation which changes a foliated open book while preserving the contactomorphism class of the associated contact manifold.  

\begin{definition}\label{def:stab} Let $\gamma\subset S_0$ be a properly embedded arc with $\partial \gamma\subset B$.  The \emph{stabilization of $(\{S_i\}, h)$ along $\gamma$} is the foliated open book $(\{S'_i\}, h')$ defined as follows:
\begin{itemize}
\item $S_i'=S_i\cup H$, where $H$ is a one-handle attached along $\partial \gamma$; and 
\item $h'=\tau \circ \overline{h}$, where $\tau$ is a positive Dehn twist along the circle formed by $\gamma$ and the core of $H$. (Here $\overline{h}$ denotes the extension of $h$ to $H$ by the identity.) 
\end{itemize}
\end{definition}

In order to define stabilization along some arc in a page besides $S_0$, we first describe an operation on foliated open books which preserves the contactomorphism type of the associated manifold.
\begin{definition}\cite[Definition 3.16]{LV}
The $1$-\emph{shift} of an abstract foliated open book $(\{S_i\}_{i=0}^{2k},h)$ is the foliated open book \[(\{S_i[1]\}_{i=0}^{2k},h[1])=(\{S_1,S_2,\dots,S_{2k},S_1'\},h'),\] where

\begin{itemize}
\item[-] if $S_{0}\xrightarrow[\gamma_{1}]{\text{\textbf{cut}}} S_1$, then $S_1'$ is defined by the relation $S_{2k}\xrightarrow[{h^{-1}(\gamma_{1})}]{\text{\textbf{cut}}} S_1'$ and $h'$ is the restriction of $h$  to $S_1'$;
\item[-] if $S_{0}\xrightarrow[{p_0, q_0}]{\text{\textbf{add}}} S_1$, then $S_1'$ is defined by the relation $S_{2k}\xrightarrow[{ h^{-1}(p_0), h^{-1}(q_0) }]{\text{\textbf{add}}} S_1'$ and $h'$ is $h$ extended by the identity on the added 1-handle. 
\end{itemize}
An $r$-fold iteration of the shift operation is called an \textit{$r$-shift}, and denoted by $(\{S_i[r]\}_{i=0}^{2k},h[r])$. One can analogously define $r$-shifts for $r<0$.
\end{definition}
Shifts correspond to post-composing $\pi$ with a diffeomorphism of $S^1$.

\begin{definition}
Let $\gamma\subset S_r$ be a properly embedded arc with $\partial \gamma\subset B$. The stabilization of $(\{S_i\}, h)$ along $\gamma\subset S_r$ is the foliated open book $(\{S'_i\}, h')$ defined as follows:

\begin{enumerate}
\item first perform an  $r$-shift of $(\{S_i\},h)$ to $(\{S_i[r]\},h[r])$;
\item\label{item:handle}   stabilize as in Definition~\ref{def:stab} along the image of $\gamma\subset S_0[r]$;
 \item perform a $(-r)$-shift  to obtain $(\{S_i'\},h')$, where  $S_i'$ is still obtained from $S_i$ by a handle attachment along $p$ and $q$.
 \end{enumerate}
\end{definition}

It is easy to see that stabilizing does not change the underlying manifold with foliated boundary.

\begin{theorem}\cite[Proposition 6.8, Theorem 6.9]{LV}\label{thm:supportstab}
Let $(\{S_i\}, h)$ be a foliated open book and $(\{S_i'\}, h')$ be a positive stabilization of $(\{S_i\}, h)$. Then the corresponding contact three-manifolds are contactomorphic. Furthermore, if two foliated open books support contactomorphic foliated contact three-manifolds, then they admit a common positive stabilization. \end{theorem}

\subsubsection{Sorted foliated open books}

Next, we introduce some additional bookkeeping  to record the cutting arcs and cocore arcs on all possible pages of the foliated open book.  This additional information corresponds to a choice of a gradient-like vector field for $\pi$ on the associated smooth manifold, as we will describe in Section~\ref{sec:morse}. 

If $S_{{i-1}}\xrightarrow{\text{\textbf{add}}} S_{i}$, then for all $j\geq {i}$, decorate $S_{j}$ with a cocore of the added handle and label this cocore $\gamma_i^-$. If $S_{{i-1}}\xrightarrow{\text{\textbf{cut}}} S_{i}$ along $\gamma\subset S_{{i-1}}$, then for all $j\leq {i-1}$, decorate $S_j$ with the cutting curve and label it $\gamma_i^+$. 

We call the  $\gamma_i^\pm$ \emph{sorting arcs}. Note that by an abuse of notation we use a single label $\gamma_i^\pm$ to refer to distinct copies of the sorting arcs on multiple $S_j$.  In general, sorting arcs may intersect (and consequently some sorting arcs may appear as a union of disconnected intervals on some pages). We will shortly restrict to the subclass of \emph{sorted} foliated open books \cite[Section 5.3]{LV}, which  are characterized by the requirement that the sorting arcs are disjoint on all pages, together with a prescribed ordering of the intersections between sorting arcs and regular leaves.  In the sorted case, we enhance the notation for describing a foliated open book to reflect these additional choices:

\begin{definition}\label{def:sorted}An abstract foliated open book $(\{S_i\}, h, \{\gamma_i^\pm\})$ is \emph{sorted} if  the following properties hold: 
\begin{enumerate}
  \item\label{item:pm}  on each page $S_i$ and for each component $I$ of $A_i$ with $\partial I=\{e_+, e_-\}$, there exist disjoint subintervals  $I_+$ and $I_-$  such that $I_+$ contains $e_+\cup \bigcup_j \big(\gamma^+_j\cap I\big)$ and  $I_{-}$ contains $e_-\cup \bigcup_j \big(\gamma^-_j\cap I\big)$.
 
\item\label{item:plus} for $i {<} j< l$, on each component of $A_i$ any endpoint of $\gamma_l^+$ is closer to $e_+$ than any endpoint of $\gamma_j^+ $;

\item\label{item:minus} for $j<l\leq i$, on each component of $A_i$ any endpoint of $\gamma_l^-$  lies closer to $e_+$ than any endpoint of $\gamma_j^-$; 

\item on each page, the sorting arcs are connected, properly embedded, and disjoint. 
\end{enumerate}
\end{definition}

See Figure~\ref{fig:sorttorus} in Example~\ref{ex:sorted} for an illustration of the ordering conventions given above.
 Sorted foliated open books are useful  because of their close relation to partial open books, which will be described in Section~\ref{sssec:suff}. Below, we show that a foliated open book need not be sorted.

\begin{example}
The foliated open book from Example~\ref{ex:torus} is not sorted.  To see this, we start on the $S_0$ page and attempt to decorate each successive page with cutting arcs which realize the topological transitions between pages and whose endpoints also satisfy the ordering conditions of Definition~\ref{def:sorted}. 

Cutting along $\gamma_1^+$ on $S_0$ yields $S_1$, while attaching a handle to $S_1$ gives rise to $S_2$, which is decorated with the new cocore $\gamma_2^-$.  However, any cutting arc associated to $h_3$ necessarily intersects $\gamma_2^-$ on $S_2$, as the endpoints of $\gamma_3^+$ must lie closer to the positive ends of $A_2$ than the endpoints of $\gamma_2^-$.  Thus the foliated open book is not sorted. 
\begin{figure}[h]
	\labellist
	\pinlabel {$S_0$} [] at 20 -6
	\pinlabel {$S_1$} [] at 96 -6
	\pinlabel {$S_2$} [] at 162 -6
	\pinlabel {$e_+$} [ ] at 45 24
	\pinlabel {$e_+$} [ ] at 45 81
	\pinlabel {$\gamma_1^+$} [ ] at 25 68
	\pinlabel {$\gamma_2^-$} [ ] at 161 15
	\pinlabel {$\gamma_3^+$} [ ] at 168 68
	\pinlabel {$e_+$} [ ] at 183 24
	\pinlabel {$e_+$} [ ] at 183 81
	\endlabellist
\begin{center}
\includegraphics[scale=1]{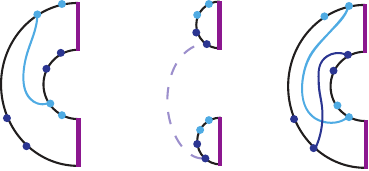}
\end{center}

\caption{The first three pages of the foliated open book from Example~\ref{ex:torus}.}\label{fig:orderofseps}
\end{figure}
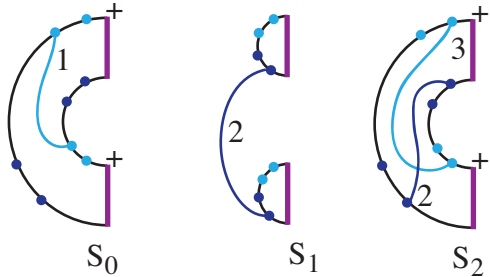
\end{example}

We may always stabilize a foliated open book to obtain a sorted one. 
\begin{theorem} \label{thm:revsortedsupport} 
Any foliated contact three-manifold $(M,\xi,\mathcal{F})$ admits a supporting sorted foliated open book. If two sorted foliated open books support the same foliated contact three-manifold, then we can stabilize each through sorted foliated open books to obtain a common sorted foliated open book. Moreover, all of these stabilizations can be assumed to be performed on the $S_0$ page  
as in Definition \ref{def:stab}.
\end{theorem}

\begin{proof} The first two statements are taken from \cite [Theorem 3.12, Proposition 8.4, Theorem 8.14]{LV}. The final claim that the necessary stabilisations may be assumed to occur on the $S_0$ page is implicit in \cite{LV}, as explained next; we include the brief argument here, although it relies on the connection between foliated and partial open books described in Section~\ref{sssec:suff}.  Specifically, the proof of \cite[Proposition 8.6]{LV}  explains why stabilizations on the $S_0$ page suffice to turn an arbitrary  sorted foliated open book into a sufficiently stabilized\footnote{This subclass consists of the foliated open books which can most easily be related to partial open books, a relationship which we discuss shortly.} foliated open book.    The proof that two sufficiently stabilized open books admit a common stabilization relies on Giroux Correspondence for the associated partial open books, and hence, leads again to stabilizations on $S_0$. 
\end{proof}

\begin{example}\label{ex:sorted}
Figure~\ref{fig:sorttorus} shows that the foliated open book from Example~\ref{ex:torus} may be stabilized to a sorted foliated open book.  Starting from index $S_0$, the first obstruction to being sorted occurred as an intersection between $\gamma_2^-$ and $\gamma_3^+$ on $S_2$.   This intersection may be removed by stabilizing ``at'' the $S_2$ page; that is,  by shifting twice and then stabilizing at the new $S_0$. Each new page is obtained from the corresponding old page by attaching a one-handle with one attaching interval on each binding component.  The new final maps to the new initial page by a positive Dehn twist; after shifting back, this becomes a positive Dehn twist identifying the two copies of the $S_2'$  page, as shown.  Thus  $\gamma_2^-$ changes by a positive Dehn twist as it ascends to higher-index pages, while  $\gamma_3^+$ changes by a negative Dehn twist as it descends to lower-index pages.  

\begin{figure}[h]
	\labellist
	\pinlabel {$S_0'$} [ ] at 40 12
	\pinlabel {$S_1'$} [ ] at 113 12
	\pinlabel {$S_2'$} [ ] at 232 12
	\pinlabel {$S_2'$} [ ] at 200 127
	\pinlabel {$S_3'$} [ ] at 314 12
	\pinlabel {$S_4'$} [ ] at 405 12
	\pinlabel {$e_+$} [ ] at 40 128
	\pinlabel {$e_-$} [ ] at 40 91
	\pinlabel {$e_+$} [ ] at 40 73
	\pinlabel {$e_-$} [ ] at 40 35
	\pinlabel {$\gamma_1^+$} [ ] at 16 97
	\pinlabel {$\gamma_3^+$} [ ] at 50 113
	\pinlabel {$\gamma_3^+$} [ ] at 134 86
	\pinlabel {$\gamma_2^-$} [ ] at 161 54
	\pinlabel {$\gamma_3^+$} [ ] at 195 71
	\pinlabel {$\gamma_2^-$} [ ] at 259 81
	\pinlabel {$\gamma_3^+$} [ ] at 239 138
	\pinlabel {$\gamma_2^-$} [ ] at 335 70
	\pinlabel {$\gamma_4^-$} [ ] at 377 85
	\pinlabel {$\gamma_2^-$} [ ] at 426 63
	\pinlabel {$\tau$} [ ] at 235 50
	\pinlabel {$\tau^{-1}$} [ ] at 197 97
	\endlabellist
\begin{center}
\includegraphics[scale=1]{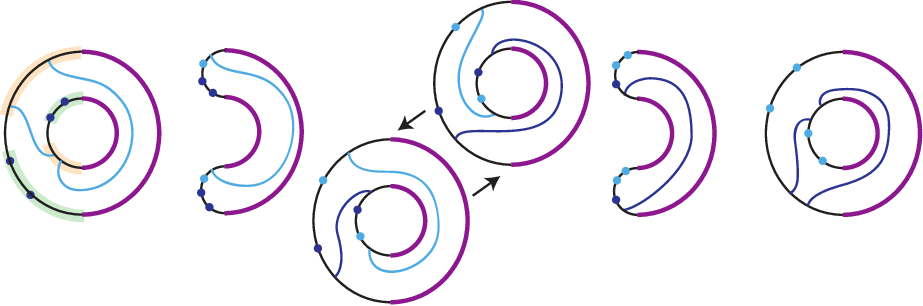} 
\end{center}

\caption{The stabilization arc $\gamma\subset S_0$ is chosen to have one endpoint on each interval of the binding $B$ so that it crosses each of of the intersecting sorting arcs in Figure~\ref{fig:orderofseps} exactly once.  There is a right-handed Dehn twist identifying the two copies of the $S_2'$ page, while the monodromy $h' \co S_4'\to S_0'$ is translation in the plane of the page. The sorting arcs $\gamma^+$ are shown in light blue, and the sorting arcs $\gamma^-$ are shown in dark blue. Subintervals $I_\pm\subset A_0$ are highlighted in green and orange on $S_0$ and may be chosen analogously on the other pages.}\label{fig:sorttorus}
\end{figure}
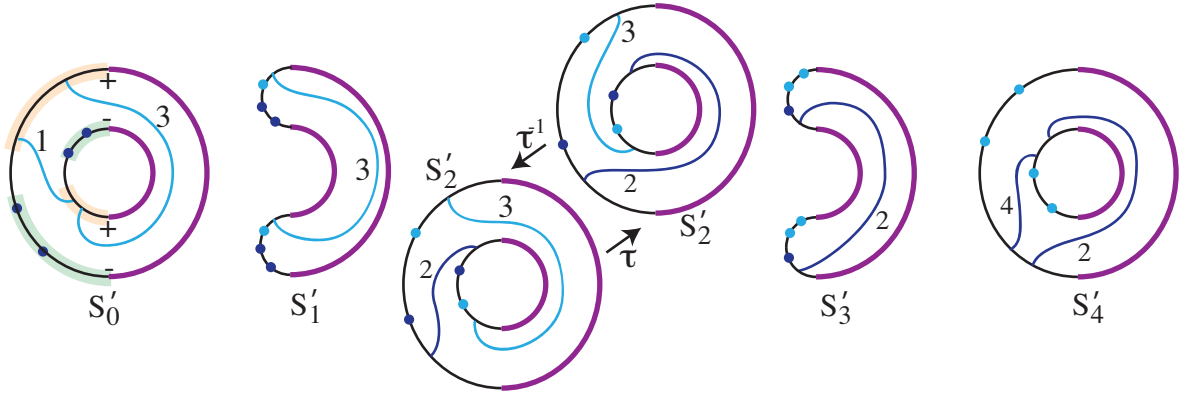

When depicting a sequence of pages, there is an implicit identification by translation in the page unless otherwise noted.  In this example, the one non-translation occurs at $S_2'$ while the map from $S'_4$ to $S_0'$ remains  translation in the plane of the page, as seen on Figure~\ref{fig:sorttorus}.  The result is a sorted foliated open book for the original contact manifold with foliated boundary.  
\end{example}

\subsubsection{Gradient flows on foliated open books}\label{sec:morse}

As in Section~\ref{sec:supp-cs}, we consider the smooth manifold with foliated boundary $(M, \partial M, \mathcal{F})$ constructed from a foliated open book. The initial abstract data determines an equivalence class of circle-valued Morse functions $\pi \co M\setminus B\rightarrow S^1$ \cite[Section 5.2]{LV}.  The level sets of $\pi$ are the pages, and the restriction $\pit \co \partial M\setminus B \rightarrow S^1$ is again a circle-valued Morse function with the same set of critical points.  In this setting, one may consider the critical submanifolds of a gradient-like vector field for $\pi$ parallel to $\partial M$, which thus restricts to a gradient-like vector field for $\pit$.  The intersection of the critical submanifolds with the regular level sets of $\pi$ determines sorting arcs on the pages of the corresponding abstract foliated open book \cite[Section 5.3]{LV}. Thus a foliated open book is sorted when a gradient-like vector field on $\partial M$ that satisfies the ordering conditions for sorting arcs extends to a gradient-like vector field on $M$ whose critical submanifolds are disjoint. 

Any surface with an open book foliation decomposes along regular leaves into square tiles that each contain precisely one hyperbolic singularity.  Figure~\ref{fig:rgradient} shows the flowlines of  a preferred $\nabla \pit$ on such a tile; these flowlines satisfy the ordering conditions in Definition~\ref{def:sorted}, although ensuring their extension to $M$ yields a sorted foliated open book may require stabilization.  

\begin{figure}[h]
	\labellist
	\pinlabel {$e_+$} [B] at 3 183
	\pinlabel {$e_-$} [B] at 198 183
	\pinlabel {$e_-$} [B] at 3 -5
	\pinlabel {$e_+$} [B] at 198 -5
	\endlabellist
\begin{center}
\includegraphics[scale=.65]{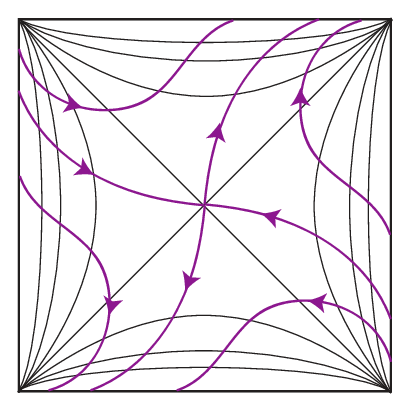}
\end{center}

\caption{The tile of $\fol$ shows  purple flowlines of $\nabla \pit$ which satisfy the sorted ordering conditions.}\label{fig:rgradient}
\end{figure}
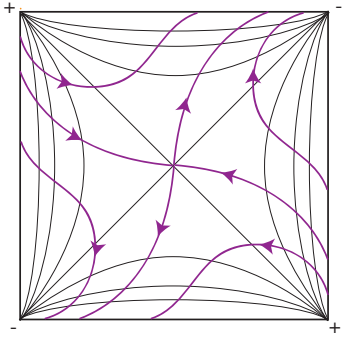

The Morse perspective is useful not only for building geometric intuition, but also for relating the boundary decorations to those in the interior of the manifold.  For example, a sorting arc $\gamma^{+}_i$ is isotopic through the associated stable critical submanifold of $\nabla\pi$ to the stable submanifold of $\nabla\pit$ on $\partial M$ which passes through 
the corresponding positive hyperbolic point $h_i$. Likewise, a sorting arc $\gamma^{-}_i$ is isotopic through the associated unstable critical submanifold of $\nabla\pi$ to the unstable submanifold of $\nabla\pit$ on $\partial M$ which passes through 
the corresponding negative hyperbolic point $h_i$.
  Conditions~\ref{item:pm}, \ref{item:plus}, and \ref{item:minus} in Definition~\ref{def:sorted} are chosen so that the graphs formed by  the (un)stable separatrices of positive (negative) hyperbolic points with respect to $\fol$ and with respect to $\nabla \pit$ are isotopic. See  \cite[Lemma 8.8]{LV}. In the next section, this will allow us to define certain subsurfaces as neighborhoods of either graph. 

\subsubsection{Sufficiently stabilized foliated open books and the  associated partial open book} \label{sssec:suff}

Sorted foliated open books are closely related to partial open books, and \cite{LV} introduced the technical designation of a ``sufficiently stabilized" foliated open book. In this section, we recall this definition and show that a sufficiently stabilized foliated open book naturally has a companion partial open book.  Throughout this section, we will let $(\{S_i\}_{i=0}^{2k},h,\{\gamma_i^\pm\})$ be a sorted foliated open book and $(M, \xi, \fol)$ be the supported foliated contact three-manifold.  We further assume that either $k>0$ or $|\partial M|>1$; the only foliated open books this excludes are those formed from an honest open book by removing a neighborhood of a single point on the binding.

Set
\[R_i=N_{\lrcorner}\big(A_i \cup(\bigcup_{i < j}\gamma_j^+)\cup(\bigcup_{i\geq j}\gamma_j^-)\big) \subset S_i,\]
 where $N_{\lrcorner}$ denotes a ``cornered'' neighborhood of $A_i \cup(\bigcup_{i < j}\gamma_j^+)\cup(\bigcup_{i\geq j}\gamma_j^-)$, with corners at $E$, so that $R_i$ meets $B$ only at $E$.  
Furthermore, define \[P_i=\overline{S_i\setminus R_i},\]
as shown on the left in Figure~\ref{fig:sssorttorus}.
Since only the $R_i$  change with $i$, we can identify all the $P_i$ with each other by the flow of the gradient-like vector field. {We denote the composition of these identifications from $P_0\subset S_0$ onto $P_{2k}\subset S_{2k}$ by $\iota$}.  
{Let $\widetilde{S}=S_{0}$, $\widetilde{P}=P_{0}$ and $\tilde{h}=h\vert_{P_{2k}}\circ \iota$}.

\begin{definition}\cite[Proof of Proposition 8.6]{LV} The sorted foliated open book $(\{S_i\}_{i=0}^{2k}, h, \{\gamma_i\})$ is \emph{sufficiently stabilized}  if each component of $P_0$ intersects the boundary of $R_0$ in at least two intervals.
\end{definition}

This condition ensures that $(\widetilde{S}, \widetilde{P}, \tilde{h})$ is a partial open book whose pages embed into the pages of $(\{S_i\}_{i=0}^{2k}, h, \{\gamma_i\})$. As the term suggests, this condition may always be achieved by a sequence of positive stabilizations, and in fact, these may be chosen to occur on the $S_0$ page.

\begin{theorem}\cite[Proposition 8.10]{LV} \label{thm:fobtopob}
With the notation above, the partial open book $(\widetilde{S},\widetilde{P},\widetilde{h})$ is compatible with a contact manifold $(\widetilde{M},\widetilde{\xi})$ which is
contactomorphic to $(M,\xi)$.  Furthermore, under this contactomorphism, the image of the dividing set of the characteristic foliation of $\xi$ on $\partial M$ divides $\widetilde{\mathcal{F}}$. 
\end{theorem}

\begin{example} 
In this example, we return to the foliated open book of Example~\ref{ex:sorted}, which is not sufficiently stabilized.  Recall that $R'_0$ is constructed as a cornered neighborhood of the sorting arcs together with $A'_0$.  Observe that $P'_0:=S_0'\setminus R_0'$ consists of two bigons, each with an edge on $B$ and an edge on the annular $\partial R'_0$.  It follows that $S'_0$ cannot be built up from $R'_0$ by adding one-handles, which is the topological requirement for the surface\slash subsurface pair to define a partial open book.

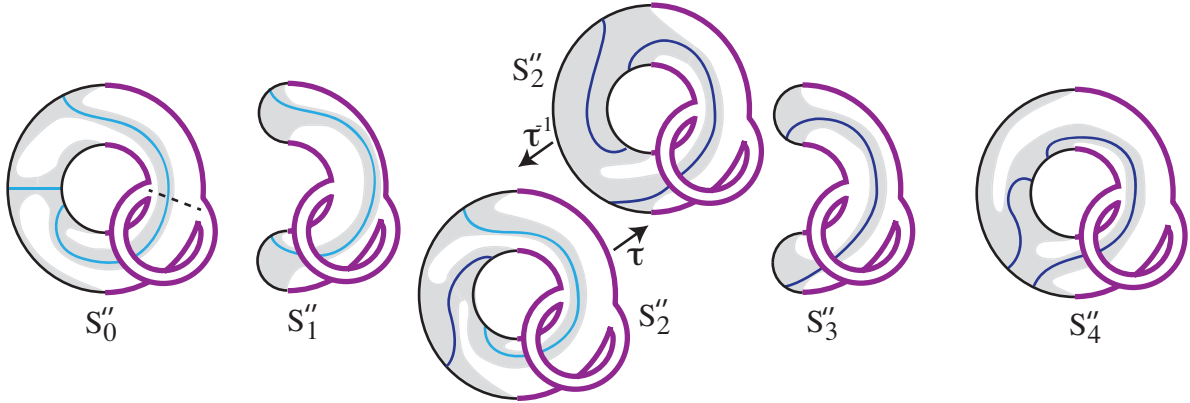
\begin{figure}[h]
	\labellist
	\pinlabel {$S_0''$} [ ] at 40 12
	\pinlabel {$S_1''$} [ ] at 113 12
	\pinlabel {$S_2''$} [ ] at 242 12
	\pinlabel {$S_2''$} [ ] at 200 127
	\pinlabel {$S_3''$} [ ] at 314 12
	\pinlabel {$S_4''$} [ ] at 405 12
	\pinlabel {$\tau$} [ ] at 237 53
	\pinlabel {$\tau^{-1}$} [ ] at 197 100
	\endlabellist
\begin{center}
\includegraphics[scale=1]{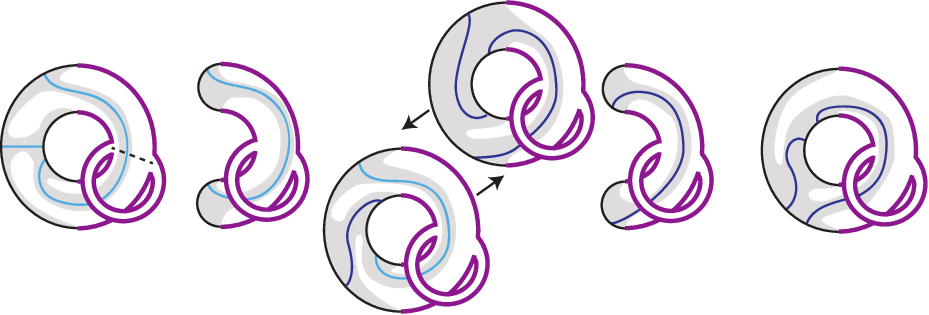}
\end{center}
\caption{ The shaded regions show the cornered neighborhoods $R_i''$ in each page of the sufficiently stabilized open book. As in Example~\ref{ex:sorted}, successive pages are identified by translation except  for the right-handed Dehn twist $\tau$  at $S_2''$.}\label{fig:sssorttorus}
\end{figure}

Figure~\ref{fig:sssorttorus} shows a sufficiently stabilized foliated open book for the same manifold, built by stabilizing along the dotted arc which  connected the two components of the binding.  The new monodromy is the composition of the original monodromy  with a positive Dehn twist along the circle formed by the core of the added handle and the stabilizing arc.

 The new $P''_0:=S_0''\setminus R_0''$ is a single disk whose boundary intersects $\partial R''_0$ along two intervals; this ensures that $(\widetilde{S}'', \widetilde{P}'', \widetilde{h}'')$ is a partial open book.   In order to understand $\widetilde{h}''$, we first examine the flow of $P_0''$ through the manifold. Recall from Example~\ref{ex:sorted} that the only non-trivial identification between successive pages is a right-handed Dehn twist at the $S_2''$ page before the cut yielding $S_3''$.  The twist is along the core of the annular $R_2''$, so $P_4''$ is isotopic, relative to $A_4''=A_0''$, to $P_0''$.   Thus, the partial open book monodromy is simply the restriction of the foliated open book monodromy to $P_4''$.

\end{example}

As shown in \cite[Lemma 8.12]{LV}  the cornered diffeomorphism type of the subsurfaces $R_i\subset S_i$ depends only on the foliation $\mathcal{F}$, rather than on the pages $S_i$. Construct the corresponding contact three-manifold $(M,\xi,\mathcal{F})$ for a sorted abstract foliated open book $(\{S_i\},h,\{\gamma_i^\pm\})$. Recall from Definition~\ref{def:div} that the subsurface $R_+(\fol)$ (respectively $R_-(\fol)$) of $\partial M$ is a closed neighborhood of the (un)stable separatrices corresponding to positive (negative) hyperbolic points of $\mathcal{F}$. 

As noted in Section~\ref{sec:morse}, the graph of positive separatrices of $\nabla \pit$ from positive hyperbolic points on $\partial M$ is (non-smoothly)  isotopic to the graph of positive separatrices of $\fol$ from positive hyperbolic points on $\partial M$.  On the other hand, the former is isotopic through the stable critical submanifolds to a neighborhood of the sorting arcs $\cup_{H_+}\gamma^+_i$ on $S_0$.   This yields an identification between $R_+(\fol)\subset \partial M$ and $R_0\subset S_0$.  Similarly, studying negative separatrices of negative hyperbolic points  and unstable critical submanifolds yields an identification between $R_-(\fol)\subset \partial M$ and $R_{2k}\subset S_{2k}$. 

Similarly, all the $R_i$ can be identified with surfaces described only using the foliation $\fol$, independently of the foliated open book, as follows. Recall that $R_i$ is the (cornered) neighborhood of the union of $A_i$  and the intersection of all stable and unstable submanifolds of $\nabla \pi$ with $S_i$. Each of these intersections can be pushed up (or down) onto $\partial M$ along the corresponding stable (or unstable) submanifolds. While performing these isotopies we push the half-neighborhood of $A_i$ up along $I_+$ and down along $I_-$; the surface obtained after the isotopy is not a submanifold of $\partial M$, as it has a half twist in the middle of each component $I$ of $A_i$ which extends into the interior of $M$. See Figure \ref{fig:rtoq}.
\begin{figure}[h]
\begin{center}
\includegraphics[scale=.7]{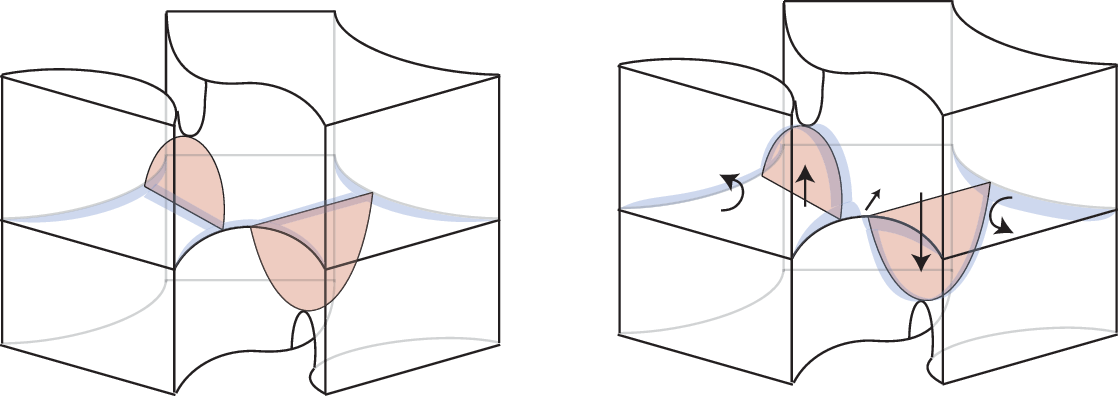}
\end{center}

\caption{The blue region on the left-hand picture is the cornered neighborhood $R_i\subset S_i$.  On the right-hand figure, $R_i$ has been isotoped through $M$ in the direction of the indicated arrows to lie mostly on $\partial M$.  The small arrow indicates the half twist which extends into the interior of $M$. Note that $B$ has been blown up to $B\times I$ for clarity.}\label{fig:rtoq}
\end{figure}

\subsubsection{Gluing foliated open books.}\label{ssec:gluing} In order to glue foliated open books, we describe the matching conditions imposed upon their boundary foliations. 
Recall that a foliation $\fol$ on a surface $F$ always refers to the data of a signed foliation together with the function $\pit\colon F\setminus E\to S^1$ whose level sets are the leaves of $\fol$. 

\begin{definition} The \emph{reverse of a foliation $\fol$} on $F$ is the foliation on $-F$ which is equal to $\fol$ pointwise, but with the leaf orientations and signs of singular points reversed.  \end{definition}

The main advantage of foliated open books is that they are well suited for gluing. Suppose that  $(M^L,\xi^L,\mathcal{F}^L)$ and $(M^R,\xi^R,\fol^R)$ are foliated contact three-manifolds such that there exists  an orientation-reversing diffeomorphism $\psi\colon \partial M^L  \to \partial M^R$ that maps the reverse of the foliation $\fol^L$ onto $\fol^R$. 

Recall that by hypothesis, the boundary foliations  are  divided, so the boundaries $\partial M^{\bullet}$ are convex with respect to $\xi^{\bullet}$. (Here, and throughout this section, $\bullet$ will be an element of the set $\{R, L\}$.) Since the contact structures are $I$-invariant near the boundaries, $\psi$ determines a closed contact three-manifold \[(M,\xi)=(M^L\cup_\psi M^R,\xi^L\cup_\psi\xi^R).\]  

If the initial contact manifolds were supported by foliated open books $(\{S^L_i\},h^L)$ and  $(\{S^R_i\},h^R)$, respectively, then $(M,\xi)$ naturally inherits a supporting open book whose pages and binding are built by gluing the pages and bindings of the constituent pieces.  More precisely, recall that  $\partial S^{\bullet}_i=A_i^{\bullet}\cup B^{\bullet}$, where all the $B^{\bullet}$ are identified   when forming the manifolds $M^{\bullet}$.  The map $\psi$ identifies the intervals $A_i^L$ and $-A_i^R$ for all $i$, forming the surfaces \[S_i=S_i^L\bigcup_{A_i^L\xrightarrow{\psi}-A_i^R} S_i^R\] with boundary $B=B^L\cup B^R$.  Observe that a cut to $S_i^R$ pairs with a handle addition to $S_i^L$, and vice versa, so the surfaces  $S_i$  are  diffeomorphic  for all $i$.  This allows us to identify  $S_0\cong S_1 \cong\cdots \cong S_{2k}$ and we denote the composition of these identifications by  $\iota \colon S_0\to S_{2k}$.  Note that $\iota$  fixes $B$, {and it restricts to $P_{0}^\bullet$ as the identification $\iota^\bullet:P_0^\bullet \rightarrow P_{2k}^\bullet$  defined in Section \ref{sssec:suff}}. Letting $S:=S_0$,  the monodromy $h\colon S\to S$ for the glued-up open book can be obtained as the composition of $\iota $  with $h^L\cup h^R\colon S_{2k}=S_{2k}^L\cup S_{2k}^R\to S_{0}^L\cup S_{0}^R=S_0$. By construction, $h$ fixes $B$, as required.

\begin{theorem}\cite[Theorem 6.2]{LV} \label{prop:glue} 
Suppose that the foliated open books $(\{S_i^L\},h^L)$ and $(\{S_i^R\},h^R)$ define the three-manifolds with foliated boundary
 $(M^L,\fol^L)$ and $(M^R,\fol^R)$, and assume that there is an orientation-reversing diffeomorphism $\varphi\colon \partial M^L\to \partial M^R$ that takes the reverse of the foliation $\fol^L$ onto $\fol^R$. 

Then there are contact structures $\xi^L$ and $\xi^R$ supported by  $(\{S_i^L\},h^L)$ and $(\{S_i^R\},h^R)$, respectively, so that $\xi=\xi^L\cup_\varphi\xi^R$ is a contact structure $\xi$ on the manifold $M=M^L\cup_\varphi M^R$  that is supported by the honest open book $(S,h)$ constructed by pagewise gluing, as defined above.
\end{theorem}

Since $\iota \colon S_0\to S_{2k}$ is given as a sequence of identifications, the monodromy $h$ may be  difficult  to reconstruct in a complicated case, and the above construction does not automatically give a factorization in terms of  Dehn twists. For sorted open books, however, there is a straightforward way to describe $\iota$ (and thus $h$) as follows. 
Suppose now, that $(\{S^L_i\},h^L, \{\gamma_i^{\pm,L}\})$ and  $(\{S^R_i\},h^R,\{\gamma_i^{\mp,R}\})$ are sorted foliated open books compatible with 
 $(M^L,\xi^L,\mathcal{F}^L)$ and $(M^R,\xi^R,\fol^R)$, respectively. 
Recall that each page $S_i^{\bullet}$  decomposes as the union of a ``constant'' part $P_i^{\bullet}$ and an $i$-dependent part $R_i^{\bullet}$ which is determined by the foliation $\fol^{\bullet}$. By construction, $P_0^{\bullet}\cong P_1^{\bullet}\cong \cdots \cong P_{2k}^{\bullet}$. These diffeomorphisms are explicit in the construction of the pages $S_i^{\bullet}$ and together they determine $\iota\vert_{P_0^L\cup P_0^R}=\iota^L\cup \iota^R\colon P_0^L\cup P_0^R\to P_{2k}^L\cup P_{2k}^R$.

As  described earlier,  sorted foliated open books have the property that $R_0^L$ is isotopic in $M^L$ to $R_+(\fol^L)$.  This surface is mapped by $\psi$ onto $R_-(\fol^R)$, which in turn is isotopic in $M^R$ to $R_{2k}^R$. The composition of these three maps gives $\iota\vert_{R_0^L}\colon R_0^L\to R_{2k}^R$, and we can similarly obtain $\iota\vert_{R_0^R}\colon R_0^R\to R_{2k}^L$. This yields a concrete description of  the entire map $\iota$ on the remaining parts $R_{2k}^L\cup R_{2k}^L$. Together,  the above maps determine $h$ everywhere.


\section{Construction of the contact invariant}
\label{sec:construction}

\subsection{Bordered manifold associated to a triple}\label{ssec:suturedbordered}

Let $(M,\xi,\mathcal{F})$ be a foliated contact three-manifold with signed singular foliation $\mathcal{F}=(\pit, H=H_-\cup H_+,E = E_-\cup E_+)$. Set $\bsGamma:=\pit^{-1}(0)$, so that $\bsGamma$ is a disjoint union of oriented leaves connecting positive to negative elliptic points. For each leaf $I$,  let $e_+(I)$ and $e_-(I)$ be the corresponding positive and negative elliptic points, respectively. Choose $\epsilon>0$ small enough so that $\widetilde{\pi}^{-1}[-\epsilon, \epsilon]$ contains no critical points of $\pit$. 
Note that then $\pit^{-1}[0,\epsilon]$ is a union of disks, one disk containing each leaf $I\subset\bsGamma$ and denoted by $R_+(I)$.  Similarly, denote the connected component of $\pit^{-1}[-\epsilon,0]$ containing $I$ by $R_-(I)$. Write  $R_+(\bsGamma)= \bigcup_{I\subset \bsGamma} R_{+}(I)$ and  $R_-(\bsGamma)= \bigcup_{I\subset \bsGamma} R_{-}(I)$.  Set $F=\bdy M\setminus (R_+(\bsGamma)\cup R_-(\bsGamma))$.

Next, we use the foliation to define a natural parametrization of $F$ via an arc diagram $\zz=(Z,a,m)$. 
 For each hyperbolic point $h_i$, let $\delta_i$ be the union of the two stable separatrices at $h_i$ if $h_i$ is positive, or the union of the two unstable separatrices at $h_i$ if $h_i$ is negative.
Let $Z(I)$ be a pushoff of $\pit^{-1}(-\epsilon)\cap R_-(I)\subset \bdy F$ 
into $F$ satisfying the following: 
\begin{itemize}
\item $\bdy Z(I)=\left(-e_-(I)\right)\cup e_+(I)$
\item if $\delta_i$ has an endpoint at $e_\pm(I)$, then $Z(I)$ intersects $\delta_i$ in a unique point; otherwise,  $\delta_i$ and $ Z(I)$ are disjoint.
\end{itemize}
Let $Z=\sqcup_{I\subset \bsGamma} Z(I)$.  Define $a$ to be the set of all intersection points of $Z$ with the union of the 
$\delta_i$ and let $m$ be the pairing induced on the points in $a$ by $\delta_i$. See Figure~\ref{fig:torus-bs-NEW}. 
Observe that $Z$ divides $F$ into two subsurfaces: a surface containing all hyperbolic points (shown in white on Figure~\ref{fig:torus-bs-NEW}), and a union of disks (shown in dark red on Figure~\ref{fig:torus-bs-NEW}). 
Let $e_i$ be the curve that is the intersection of $\delta_i$ with the white surface.
Note that after removing the disks bounded by the circles $I\cup Z(I)$ from $F$, and decomposing the resulting surface along the union of $e_i$, we get a disjoint union of disks so that each one of them contains exactly one of the leaves in $\pi^{-1}(\epsilon)$ on its boundary. Thus, $\zz=(Z,a,m)$ parametrizes $F$ as a \emph{$\beta$-type arc diagram} as defined in \cite[Definition 3.2]{bs-JG}; i.e., $F$ can be  identified with $F(\mathcal{Z})$.

\begin{figure}[h]
\begin{center}
\labellist
  	\pinlabel $\textcolor{blue}{\delta_1}$ at 65 55
  	\pinlabel $\textcolor{blue}{\delta_1}$ at 158 100
	\pinlabel $\textcolor{blue}{\delta_2}$ at 158 30
	\pinlabel $\textcolor{blue}{\delta_3}$ at 227 30
	\pinlabel $\textcolor{blue}{\delta_4}$ at 227 100
	\pinlabel $\textcolor{OliveGreen}{e_-}$ at 116 77
	\pinlabel $\textcolor{OliveGreen}{e_+}$ at 116 145
	\pinlabel $\textcolor{OliveGreen}{e_-}$ at 32 90
	\pinlabel $\textcolor{OliveGreen}{e_+}$ at 32 140
	\pinlabel $\textcolor{OliveGreen}{\bsGamma}$ at 116 109
	\pinlabel $Z$ at 211 109
\endlabellist
\includegraphics[scale=1.4]{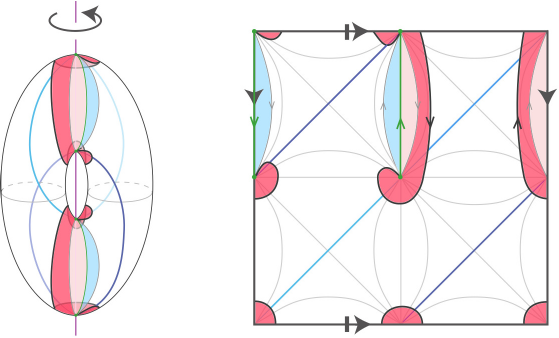}
\end{center}

\caption{In this picture, $R_+(\bsGamma)$ and $R_-(\bsGamma)$ are colored in light blue and light red, respectively. Thus, $F$ is the union of the dark red and white regions. The green curves are $\bsGamma$, the blue curves are the $\delta_i$, so that their intersections with the white surface are the $e_i$. Furthermore, dark red may be identified with $Z\times [0,1/2]$ under the identification of $F$ with $F(\zz)$, and the $e_i$ may be identified with the parametrizing arcs for $\zz$.}
\label{fig:torus-bs-NEW}
\end{figure}

\begin{definition}\label{def:fc-bs}
The triple $(M,\bsGamma,\zz)$ described above is the \emph{bordered sutured manifold associated to the foliated contact three-manifold $(M,\xi,\mathcal{F})$}.
\end{definition}

\subsection{Adapted bordered Heegaard diagram}\label{ssec:bordereddiagram}

Let $(\{S_i\}_{i=0}^{2k},h, \{\gamma_i^\pm\})$ be an abstract sorted foliated open book compatible with the foliated contact three-manifold $(M,\xi,\mathcal{F})$.
In what follows, we will describe a generator in a bordered sutured Heegaard diagram $\HH = (\Sigma, \betas,\alphas, \zz)$ constructed from the data of  $(\{S_i\},h, \{\gamma_i^\pm\})$. Let $g_i$ be the genus of $S_i$ and let $n_i$ be the number of boundary components of $S_i$. Recall that  the boundary of the cornered surface $S_i$ is $B\cup A_i$, where $B$ is a union of circles and arcs, and $A_i$ is a union of intervals only. 
 
We let $\Sigma=S_{0}\cup_{B}-S_{0}$.  In order to distinguish the two copies, we will write 
\[\Sigma=S_{\epsilon}\cup_{B}-S_{0},\] 
but we emphasize that $S_\epsilon$ can be identified with $S_0$. The surface $\Sigma$ has genus $2g_0+n_0-1$ and $|A_0|$ boundary components.  For $i\in H_+$, let $\gamma_i^+$ be the $S_{\epsilon}$ copy of the associated sorting arc.  The endpoints of $\gamma_i^+$ lie  near the $E_+$ end of intervals of $A_\epsilon$. Isotope the arcs $\{\gamma_i^+\}$ (simultaneously, to preserve disjointness) near the endpoints by pushing them along $\bdy\Sigma$ in the direction opposite the orientation of the boundary until the endpoints all lie in $I_+\subset A_0$;  the isotopy stops after crossing $E_+$ and before encountering $\cup_{j\in H_-}h(\gamma_j^-)\subset S_0$. Call the resulting arcs $\beta_i^a$. 

For $i\in H_-$, consider the $S_{2k}$ copies of the sorting arcs $\gamma_i^-$, and let $\beta_i^a = h(\gamma_i^-)$ on $S_0$. 
Write $\betas^a = \{\beta_1^a, \ldots, \beta_{2k}^a\}$.   For convenience, instead of $\beta_i^a$ we will sometimes write $\beta_i^{\pm}$ if $i\in H_{\pm}$.

For $i\in H_{+}$, let $b_i^+ = \beta_i^a\cap S_{\epsilon}$. That is, $b_i^+$ is the arc $\beta_i^a$ with the two end segments that lie on $-S_0$ removed, so that $\partial b_i^+$ lies on $B$.
We claim that after cutting $S_{\epsilon}$ along the arcs $b_i^+$, each connected component contains at least one interval of $A_{\epsilon}$.   

\begin{lemma}\label{lem:cut-gamma}
Each component of  $S_{\epsilon}\setminus \cup_{i\in H_+}b_i^+$ contains an interval component of $A_{\epsilon}$, and hence at least one point in $E_+$, and, equivalently, at least one point in $E_-$. 
\end{lemma}
\begin{proof}
Recall that moving along any given interval of $A_{\epsilon}$, we encounter curves $\gamma_i^+$ indexed in decreasing order. Thus, after the isotopy, moving along any given interval of $B$, we encounter curves $b_i^+$ indexed in decreasing order as well.  Now, suppose there is a component  $C$ of $S_{\epsilon}\setminus \cup_{i\in H_+}b_i^+$ that does not intersect $A_{\epsilon}$. Then the boundary of  $C$ consists of possibly some complete binding circles and at least one circle that alternates between intervals of $B$ and entire arcs $b_i^+$. Since $C$ is a subsurface of $S_{\epsilon}$, the orientation on $B$ agrees with the orientation on $\bdy C$. When traversing this circle of $\bdy C$, it follows  that each interval of $B$ starts at an intersection point with some $b_i^+$ and ends at an intersection point with some $b_j^+$ where $j<i$, thus inducing a circular $<$ ordering on integral indices, which is a contradiction.
\end{proof}

 Let ${\boldsymbol b} = \{b_1, \ldots, b_{2g_0+n_0+|A_0|-k-2}\}$ be a set of cutting arcs for $P_{\epsilon}\subset S_{\epsilon}$ disjoint from $\betas^a$ and with endpoints on $B$, so that each connected component of $S_{\epsilon}\setminus ({\boldsymbol b}\cup \betas^a) = S_{\epsilon}\setminus ({\boldsymbol b}\cup \{b_i^+\}_{i\in H_+})$ is a disk with exactly one interval of $A_{\epsilon}$ on its boundary.
 Note that Lemma~\ref{lem:cut-gamma} guarantees this can be achieved.  In other words, ${\boldsymbol b}$ is a basis for $H_1(P_\epsilon, B)$. Recalling the identification  $S_\epsilon=S_0$, we may push $b_i\subset S_0$ through $M$ to lie on $S_0$ again and define
   \[\beta_i=b_i\cup -h\circ \iota(b_i)\subset S_{\epsilon}\cup_{B}-S_{0},\]
   where $\iota$ is the identification of $P_0$ with $P_{2k}$ from  Section \ref{sssec:suff}. 
Write $\betas^c=\{\beta_1, \dots, \beta_{2g_0+n_0+|A_0|-k-2}\}$.
  
  Note that by modifying the isotopy that transforms $\{\gamma_i^+\}$ into $\{\beta_i^+\}$, possibly forcing it to happen in a smaller neighborhood of $E_+$, we may assume that all curves in $\betas^a \cup \betas^c$ are pairwise disjoint.

 For each cutting arc $b_i\in \boldsymbol b$ on $S_\epsilon$, let $a_i$ be an isotopic curve formed by pushing the endpoints negatively along the boundary so that $a_i$ and $b_i$ intersect once transversely. Similarly, for each arc $b_j^+$, let $\tilde a_j$ be an isotopic curve formed by pushing the endpoints negatively along the boundary so that $\tilde a_j$ and $b_j^+$ intersect once transversely.  We ``double" each of these arcs to form the $\alpha$-circles which define the handlebody $S_0\times [0,\epsilon]$. Namely, define  
  \begin{align*}
 \alpha_i &= a_i\cup -a_i \subset S_{\epsilon}\cup_{B}-S_{0}\\
 \widetilde{\alpha}_j &= \tilde a_j\cup -\tilde a_j \subset S_{\epsilon}\cup_{B}-S_{0},
 \end{align*}
\noindent and write $\alphas = \{\widetilde\alpha_i\}_{i\in H_+}\cup \{\alpha_1, \dots, \alpha_{2g_0+n_0+|A_0|-k-2}\}$. Finally, let
\[Z=\left(\bigcup_{I\subset A_{\epsilon}}(I_+\cup I_-)\right)\cup -A_0\subset \bdy(S_{\epsilon}\cup_{B}-S_{0}) = \bdy\Sigma.\]
We obtain $\zz = (Z, \bdy\betas, m)$, where $m$ matches a pair of points if they belong to the same $\beta$-arc, and we get an identification of $G(\zz)$ with $Z\cup \betas^a\subset \Sigma$.

We say that a bordered sutured Heegaard diagram constructed as above is \emph{adapted} to the sorted abstract foliated open book $(\{S_i\},h, \{\gamma_i^\pm\})$. 

Let $\HH = (\Sigma, {\alphas}, {\betas},\zz)$ be a bordered sutured Heegaard diagram adapted to $(\{S_i\},h, \{\gamma_i^\pm\})$.
 Using the  notation introduced above, 
define the set 
\[\xxx = \{x_1, \ldots, x_{2g_0+n_0+|A_0|-k-2}\} \cup \{x_i^+ \mid i\in H_+\}\]
as the unique intersection points 
\begin{align*}
x_i &=  a_i\cap b_i \in S_{\epsilon}\subset \Sigma\\
x_i^+ &= \tilde a_i\cap b_i^+\in S_{\epsilon} \quad \textrm{ if } i\in H_+.
\end{align*}

\begin{example}\label{ex:HeegaardDiagram}
We show the step-by-step construction of the bordered sutured Heegaard diagram adapted to the sorted foliated open book from Example~\ref{ex:sorted}. {To simplify notation, we write $S_i$ and $h$ instead of the $S_i'$ and $h'$ labels used in Example~\ref{ex:sorted}.}

	In Figure \ref{fig:HD1_betaplus}, we see the sorting arcs $\gamma_i^+$ and their images $\beta_i^+$ after the isotopy taking their endpoints to $I_+\subset A_0$. The arcs $b_i^+$ are the intersections $\beta_i^+ \cap S_\epsilon$.
	
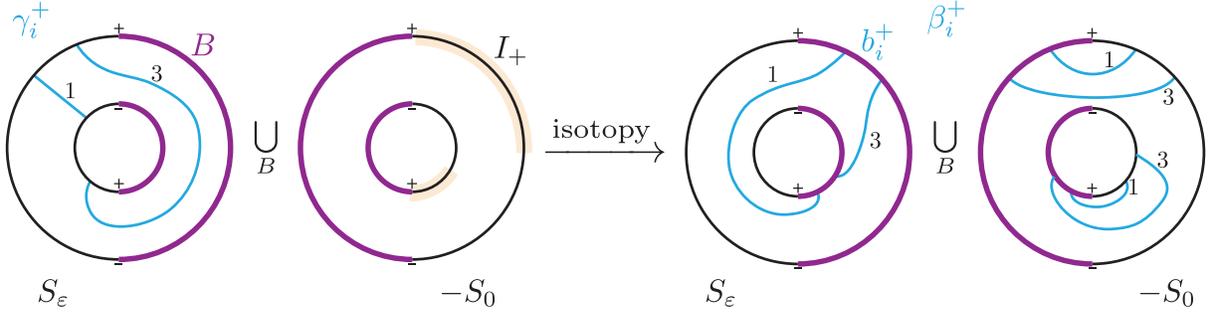
\begin{figure}[h]
	\labellist
	\pinlabel {$e_+$} [b] at 75 142
	\pinlabel {$e_-$} [t] at 75 101
	\pinlabel {$e_+$} [b] at 75 45
	\pinlabel {$e_-$} [t] at 75 3
	\pinlabel {$e_+$} [b] at 258 142
	\pinlabel {$e_-$} [t] at 258 101
	\pinlabel {$e_+$} [b] at 258 45
	\pinlabel {$e_-$} [t] at 258 3
	\pinlabel {$e_+$} [b] at 501 142
	\pinlabel {$e_-$} [t] at 501 101
	\pinlabel {$e_+$} [b] at 501 45
	\pinlabel {$e_-$} [t] at 501 3
	\pinlabel {$e_+$} [b] at 686 142
	\pinlabel {$e_-$} [t] at 686 101
	\pinlabel {$e_+$} [b] at 686 45
	\pinlabel {$e_-$} [t] at 686 3
	\pinlabel {isotopy} [ ] at 374 81
	\pinlabel {$B$} [ ] at 163 59
	\pinlabel {$B$} [ ] at 590 59
	\pinlabel {$S_\varepsilon$} [ ] at 12 4
	\pinlabel {$-S_0$} [ ] at 310 4
	\pinlabel {$S_\varepsilon$} [ ] at 442 4
	\pinlabel {$-S_0$} [ ] at 741 4
	\pinlabel {${\color{Plum}B}$} [ ] at 122 136
	\pinlabel {$I_+$} [ ] at 313 135
	\pinlabel {$\gamma_1^+$} [ ] at 29 93
	\pinlabel {$\gamma_3^+$} [ ] at 80 128
	\pinlabel {$1$} [ ] at 487 123
	\pinlabel {$3$} [ ] at 547 86
	\pinlabel {{\color{cyan}$b_i^+$}} [ ] at 550 136
	\pinlabel {{\color{cyan}$\beta_i^+$}} [ ] at 594 148
	\pinlabel {$1$} [ ] at 690 132
	\pinlabel {$3$} [ ] at 718 119
	\pinlabel {$1$} [ ] at 707 40
	\pinlabel {$3$} [ ] at 733 60
	\endlabellist
\begin{center}
\includegraphics[scale=0.6]{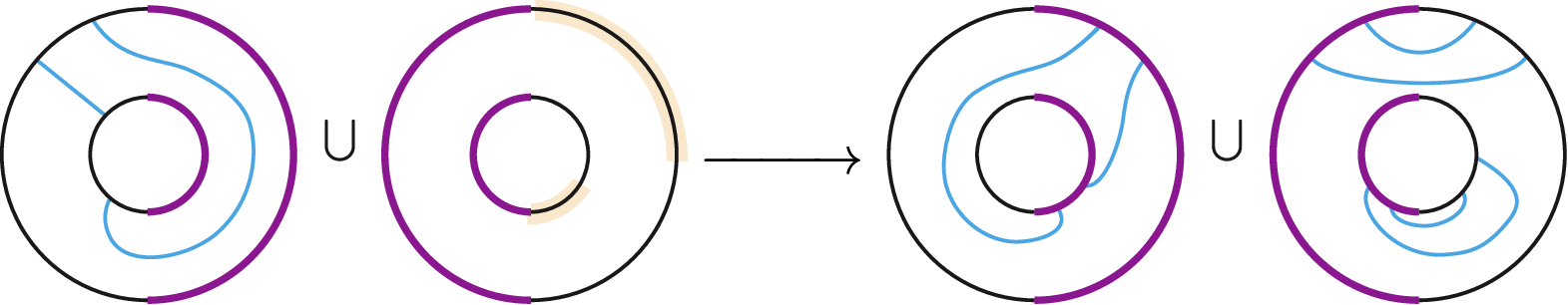}
\end{center}

\caption{The sorting arcs $\gamma_i^+$ (on the left) and their images $\beta_i^+$ (on the right) after the isotopy taking their endpoints to the highlighted subintervals $I_+\subset A_0$. The parts of $\beta_i^+$ lying on $S_\epsilon$ are $b_i^+$.}\label{fig:HD1_betaplus}
\end{figure}
	
	Figure \ref{fig:HD2_betaminus} shows the sorting arcs $\gamma_{i}^-$ on $S_0$ and, after mirroring, their images $\beta_i^- = -h(\gamma_{i}^-)$ on $-S_0$; recall that in this example $h$ is the identity. Recall that $\betas^a = \{\beta_i^{+} \mid i\in H_+\}\cup \{\beta_i^{-} \mid i\in H_-\}$.	
	
	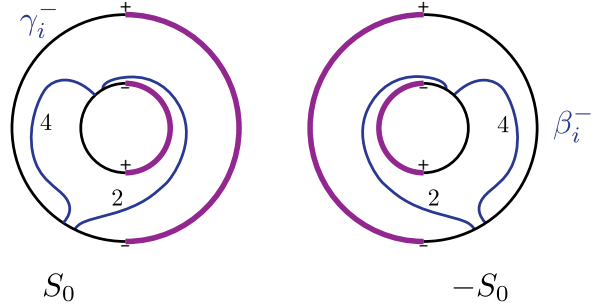
\begin{figure}[h]
		\labellist
		\pinlabel {$e_+$} [b] at 75 142
		\pinlabel {$e_-$} [t] at 75 101
		\pinlabel {$e_+$} [b] at 75 45
		\pinlabel {$e_-$} [t] at 75 3
		\pinlabel {$e_+$} [b] at 258 142
		\pinlabel {$e_-$} [t] at 258 101
		\pinlabel {$e_+$} [b] at 258 45
		\pinlabel {$e_-$} [t] at 258 3
		\pinlabel {$\gamma_4^-$} [B] at 29 78
		\pinlabel {$\gamma_2^-$} [B] at 114 40
		\pinlabel {$S_0$} [ ] at 15 0
		\pinlabel {$-S_0$} [ ] at 308 0
		\pinlabel {$\beta_4^-$} [B] at 297 78
		\pinlabel {$\beta_2^-$} [B] at 208 40
		\endlabellist
		\begin{center}
			\includegraphics[scale=0.61]{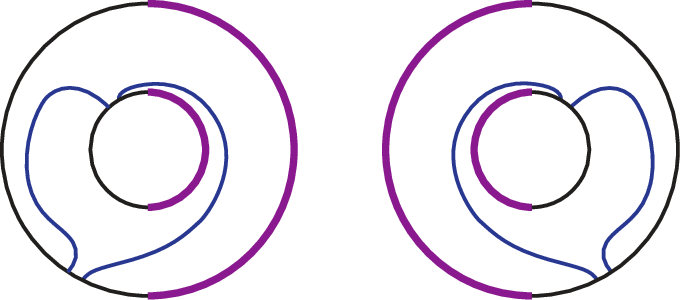}
		\end{center}
			\caption{The sorting arcs $\gamma_i^-$ and their images $\beta_i^- = -h(\gamma_{i}^-)$ on $-S_0$.}\label{fig:HD2_betaminus}
	\end{figure}

	Observe that in this example we do not need any cutting arcs $\boldsymbol b$, because $S_\epsilon\setminus \{b_i^+\}_{i=1,3}$ consists of two disks, each with exactly one interval of $A_\epsilon$ on its boundary, see Figure \ref{fig:HD3_R}. Therefore, $\betas^c$ is empty, and we have no $a_i$ arcs, either.	
	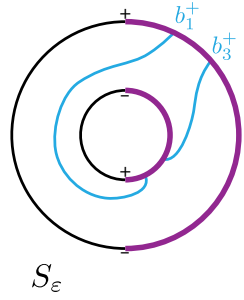
\begin{figure}[h]
		\labellist
		\pinlabel {$e_+$} [b] at 75 142
		\pinlabel {$e_-$} [t] at 75 101
		\pinlabel {$e_+$} [b] at 75 45
		\pinlabel {$e_-$} [t] at 75 3
		\pinlabel {$S_{\epsilon}$} [ ] at 15 0
		\pinlabel {{\color{cyan}$b_1^+$}} [ ] at 108 145
		\pinlabel {{\color{cyan}$b_3^+$}} [ ] at 133 127
		\endlabellist
		\begin{center}
			\includegraphics[scale=0.61]{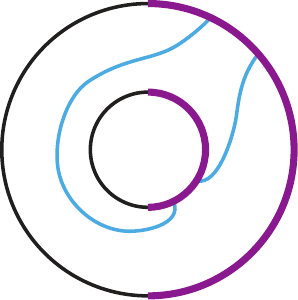}
				\end{center}
			\caption{Cutting $S_{\epsilon}$ along $b_i^+$ yields two disks, each with exactly one interval of $A_\epsilon$ on its boundary. Therefore, we do not need any more cutting arcs.}\label{fig:HD3_R}
	\end{figure}

We next define $\tilde a_i$ by pushing the endpoints of  $b_i^+$ negatively along the boundary so that $\tilde a_i$ and $b_i^+$ intersect once transversely. Figure \ref{fig:HD4_alpha} shows the arcs $\tilde a_i$ on $S_\epsilon$ and their mirror images on $-S_0$. Glued together, they form $\widetilde{\alpha}_i$.  Since this example has no $\alpha_i$ curves, we have $\alphas= \{\widetilde\alpha_i\}_{i\in H_+}$.
	
	\begin{figure}[h]
		\labellist
		\pinlabel {$S_\varepsilon$} [ ] at 12 54
		\pinlabel {$-S_0$} [ ] at 310 54
		\pinlabel {{\color{red}$\alpha_i=\emptyset$}} [l] at 97 34
		\pinlabel {{\color{red}$\tilde\alpha_j=\tilde a_j\cup-\tilde a_j$}} [l] at 96 10
		\pinlabel {{\color{red}$\alpha_c$}} [ ] at 234 20
		\pinlabel {$e_+$} [b] at 74 192
		\pinlabel {$e_-$} [t] at 74 152
		\pinlabel {$e_+$} [b] at 74 95
		\pinlabel {$e_-$} [t] at 74 55
		\pinlabel {{\color{cyan}$b_1^+$}} [ ] at 102 199
		\pinlabel {{\color{red}$\tilde a_1$}} [ ] at 114 192
		\pinlabel {{\color{cyan}$b_3^+$}} [ ] at 129 182
		\pinlabel {{\color{red}$\tilde a_3$}} [ ] at 140 166
		\pinlabel {$B$} [ ] at 162 113
		\pinlabel {{\color{red}$-\tilde a_3$}} [ ] at 184 174
		\pinlabel {{\color{red}$-\tilde a_1$}} [ ] at 207 196
		\pinlabel {$e_+$} [b] at 258 192
		\pinlabel {$e_-$} [t] at 258 152
		\pinlabel {$e_+$} [b] at 258 95
		\pinlabel {$e_-$} [t] at 258 55
		\endlabellist
		\begin{center}
			\includegraphics[scale=0.61]{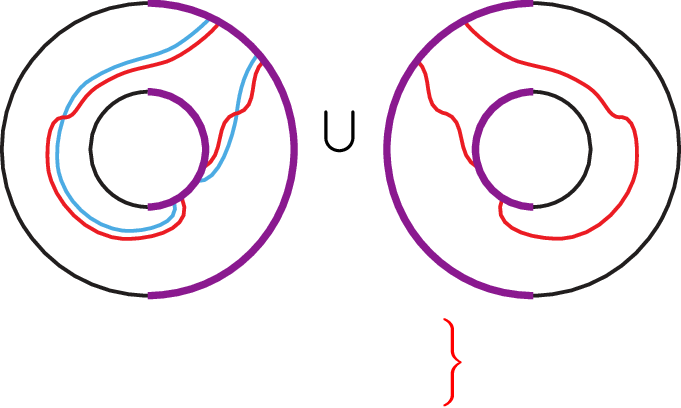}
		\end{center}
			\caption{The $\widetilde{\alpha}_i$ curves obtained from the $b_i^+$ arcs.}\label{fig:HD4_alpha}
	\end{figure}
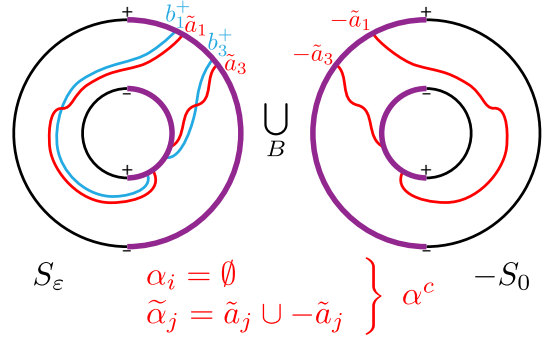

Finally, Figure \ref{fig:HD5_diagram} illustrates the bordered sutured Heegaard diagram adapted to the sorted  foliated open book of Example \ref{ex:sorted}. In this example, the set of intersection points $\xxx$ consists only of $x_i^+ = \tilde a_i\cap b_i^+\in S_{\epsilon}$ for  $i\in H_+$.

\begin{figure}[h]
	\labellist
	\pinlabel {$e_+$} [b] at 75 140
	\pinlabel {$e_-$} [t] at 75 101
	\pinlabel {$e_+$} [b] at 75 43
	\pinlabel {$e_-$} [t] at 75 3
	\pinlabel {$e_+$} [b] at 258 140
	\pinlabel {$e_-$} [t] at 258 101
	\pinlabel {$e_+$} [b] at 258 43
	\pinlabel {$e_-$} [t] at 258 3
	\pinlabel {$S_\varepsilon$} [ ] at 15 0
	\pinlabel {$-S_0$} [ ] at 308 0
	\pinlabel {$B$} [] at 162 61
	\endlabellist
	\begin{center}
		\includegraphics[scale=0.61]{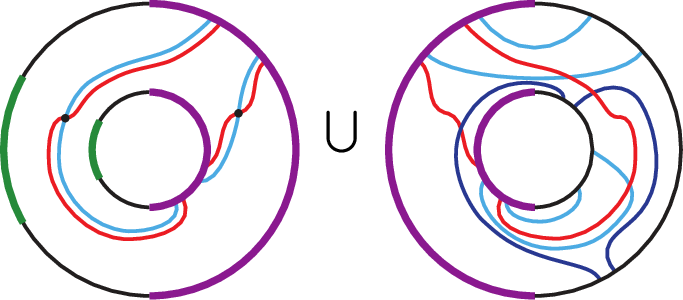}
	\end{center}
	\caption{The bordered sutured Heegaard diagram adapted to the sorted  foliated open book of Example \ref{ex:sorted}. The intersection points $x_i^+ = \tilde a_i\cap b_i^+\in S_{\epsilon}$ for $i\in H_+$ are marked as black dots. The black portion of the resulting boundary is $Z$, while the green is $\bdy\Sigma\setminus Z$.}\label{fig:HD5_diagram}
\end{figure}
\end{example}

We will use $\xxx$ to define two contact invariants in bordered sutured Floer homology. In order to do that, we first show that $\HH$ is an admissible diagram for the bordered sutured manifold associated to $(M,\xi,\mathcal{F})$. 

\begin{proposition}\label{prop:hd-fc}
Let $(\{S_i\},h, \{\gamma_i^\pm\})$ be a  sorted abstract foliated open book supporting a foliated contact three-manifold $(M,\xi,\mathcal{F})$. 
Suppose $\HH = (\Sigma, \alphas,\betas,\zz)$ is a  bordered sutured Heegaard diagram adapted to $(\{S_i\},h, \{\gamma_i^\pm\})$. Then $\HH$ gives an admissible bordered Heegaard diagram for  the bordered sutured manifold $(M,\bsGamma,\zz)$ associated to $(M,\xi,\mathcal{F})$.
\end{proposition}
\begin{proof}

Starting from a sorted foliated open book $(\{S_i\},h, \{\gamma_i^\pm\})$, we may either first construct a foliated contact manifold $(M,\xi,\mathcal{F})$ as in Section~\ref{sec:supp-cs} and an associated bordered sutured manifold $(M,\bsGamma,\zz)$ as in Section~\ref{ssec:suturedbordered}, or we may start by constructing a bordered sutured Heegaard diagram $\HH = (\Sigma, \alphas,\betas,\zz)$ as above. To temporarily distinguish between the two constructions of the set of matched arcs, we will in fact temporarily replace this notation with $\HH = (\Sigma, \alphas,\betas,\zz')$, with $Z'$ the set of arcs in $\zz'$. We may then use the construction of Section~\ref{sssec:bs} to produce a bordered sutured manifold. We wish to verify that the bordered sutured manifold arising this way is $(M, \bsGamma, \zz)$.

Recall that the foliated contact manifold $(M, \xi, \mathcal{F})$ is built via saddle cobordisms specified by the sorting arcs. We intend for the $\alpha$ curves to specify the handlebody from $S_0$ to $S_{\epsilon}$ and for the $\beta$ curves to specify the handlebody from $S_{\epsilon}$ to $S_0$. We will start by modifying the graph $G(\zz')\subset \Sigma$ and the bordered sutured structure on $\bdy M$ via isotopy, allowing us to identify the Heegaard surface $\Sigma$ with the surface $S_\epsilon\cup_B-S_0$ in $M$ in such a way that $\bdy\Sigma\setminus Z'$ is identified with $\bsGamma$. (This is equivalent to recovering $\bsGamma$ as $(\bdy\Sigma\setminus Z') \times \{\frac{1}{2}\}$ in a thickened copy of the surface $\Sigma \times [0,1]$.)  Recall that we chose $Z'$ to be the union of $A_0$ with the arcs $\bigcup_{I\subset A_{\epsilon}}(I_+\cup I_-)$; the purpose of this choice is to enable a simple discussion of gluing in Section \ref{sec:gluing}. However, we may isotope $Z'$ inward along the arcs $I_{\pm}$ to $A_0$.  On $\bdy M$, we may isotope $Z$ (resp.\ $\bsGamma$) through the disks bounded between $Z$ and $A_0$ (resp.\ $A_0$ and $A_\epsilon$) and identify it with $A_0$ (resp.\ $A_\epsilon$). This isotopy can be chosen so that it carries the matched points on $Z$ to the matched points on $Z'$ (as identified with $A_0$). Note that under this isotopy the edges $e_i$ of the graph $G(\zz)$ will twist near their endpoints around the elliptic points.

Next, $\{\widetilde{a}_i\}_{i\in H_+}\cup\{a_1,\cdots,a_{2g_0+n_0+|A_0|-k-2}\}$ is a cut system for $S_{\epsilon}$. At each time $0\leq t\leq \epsilon$, we may consider the image of each of these arcs on $\pi^{-1}(t)$ coming from applying the flow of $\pi$; for each arc, the union of these images is a disk with boundary the corresponding $\alpha$-circle. (We say that these arcs \emph{trace out} disks in time $[0,\epsilon]$ whose boundaries are the $\alpha$-circles.) So, $\alphas$ specifies the handlebody from $S_0$ to $S_{\epsilon}$. Similarly, the cutting arcs $\{b_1,\cdots,b_{2g_0+n_0+|A_0|-k-2}\}$ for $P_\epsilon$ will trace out disks in time $[\epsilon,2\pi]$, whose boundaries under the identification of $S_{2\pi}$ with $S_0$ using monodromy $h$ are the $\beta$-circles on $\Sigma$.

Finally, for each $i\in H_-$, flowing backward from time $t=i\pi/k$ to $t=\epsilon$, the sorting arc $\gamma_i^-$  traces out a disk corresponding to the cocore of the handle addition. The boundary of this disk is the arc $\beta^-_i$ together with the twisted and truncated copy of the separatrix $\delta_i$, that is, with the twisted arc $e_i$. These are the disks attached along the arcs $\beta^-_i$. Likewise, for $i\in H_+$, flowing forward from $t=i\pi/k$ to $t=2\pi$, the sorting arc $\gamma_i^+$ traces out a disk corresponding to the cutting arc for the handle subtraction, whose boundary under the identification of $S_{2\pi}$ with $S_0$ via the monodromy $h$ is the arc $\beta_i^+$ together with the twisted and truncated copy of the separatrix $\delta_i$. These are the disks attached along the arcs $\beta^+_i$.

As for admissibility, we notice that just as in the closed or sutured cases, admissibility is automatic:  every periodic domain with an $\alpha$-circle on its boundary crosses the only intersection of that circle with $\betas$ curves or circles on $S_{\epsilon}$, at which point the sign of the connected component of $\Sigma\setminus (\alphas\cup\betas)$ must change.
\end{proof}

If $\HH$ is a bordered  sutured Heegaard diagram adapted to  a sorted abstract foliated open book $(\{S_i\},h, \{\gamma_i^\pm\})$ and  $(\{S_i\},h, \{\gamma_i^\pm\})$ is compatible with the foliated contact three-manifold $(M,\xi,\mathcal{F})$, we say that $\HH$ is \emph{adapted} to $(M,\xi,\mathcal{F})$.

By Proposition~\ref{prop:hd-fc} and \cite[Section 3.4]{bs-JG}, the diagram ${\overline{\HD}=(\Sigma, \betas,\alphas,\overline{\zz})}$ obtained by exchanging the roles of the two sets of curves and formally replacing the arc diagram $\zz$ of $\beta$-type (which is to say, parametrized by arcs which are part of the second set of curves)  with the identical arc diagram $\overline{\zz}$ of $\alpha$-type (parametrized by arcs which are part of the first set of curves) is a bordered sutured diagram for $(-M,\bsGamma,\overline{\zz})$. Write $\overline{\zz}=(Z, a, m)$. We have the following proposition:

\begin{proposition}\label{prop:xd-def}
The above $\xxx$ gives a well defined generator
\[
\xxx_D := \xxx\in \bsdhat(\overline{\HH})
\]
with  $I_D(\xxx) = I(H_-)$ and $\delta^1(\xxx_D) = 0$, 
and a well defined generator
\[
\xxx_A  := \xxx\in \bsahat(\overline{\HH})
\]
with $I_A(\xxx_A) = I(H_+)$ and $m_{i+1}(\xxx_A, a(\brho_1), \ldots, a(\brho_i))=0$ for all $i\geq 0$ and all sets of Reeb chords $\brho_j$ in $(Z, a)$.
\end{proposition}
\begin{proof}
Since each of the intersections in the definition of $\xxx$ is unique,  the two generators are well defined. The idempotents are as stated, since the $\beta$-arcs occupied by $\xxx$ are indexed by $H_+$, whereas the unoccupied $\beta$-arcs are indexed by $H_-$. Last, consider  regions of $\Sigma\setminus(\alphas\cup\betas)$ that have a point $x\in \xxx$ as a corner vertex.  Each such region either has an interval of $\bdy\Sigma\setminus Z$ on its boundary or it has some corner near which it  agrees with one of the thin strips obtained by perturbing the $b_i$ and $b_j^+$ arcs to $a_i$ and $\tilde a_j$ arcs. The positively oriented boundary of such a thin strip near $x$ has a segment of a $\beta$ curve ending at $x$, and a segment of an $\alpha$ curve starting at $x$.  Such a region cannot be in the projection $\pi_{\Sigma}\circ u$ of an index $1$ holomorphic map  
\[u:(S, \bdy S) \to  (\Int(\Sigma)\times \mathbb D \times \R, (\betas\times \{1\}\times \R)\cup (\alphas\times \{0\}\times \R))\]
satisfying \cite[Section 5.2, Conditions (1)-(11)]{bs-JG} and asymptotic to $\xxx\times [0,1]$ at $-\infty$.
Thus, $\delta^1(\xxx_D) = 0$, and $m_{i+1}(\xxx_A, a_1, \ldots, a_i)=0$ for all $i\geq 0$ and all $a_j\in \mathcal A(\bdy \overline{\HH})$.
\end{proof}


\section{Invariance}
\label{sec:invariance}

In this section we complete the proof of Theorem \ref{thm:ca-cd}. Specifically, we show that given two different sets of choices made in constructing a bordered sutured Heegaard diagram for the bordered sutured three-manifold ${(-M,\bsGamma, \overline{\zz})}$ associated to a foliated contact three-manifold $(M,\xi, \mathcal{F})$ resulting in Heegaard diagrams $\HD$ and $\HD'$ with associated type $A$ bordered sutured Floer homologies $\bsahat(\HD)$ and $\bsahat(\HD')$, the contact elements $\x_A$ and $\x'_A$ are equivalent in the sense of Section~\ref{sssec:bs}, and likewise for $\x_D$ and $\x_D'$. We prove equivalence of these classes under

\begin{enumerate}
\item isotopy  of the monodromy $h$ of the foliated open book $(\{S_i\},h, \{\gamma_i^\pm\})$ for $(M,\xi, \mathcal{F})$;
\item  the choice of complex structure;
\item stabilizations of the foliated open book $(\{S_i\},h, \{\gamma_i^\pm\})$ on the $S_0$ page; and
\item the choice of the cutting arcs made when constructing the bordered Heegaard diagram in Section~\ref{ssec:bordereddiagram}.
\end{enumerate}

Note that we do not have to prove invariance under time shift, since the definition of a triple $(M,\xi, \mathcal{F})$ has an implicit choice of initial time $t=0$; see Remark~\ref{rmk:t}. 
Invariance under the choice of complex structure follows from \cite[Theorem 7.8]{bs} and the observation that since there are no nontrivial topological domains beginning at the generator ${\bf x}$ on the bordered sutured Heegaard diagram $\HD$, the type $A$ and type $D$ chain homotopy equivalences between bordered sutured complexes associated to pairs $(\HD, J_1)$ and $(\HD, J_2)$ count a single pseudoholomorphic strip constant in the $\Sigma$ coordinate beginning at $\x_D$ or $\x_A$, and thus in each case carry the special generator to itself. It is a consequence of Theorem~\ref{thm:revsortedsupport} that stabilizations on the $S_0$ page are sufficient.

We thus focus on the remaining three choices made in the construction of the contact invariants. For invariance under the isotopy class of $h$ and choice of the cutting arcs our proofs heavily parallel those of \cite[Section 3]{HKM09_HF} and \cite[Section 3]{hkm09}; our proof of stabilization invariance has a slightly different structure.

\begin{proposition}\label{prop:homot} Suppose we choose two representatives for the isotopy class of $h$, resulting in bordered sutured Heegaard diagrams $\HD$ and $\HD'$. The resulting contact elements $\x_A \in \bsahat(\HD)$ and $\x'_A \in \bsahat(\HD')$ are equivalent in the sense of Section~\ref{sssec:bs}, and likewise for $\x_D$ and $\x_D'$.
\end{proposition}

\begin{proof}
The proof given in \cite[Lemma 3.3]{HKM09_HF} carries over to the bordered sutured case essentially verbatim. As previously, the point is that there are no holomorphic curves beginning at the generators $\x_A$ and $\x_D$.
\end{proof}

We now turn our attention to equivalence under choice of the cutting arcs in the case that we start with a sufficiently stabilized foliated open book. (Subsequently, in Proposition \ref{prop:stab-invariance}, we will prove stabilization invariance, thus showing invariance of the choice of cutting arcs for all sorted foliated open books.)

First we show that two different choices of cutting arcs  ${\boldsymbol b}$ and $\widetilde{{\boldsymbol b}}$ for $P_{\epsilon}\subset S_{\epsilon}$ are related by arcslides on $P_{\epsilon}$.

\begin{proposition}\label{prop:invt-cut} 
Suppose that $(\{S_i\},h,\{\gamma_i\})$ is a sufficiently stabilized foliated open book. Then
any two sets of cutting arcs ${\boldsymbol b} = \{b_1, \ldots, b_{2g_0+n_0+|A_0|-k-2}\}$ and 
$\widetilde{{\boldsymbol b}} = \{\widetilde{b}_1, \ldots,  \widetilde{b}_{2g_0+n_0+|A_0|-k-2}\}$  for $P_{\epsilon}$ are related by a sequence of arcslides. 
\end{proposition}

\begin{proof}
The proof is modeled on \cite[Lemma 3.3]{hkm09} and \cite[Lemma 3.6]{HKM09_HF}. Let $r=2g_0+n_0+|A_0|-k-2$.

First, we will show that $\cup_{i=1}^r b_i$ and $\cup_{i=1}^r \widetilde{b}_i$ can be made disjoint from each other by performing a sequence of arcslides. Then we verify the statement of the proposition for the case of two disjoint sets of cutting arcs. Suppose that $\left(\cup_{i=1}^r b_i\right)\cap \left(\cup_{i=1}^r \widetilde{b}_i\right)\neq \emptyset$. We may assume the curves intersect efficiently. We can decrease the number of intersections as follows.

Cut $P_{\epsilon}$ along the arcs in  ${\boldsymbol b}$ to obtain a set of polygons. For each cutting arc $b_i$, label the two resulting boundary segments on $\bdy(P_{\epsilon}\setminus {\boldsymbol b})$ by  $b_i$ and $b_i^{-1}$, and call these two segments \emph{partners}.

First we show that every connected component of $P_\epsilon\setminus\boldsymbol{b}$ intersects $R_\epsilon$ in exactly one interval. To see that each connected component of $P_\epsilon\setminus\boldsymbol{b}$ intersects $R_\epsilon$, observe that otherwise the boundary of the component would only contain arcs in $B$ and in $\boldsymbol{b}$. But the connected components of $P_\epsilon\setminus\boldsymbol{b}$ are naturally identified with the connected components of $S_{\epsilon}\setminus ({\boldsymbol b}\cup \betas^a)$ and by the definition of cutting arcs, each connected component of $S_{\epsilon}\setminus ({\boldsymbol b}\cup \betas^a)$ is a disk with exactly one interval of $A_{\epsilon}$ on its boundary. Moreover, $P_\epsilon\setminus\boldsymbol{b}$ cannot intersect $R_\epsilon$ in more intervals. Assume it intersects $R_\epsilon$ in at least two intervals. There are two types of such a segment of $R_\epsilon$: either it corresponds only to an arc in  $A_\epsilon$ or also to an arc of type $\gamma^\pm_i$ whose both endpoints connect to arcs in $A_\epsilon$. In this latter case, the isotopy of $\gamma^\pm_i$ resulting in $\betas^\pm_i$ leaves at least one of these arcs in $A_\epsilon$ on the boundary of the connected component of  $S_{\epsilon}\setminus ({\boldsymbol b}\cup \betas^a)$. In the former case, the only arc in $A_\epsilon$ is also on the boundary of the corresponding connected component of $S_{\epsilon}\setminus ({\boldsymbol b}\cup \betas^a)$.
In both cases, the segment of $R_\epsilon$ corresponds to an arc in $A_\epsilon$ on $\partial S_{\epsilon}\setminus ({\boldsymbol b}\cup \betas^a)$. Therefore, our assumption contradicts the fact that $S_{\epsilon}\setminus ({\boldsymbol b}\cup \betas^a)$ contains exactly one interval of $A_{\epsilon}$ on its boundary.

Consider a connected component $C$ of $P_{\epsilon}\setminus {\boldsymbol b}$ containing an intersection point of $\cup_{i=1}^r b_i$ and $\cup_{i=1}^r \widetilde{b}_i$.
Then $C$ is a polygon with some number of edges in ${\boldsymbol b}$ (of  types $b_j$ and $b_j^{-1}$), some number of edges in $B$, and exactly one edge in $R_{\epsilon}$. 

After possibly reordering the elements of ${\boldsymbol b}$, we may assume that $b_1 \cap \widetilde{b}_1\neq \emptyset$ and that we can find a subarc $\widetilde{b}_1^0\subset \widetilde{b}_1$ in $C$ such that $\widetilde{b}_1^0$ starts from a segment $B_1$ of the binding $B$, ends on $b_1$, and has interior disjoint from $\cup_{i=1}^r b_i$, as in Figure \ref{fig:aslicedisj}. We may also assume that $b_1$ and $B_1$ are not adjacent, as we could otherwise isotope $\widetilde{b}_1^0$ to decrease the number of intersection points between $\widetilde{b}_1^0$ and $b_1$. 
\begin{figure}[h]
\begin{center}
\labellist
  	\pinlabel $\textcolor{red}{b_1}$ at 4 40
	\pinlabel $\textcolor{red}{\widetilde{b}_1}$ at 12 8
	\pinlabel $\textcolor{red}{\widetilde{b}_1^0}$ at 81 39
	\pinlabel $C$ at 50 70
	\pinlabel $\textcolor{violet}{B_1}$ at 102 100
\endlabellist
\includegraphics[scale=0.9]{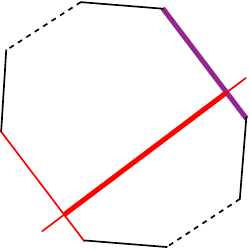}
\end{center}
\caption{The polygon $C$, with the subarc $\widetilde{b}^0_1$ crossing between the nonadjacent arcs $b_1$ and $B_1$.}
\label{fig:aslicedisj}
\end{figure}
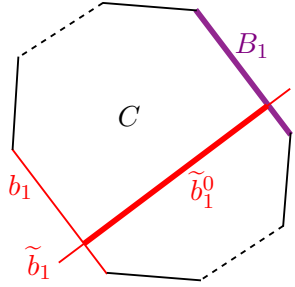

Now $\widetilde{b}_1^0$ cuts $C$ into two polygons $C_1$ and $C_2$, at least one of which does not contain $b_1^{-1}$. Suppose first that  $C_1$ does not contain $b_1^{-1}$ and also that $\partial C_1\cap \partial R_{\epsilon}=\emptyset$. In this case, we may slide $b_1$ over the remaining arcs in $\partial C_1$ until we obtain $b^\prime_1$. Note that $b'_1$ has a subarc which lies on $b_1$, while the rest of $b'_1$ lies in a small neighborhood of $\widetilde{b}_1$. This means that no new intersection points are created, while we cancel the one we started with. Therefore, $b^\prime_1\cup b_2\cup ...\cup b_r$ has fewer intersection points with $\cup_{i=1}^r \widetilde{b}_i$. We proceed similarly in the case when  $\partial C_1\cap \partial R_{\epsilon}\neq\emptyset$ but $b_1^{-1}$ is not contained in $\partial C_2$, because then we may slide $b_1$ over the remaining arcs in $\partial C_2$ to get $b^\prime_1$. The case when $\partial C_1$ intersects $\partial R_{\epsilon}$ in an arc $\gamma$ and $C_2$ contains $b_1^{-1}$ on its boundary is handled as follows. Since the foliated open book is sufficiently stabilized, any connected component of $P_{\epsilon}$ intersects $\partial R_{\epsilon}$ in at least two intervals. Therefore, there is an arc $c\subset \partial C$ which is of type $b_i$ or $b_i^{-1}$ such that $c^{-1}$ is not in $\partial C$; otherwise, $C$ would glue up to a connected component of $P_{\epsilon}$ intersecting $\partial R_{\epsilon}$ in only one interval.
Sliding $c$ over all other arcs in $\partial C$ except $\gamma$ gives us  $c'$ parallel to $\gamma$. After this, we may slide $b_1$ over the arcs in $\partial C_1$, since now we only need to slide it over $c'$ instead of $\gamma$. Finally, we get $b^\prime_1$ which has a subarc lying on $b_1$, while the rest lies in a small neighborhood of $\widetilde{b}_1$ as before. This means that no new intersection points are created, while we cancel the one we started with. Therefore, $\{b^\prime_1,b_2,...,b_r\}$ has fewer intersection points with $\cup_{i=1}^r \widetilde{b}_i$ than $\cup_{i=1}^r b_i$ did. Iterating this procedure a finite number of times, we obtain a new set of cutting arcs which are disjoint from the arcs $\widetilde{\boldsymbol b}$.

Next, we show that for $\cup_{i=1}^r b_i$ and $\cup_{i=1}^r \widetilde{b}_i$ disjoint, ${\boldsymbol b}$ can be turned into $\widetilde{{\boldsymbol b}}$ via a finite sequence of arcslides. Let $C$ denote a connected component of $P_{\epsilon}\setminus {\boldsymbol b}$ as before. If each arc $\widetilde{b}_i$ in $C$ is parallel to some $b_j$ or $b_j^{-1}$, we are done with $C$. Suppose that there is an arc, say $\widetilde{b}_1$, which is not parallel to any $b_j$ or $b_j^{-1}$. Then $\widetilde{b}_1$ cuts $C$ into two pieces $C_1$ and $C_2$, with the property that each $\partial C_i$ contains more than one arc of type $b_j$ or $b_j^{-1}$. Choose the labeling such that $C \cap R_{\epsilon}$ is in $C_1$. Moreover, we can find a pair of arcs $b_j$ and $b_j^{-1}$ such that $b_j \subset \partial C_2$ and $b_j^{-1}\not\subset \partial C_2$ (or vice versa);  otherwise every segment of $\partial C_2$ has its partner in $\partial C_2$, implying that some part of $C_2$ glues up to a connected component of $P_{\epsilon}\setminus \boldsymbol{b}$ which does not intersect $R_{\epsilon}$. If each such $b_j$ were parallel to an arc $\widetilde{b}_j$, then  $\cup_{i=1}^r \widetilde{b}_i$ would again cut off a subsurface which does not intersect $R_{\epsilon}$. Therefore we may assume that there exists a $b_j$ which is not parallel to any $\widetilde{b}_l$. Now we can slide $b_j$ over all the arcs $b_l$ and $b_l^{-1}$ in $C_1$ until it becomes parallel to $\widetilde{b}_1$. This way we associated an arc which was not parallel to any $\widetilde{b}_l$ before to $\widetilde{b}_1$ not parallel to any $b_j$ or $b_j^{-1}$ before. Repeating this, we can find a parallel pair for each $\widetilde{b}_i$ in $C$.
\end{proof}

\begin{proposition}\label{prop:invt-slide} 
Let $\boldsymbol{b} = \{b_1,b_2, \dots, b_r\} \rightarrow  \{b_1 + b_2,b_2 \dots, b_r \} = \boldsymbol{b'}$ be an arcslide among the cutting arcs on $P_{\epsilon}$. If $\HD$ and $\HD'$ are the resulting bordered sutured Heegaard diagams, then the resulting contact elements $\x_A \in \bsahat(\HD)$ and $\x_A' \in \bsahat(\HD')$ are equivalent in the sense of Section~\ref{sssec:bs}, and likewise for $\x_D$ and $\x_D'$. \end{proposition}

\begin{proof} 
An arcslide induces a pair of handleslides, one among the $\alpha$-circles followed by one among their pushoffs, i.e. the $\beta$-circles. The induced maps on $\bsdhat$ and on $\bsahat$  count holomorphic triangles that do not interact with the boundary of the Heegaard diagram. Thus, the proof is a straightforward adaptation of \cite[Lemma 3.5]{HKM09_HF}.  
 \end{proof}

We conclude the following.

\begin{proposition} \label{prop:cutting} Let ${\boldsymbol b} = \{b_1, \dots, b_r\}$ and $\widetilde{\boldsymbol b} = \{\widetilde{b}_1, \dots, \widetilde{b}_r\}$ be two possible choices of cutting arcs made in the  construction of the bordered Heegaard diagram associated to the sufficiently stabilized open book $(\{S_i\},h,\{\gamma_i\})$, and let $\HD$ and $\widetilde{\HD}$ be the resulting bordered sutured Heegaard diagrams. Then the resulting contact elements $\x_A \in \bsahat(\HD)$ and $\x_A' \in \bsahat(\HD')$ are equivalent in the sense of Section~\ref{sssec:bs}, and likewise for $\x_D$ and $\x_D'$. \end{proposition}
\begin{proof}
This follows immediately from Propositions~\ref{prop:invt-cut} and \ref{prop:invt-slide}.
\end{proof}

\begin{proposition} \label{prop:stab-invariance} Let $(\{S_i'\},h',\{{\gamma_i'}^\pm\})$ be the open book resulting from stabilization of the open book $(\{S_i\},h,\{\gamma_i^\pm\})$ on the $S_0$ page, and let $\HD$ and $\HD'$ be the resulting bordered sutured Heegaard diagrams. The resulting contact elements $\x_A \in \bsahat(\HD)$ and $\x_A' \in \bsahat(\HD')$ are equivalent in the sense of Section~\ref{sssec:bs}, and likewise for type $D$.
\end{proposition}

\begin{proof}
Let $\HH = (\Sigma, {\alphas}, {\betas},\zz)$ be a bordered sutured Heegaard diagram adapted to $(\{S_i\},h, \{\gamma_i^\pm\})$. We take  $(\{S_i'\},h', \{{\gamma_i'}^\pm\})$ to be the result of stabilizing $(\{S_i\},h, \{\gamma_i^\pm\})$ along a curve $\gamma\subset S_0$. Recall that  if we denote by $\overline{h}$ the extension of $h$ to the added 1-handle by the identity, then $h'=\tau \circ \overline{h}$, where $\tau$ is a positive Dehn twist along the circle $c$ formed by $\gamma$ and the core of the attached 1-handle. This means that  we can obtain from $\HH$ a diagram $\HH' =  (\Sigma', {\alphas'}, {\betas'},\zz)$ adapted to $(\{S_i'\},h', \{{\gamma_i'}^\pm\})$ as follows. Add a 1-handle to the two pages $S_{\epsilon}$ and $-S_0$ that form the Heegaard surface in order to obtain $\Sigma' = S_{\epsilon}'\cup -S_0'$. Let $b'$ be the cocore of the 1-handle on $S_{\epsilon}'$, and let $a'$ be a perturbation of the new cutting arc $b'$ to the left. Let 
   \[\beta'=b'\cup -h'\circ \iota(b')\subset S_{\epsilon}'\cup_{B}-S_{0}',\]
   and let
   \[ \alpha' = a'\cup -a' \subset S_{\epsilon}'\cup_{B}-S_{0}'.\]
 Last, modify all $\beta$ curves that intersect $\gamma$  by applying to them a positive Dehn twist along $c$, see Figure~\ref{fig:hd-stab-1}(b).
 Clearly $\HH'$ is adapted to the stabilized open book $(\{S_i'\},h', \{{\gamma_i'}^\pm\})$. 

Figures~\ref{fig:hd-stab-1}--\ref{fig:hd-stab-3}  demonstrate a sequence of Heegaard moves that transforms $\HH'$ into a diagram $\HH''$ adapted to a foliated open book $(\{S_i\},h'', \{\gamma_i^\pm\})$ whose monodromy $h''$ is isotopic to $h$. 

In particular, in Figure \ref{fig:hd-stab-1}(a) we see the original neighborhood of $\HH'$ where we intend to stabilize, with the stabilization arc drawn in (dashed) grey, and in Figure \ref{fig:hd-stab-1}(b) we see the result of the stabilization, with new curves $\alpha'$ and $\beta'$. Observe that each of $\alpha'$ and $\beta'$ may intersect some number of $\alpha$  and $\beta$ curves on $S_0$. We handleslide every $\beta$ curve other than $\beta'$ that intersects $\alpha'$ over $\beta'$ to obtain Figure \ref{fig:hd-stab-2}(a), in which we have eliminated all intersection points between $\alpha'$ and any $\beta$ curve other than $\beta'$. After a small isotopy of the diagram to Figure \ref{fig:hd-stab-2}(b), we then handleslide all $\alpha$ circles other than $\alpha'$ which intersect $\beta'$ over the curve $\alpha'$. The result of these handleslides, shown in Figure \ref{fig:hd-stab-2}(c), is a diagram in which $\alpha'$ and $\beta'$ intersect once and there is a punctured torus neighborhood of $\alpha'\cup\beta'$ which intersects no other $\alpha$ or $\beta$ curves. We excise this punctured torus, obtaining Figure \ref{fig:hd-stab-3}(a), after which we may isotope the surface to obtain Figure \ref{fig:hd-stab-3}(b); trading two bigons between $S_0$ and $S_{\epsilon}$ produces the diagram $\HH''$ shown in Figure \ref{fig:hd-stab-3}(c), whose monodromy $h''$ is isotopic to $h'$.

We will show that there is a type $D$ (resp.~type $A$) homotopy equivalence induced by the sequence of moves which carries the contact class from $\HH'$ to the contact class from $\HH''$. 
\\

\begin{figure}[h] 
\begin{center}
\labellist
	\pinlabel $\rm{(a)}$ at 30 -17
	\pinlabel $\rm{(b)}$ at 230 -17
\endlabellist
\includegraphics[scale=0.7]{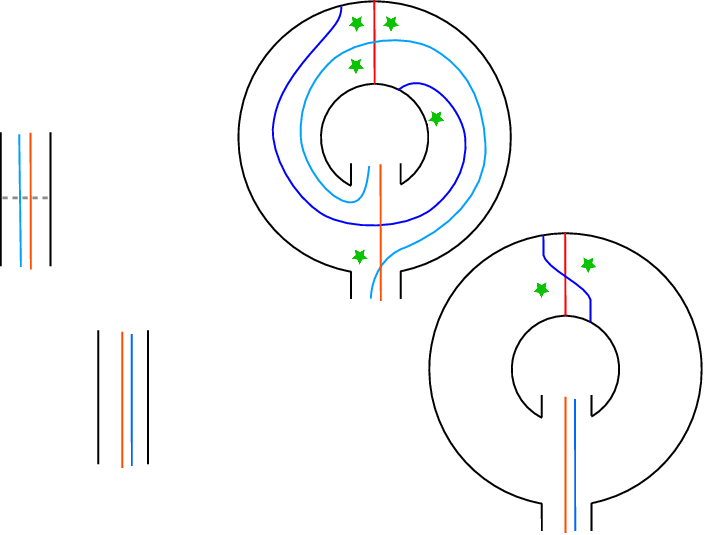}
\vspace{.7cm}
\caption{(a) The two neighborhoods on $S_{\epsilon}$ and $S_0$ of the stabilization arc, viewed on the Heegaard diagram $\HD$.  The stabilization arc is drawn as a grey dashed arc on $S_0$. The red and blue vertical lines on each neighborhood represent a (possibly) more general sequence of $\alpha$ and $\beta$ curves. (b) The Heegaard diagram $\HD'$ corresponding to the stabilized open book, where the same choice of cutting arcs is used away from the stabilization region and the cocore of the stabilizing $1$-handle is used as the final cutting arc; the diagram  $\HD'$ has one new $\alpha$ circle denoted $\alpha'$ and one new $\beta$ circle, denoted $\beta'$. 
In this and all subsequent figure in this section, black lines are identified via translation in the plane, being part of the binding. Green stars mark regions that intersect $\bdy \Sigma\setminus Z$; to see why the marked regions indeed intersect the  $\partial \Sigma\setminus Z$ part of $\partial \Sigma$ nontrivially, simply note that the stabilization arc has endpoints on $\partial \Sigma \setminus Z$.}
\label{fig:hd-stab-1}
\end{center}
\end{figure} 
\vspace{0.2cm}

\begin{figure}[h] 
\begin{center}
\labellist
	\pinlabel $\rm{(a)}$ at 120 -17
	\pinlabel $\rm{(b)}$ at 340 -17
	\pinlabel $\rm{(c)}$ at 580 -17
\endlabellist
\includegraphics[scale=0.7]{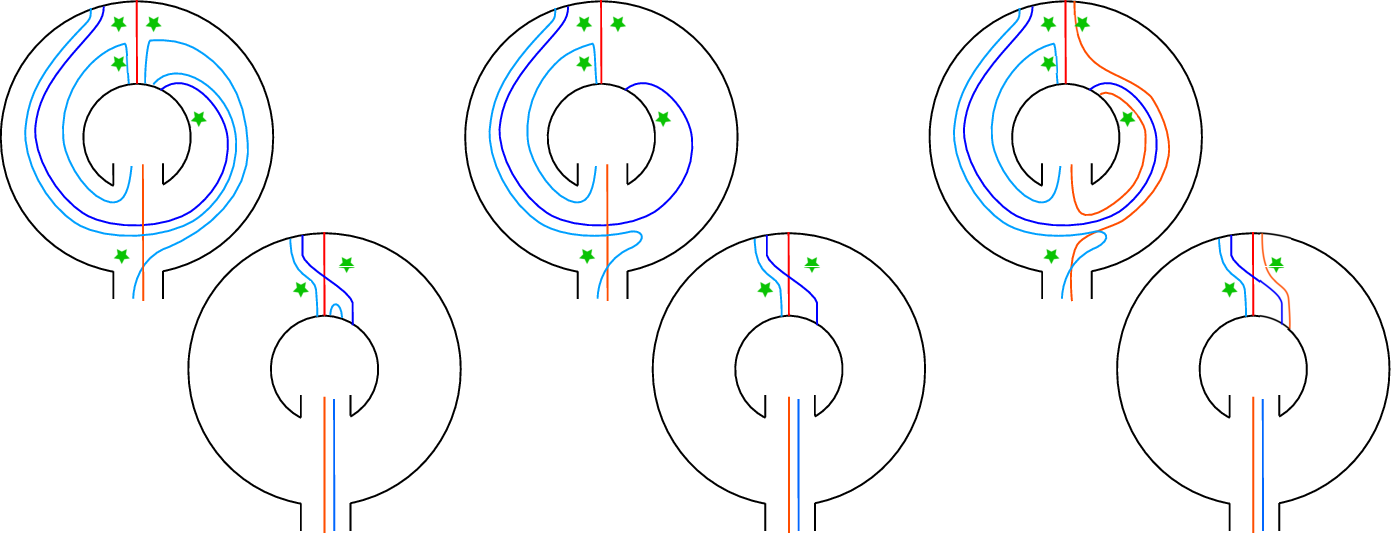}
\vspace{.7cm}
\caption{(a) The diagram after sliding all $\beta$ curves that intersect $\alpha'$ over $\beta'$  to eliminate the intersection points. (b) An isotopy of the previous figure. (c) The result of sliding all $\alpha$ circles that intersect $\beta'$ over $\alpha'$, so that the new neighborhood of $\alpha'\cup \beta'$ is a punctured torus with a single intersection point of $\alphas'\cap \betas'$.}
\label{fig:hd-stab-2}
\end{center}
\end{figure}

\begin{figure}[h] 
\begin{center}
\labellist
	\pinlabel $\rm{(a)}$ at 110 -17
	\pinlabel $\rm{(b)}$ at 310 -17
	\pinlabel $\rm{(c)}$ at 470 -17
\endlabellist
\includegraphics[scale=0.7]{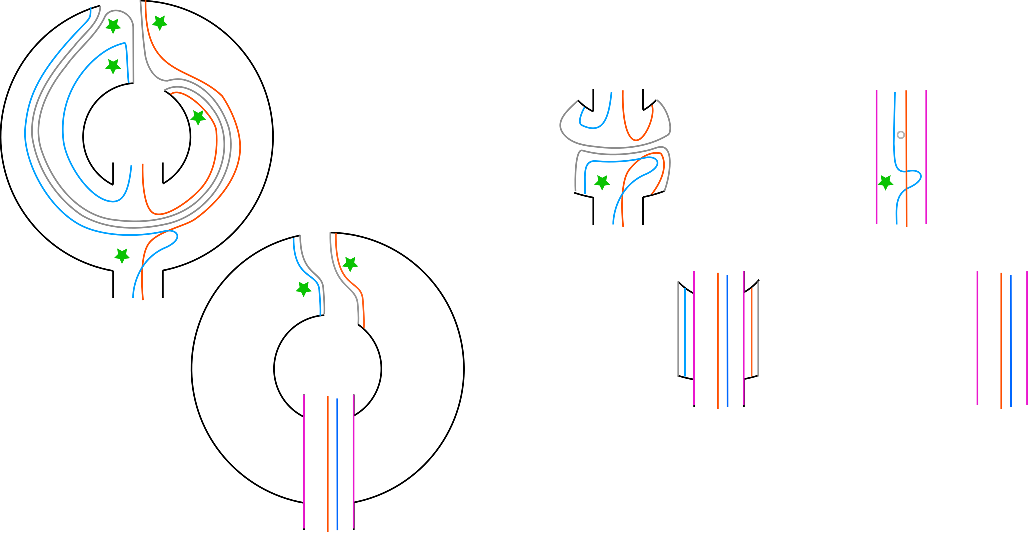}
\vspace{.7cm}
\caption{(a) The result of excising the punctured torus, along with an isotopy of the curve that represents the binding ``trading" two bigons from $S_{\epsilon}$ to $S_0$; the grey lines are the boundary of the punctured torus; the new binding is marked in purple. (b) An isotopy of the previous figure, again with the gray lines as the boundary of a puncture. (c) An isotopy of the previous figure, divided by page according to the new binding.}
\label{fig:hd-stab-3}
\end{center}
\end{figure}

Let $\x_D\in \bsdhat(\overline{\HH})$ be the generator corresponding to the set of intersection points $\x$, as defined in Section~\ref{ssec:bordereddiagram}. Let $x'$ be the unique intersection point in $\alpha'\cap \beta'$ on $\HH'$. Observe that the diagram $\HH'$ was obtained by modifying $\HH$ away from $\x$, and that $\x'\coloneqq \x\cup \{x'\}$ is the special generator for $\HH'$ coming from the construction in Section~\ref{ssec:bordereddiagram}. Let $\x_D'$ be the generator in  $\bsdhat(\overline{\HH'})$ corresponding to $\x'$. Last, let $\x''$ be the special generator on $\HH''$ and $\x''_D$ be the corresponding generator in $\bsdhat(\overline{\HH''})$.

\begin{figure}[h]
\begin{center}
\labellist
  	\pinlabel $\textcolor{black}{\scriptstyle 13}$ at 940 350
	\pinlabel $\textcolor{black}{\scriptstyle 14}$ at 980 350
	\pinlabel $\textcolor{black}{\scriptstyle 15}$ at 940 290
	\pinlabel $\textcolor{black}{\scriptstyle 16}$ at 980 290
  	\pinlabel $\textcolor{black}{\scriptstyle 1}$ at 830 368
	\pinlabel $\textcolor{black}{\scriptstyle 2}$ at 870 368
  	\pinlabel $\textcolor{black}{\scriptstyle 3}$ at 790 327
	\pinlabel $\textcolor{black}{\scriptstyle 4}$ at 830 327
	\pinlabel $\textcolor{black}{\scriptstyle 5}$ at 790 290
	\pinlabel $\textcolor{black}{\scriptstyle 18}$ at 830 290
	\pinlabel $\textcolor{black}{\scriptstyle 6}$ at 755 180
         \pinlabel $\textcolor{black}{\scriptstyle d_0}$ at 888 204
         \pinlabel $\textcolor{black}{\scriptstyle d_1}$ at 910 204
         \pinlabel $\textcolor{black}{\scriptstyle d_2}$ at 934 204
         \pinlabel $\textcolor{black}{\scriptstyle d_3}$ at 958 204
         \pinlabel $\textcolor{black}{\scriptstyle d_4}$ at 985 204
         \pinlabel $\textcolor{black}{\scriptstyle 7}$ at 785 178
          \pinlabel $\textcolor{black}{\scriptstyle 11}$ at 784 50
           \pinlabel $\textcolor{black}{\scriptstyle 17}$ at 753 50
         \pinlabel $\textcolor{black}{\scriptstyle 8}$ at 1022 153
         \pinlabel $\textcolor{black}{\scriptstyle 9}$ at 783 128         
         \pinlabel $\textcolor{black}{\scriptstyle 10}$ at 782 99
         \pinlabel $\textcolor{black}{\scriptstyle 12}$ at 1022 70
\endlabellist
\includegraphics[scale=0.455]{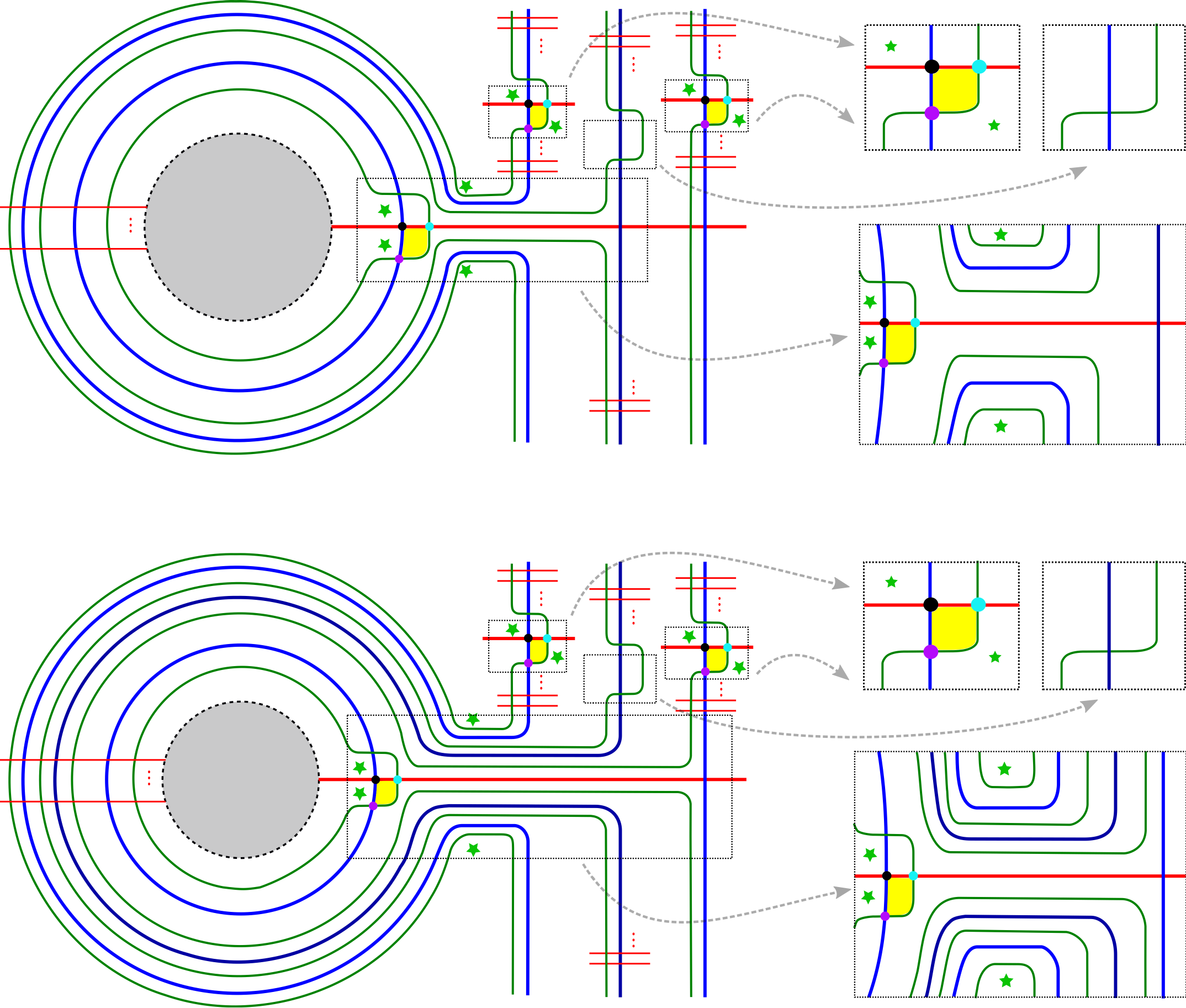}
\end{center}

\caption{Top: The case of sliding a curve that does not contain an intersection point in $\x$, i.e.\ some $\beta^-_j$. Bottom: The case of sliding a curve that contains an intersection point in $\x$.}
\label{fig:triple-seq}
\end{figure}

 Let $\x_0 = \x_D'$, and let $\betas_0 = \betas'$. Consider the bordered Heegaard triple $\mathcal H_i=(\Sigma,\betas_{i-1}, \betas_{i},\alphas', \zz)$ corresponding to the $i$th handleslide in the sequence of handleslides making $\alpha'$ disjoint from every $\beta$ curve except $\beta'$. Figure \ref{fig:triple-seq} illustrates the diagram in a neighborhood of the union of $\alpha'$, $\beta'$, and the $\beta$-circles that intersect $\alpha'$. Let $F_i$ be the map associated to the $i$th handleslide defined by counting holomorphic triangles with Maslov index zero. Then, the chain map associated with the handleslides is 
\[f_i:=F_i(\cdot,\Theta):\bsdhat(\Sigma,\betas_{i-1},\alphas', \zz)\to \bsdhat(\Sigma,\betas_{i},\alphas', \zz)\]
where $\Theta$ represents the top generator in $(\Sigma,\betas_{i-1},\betas_{i}, \zz)$. 
We will argue that $f_1$ carries $\x_0$ to a single generator $\x_1$ in $ \bsdhat(\Sigma,\betas_1,\alphas', \zz)$, and inductively, that each map $f_i$ carries the generator $\x_{i-1}$ to a generator $\x_i$.

Assume $\Delta\in\pi_2(\x_{i-1},\Theta,\y)$ is a Whitney triangle that contributes nontrivially to $f_i(\x_{i-1})$. Let $n_j$ denote the coefficient of $\Delta$ in the region labelled with $j$, and recall that $n_j \geq 0$ for all $j$.

First, consider a neighborhood of an intersection point in $\x_{i-1}$ on one of the $\beta$ curves other than $\beta'$. This situation is illustrated in the top left rectangle in the bottom image of Figure \ref{fig:triple-seq}. Because the black and pink intersection points must both appear as corners in the projection of $\Delta$ to $\mathcal H_i$, we have
\[n_1 + n_3 + 1 = n_4 \quad\quad n_4+n_5=n_3+1,\]
so taking the sum of the two equations and cancelling terms, we see that $n_1+n_5=0$, implying that $n_1=n_5=0$. Thus we see $n_4=n_3+1$, so in particular $n_4\geq 1$. But the multiplicity of the intersection point marked in turquoise in $\y$ is $n_4+n_2- n_1 = n_4+n_2 \leq 1$, so $n_4=1$ and $n_2=0$.  

For the central rectangle, we first observe that $n_7 = n_{11}$ and $n_6=n_{17}$; otherwise, some corner of the projection of $\Delta$ to $\mathcal H_i$ lies on the intersection of $\beta'$ with an $\alpha$ curve other than $\alpha'$, which cannot occur since $\x_{i-1}$ contains an intersection point between $\beta'$ and $\alpha'$. Next, $n_{d_0}=0$ since the region marked with $d_0$ intersects $\bdy\Sigma\setminus Z$ nontrivially. We claim that if $n_{d_{j-1}}=0$ then $n_{d_j}=0$. For consider the green or blue curve separating $d_{j-1}$ or $d_j$. If it is part of a pair of green and blue curves containing an intersection point of a generator that moves under the cobordism map then either $n_5=0$ or $n_{18}=0$ (because the region labelled $18$  intersects $\bdy\Sigma\setminus Z$ nontrivially) can be used to deduce that if $n_{d_{j-1}}=0$ then so is $n_{d_j}$. If, however, the curve separating $d_{j-1}$ or $d_j$ is part of a pair of green and blue curves without an intersection point of a generator that moves under the cobordism map, then the fact that $n_{13}=n_{14}=n_{15}=n_{16}$ shows that if $n_{d_{j-1}}=0$ then $n_{d_j}=0$. So all $n_{d_j}=0$, and by the same logic $n_7=n_{11}=0$.

Now we see that 
\[n_{10} = n_9 +1 \qquad \qquad n_{10}+n_6=1 \qquad \qquad n_9+n_6=0,\]
 implying that $n_6=n_9=0$ and $n_{10}=1$. Hence the domain of $\Delta$ is the disjoint union of the small yellow triangles, and $\y$ is the intersection point $\x_i$ shown in turquoise in Figure \ref{fig:triple-seq}. Consequently, for an appropriate choice of the complex structure $f_i(\x_{i-1})=\x_i$. So the map induced by the $i$th handleslide among the $\beta$ curves takes each intersection point between $\betas_{i-1}$ and $\alphas$ in $\x_{i-1}$ to the nearest intersection point between $\betas_{i}$ and $\alphas'$. 
 
The argument for handleslides among the $\alpha$ curves is analogous. Let $\z$ be the generator that $\x'$ is carried to by the sequence of $\beta$ and $\alpha$ handleslides. Next, destabilizing along a neighborhood of $\alpha'\cup \beta'$ results in a Heegaard diagram $\HH''$ adapted to a foliated open book $(\{S_i\},h'', \{\gamma_i^\pm\})$ whose monodromy $h''$ is isotopic to $h$, as in Figure~\ref{fig:hd-stab-3}(c). It is clear that the destabilization map sends $\z$ to $\x_D''$.  
By Proposition ~\ref{prop:homot}, $\x_D''$ is equivalent to $\x_D$, the generator for the original Heegaard diagram. 

The argument for type $A$ is similar. The $i$th handleslide induces maps \[(f_i)_j \colon\bsahat(\Sigma,\betas_{i-1},\alphas', \overline{\zz}) \otimes \mathcal A(\overline{\zz})^{\otimes(j-1)}\to \bsahat(\Sigma,\betas_{i},\alphas', \overline{\zz})\] \noindent indexed by $j$; each map $(f_i)_1$ carries the generator $\x_{i-1}$ to the generator $\x_i$, and whenever $j>1$, we have $(f_i)_j(\x_{i-1}, a_1, \ldots, a_{j-1})=0$ for any $a_1, \ldots, a_j\in A(\overline{\zz})$.
\end{proof}

\begin{proof}[Proof of Theorem \ref{thm:ca-cd}] Proposition \ref{prop:homot} shows that isotopy of the monodromy $h$ induces maps of bordered sutured Floer modules carrying the distinguished generator to the distinguished generator, and therefore that isotopy of the monodromy preserves the homotopy equivalence class of $\xxx_D$ in $\bsdhat(-M,\Gamma,\overline{\zz})$. Similarly, the discussion at the start of this section shows that changing the choice of almost complex structure induces maps of bordered sutured Floer modules which are identity on the special generators, hence preserve the distinguished generator $\xx_D$. Proposition \ref{prop:stab-invariance} shows that stabilization on the $S_0$ page of the open book $(\{S_i\}, h, \{\gamma_i\})$ induces maps of bordered sutured Floer modules carrying the distinguished generator to the distinguished generator. By Theorem~\ref{thm:revsortedsupport}, these stabilizations suffice. Proposition \ref{prop:cutting} shows that for a sufficiently-stabilized open book, varying the choice of cutting arcs also induces maps of bordered sutured Floer modules carrying the distinguished generator to the distinguished generator. 

The argument for type A proceeds similarly. \end{proof}

Thus, we may finally define our contact invariants.  

\begin{definition}\label{def:cxi}
Let $c_D(M,\xi, \fol)$ be the type $D$ homotopy equivalence class (in the sense of Section \ref{sssec:bs}) of $\xxx_D$ in $\bsdhat(-M,\Gamma,\overline{\zz})$. Similarly, let $c_A(M,\xi, \fol)$ be the type $A$ homotopy equivalence class of $\xxx_A$ in $\bsahat(-M,\Gamma,\overline{\zz})$. 
\end{definition}


\section{Gluing}
\label{sec:gluing}

In this section, we show that after gluing, the tensor product of the classes from Definition~\ref{def:cxi}  recovers the Ozsv\'ath--Szab\'o contact class.

Suppose that $(M^L,\xi^L,\fol^L)$ and $(M^R,\xi^R,\fol^R)$ are two foliated  contact three-manifolds and that $\psi\colon \partial M^L  \to \partial M^R$ is an orientation-reversing diffeomorphism that takes the reverse of $\fol^L$ to $\fol^R$. Let $(M,\xi)=(M^L\cup_\psi M^R,\xi^L\cup_\psi\xi^R)$. 
If we glue the corresponding bordered sutured three-manifolds  $(M^L,\bsGamma^L,\mathcal{F}^L)$ and $(M^R,\bsGamma^R,\mathcal{F}^R)$, then we obtain a sutured three-manifold $(M(|\Gamma|),\Gamma)=(M^L\cup_{F(\zz)}M^R,\bsGamma^L\cup\bsGamma^R)$, where $M(|\Gamma|)$ denotes $M$ with $|\Gamma|$  balls removed. The associated  chain complex computes multipointed Heegaard Floer homology.  It follows that $\sfh(-M(|\Gamma|),\Gamma)\cong \hfhat(-M)\otimes H_*(T^{|\Gamma|-1})$.
The gluing statement is approximately the following: under the map  
\[\bsahat(-M^L,\bsGamma^L,\overline{\zz})\boxtimes_{\mathcal{A}(\overline{\zz})}\bsdhat(-M^R,\bsGamma^R, -\overline{\zz})\to \sfc(-M(|\Gamma|),\Gamma),\]
the box tensor product of the bordered contact invariants $(M^L,\xi^L)$ and $(M^R,\xi_R)$ maps to the generator corresponding to the multipointed contact invariant of $(M,\xi)$. In order to state the theorem precisely, we must fix the respective Heegaard diagrams, as we describe next.

We begin with sorted abstract foliated open books for the two foliated contact three-manifolds. We will show that gluing two  bordered sutured Heegaard diagrams adapted to these open books yields a multipointed Heegaard diagram adapted to an open book that supports $(M, \xi)$. 

In fact, we can do two different constructions. In the  first construction, we glue the two foliated open books as described in Section \ref{ssec:gluing} to obtain an open book $(S,h)$ for $(M,\xi)$. We then  construct an adapted (multipointed) Heegaard diagram $\HH$  with a distinguished generator $\xxx$  for $c(\xi)\otimes \theta^{|\bsGamma|-1}$. 
In the second construction, we begin with  bordered sutured Heegaard diagrams adapted to the foliated contact three-manifolds constructed as in Section~\ref{ssec:bordereddiagram}; these have distinguished generators $\xxx^L$ and $\xxx^R$ which give $c_D(\xi^L)$ and $c_A(\xi^R)$, respectively. We then glue the diagrams together to obtain a sutured Heegaard diagram  $\HH'$ for $(-M(|\Gamma|),\Gamma)$ with the distinguished generator $\xxx' \coloneqq \xxx^L\boxtimes \xxx^R$ representing $c_D(\xi^L)\boxtimes c_A(\xi^R)$.
 Each of these constructions involves some choice of cutting arcs. We will start with the latter construction and show that the obtained sutured Heegaard diagram for $(-M(|\Gamma|),\Gamma)$ --- interpreted as a multipointed Heegaard diagram for $M$, as in the end of Section \ref{sssec:hfc} --- is the Heegaard diagram coming from the former construction.

  Let $(\{S_{i}^L\}_{i=0}^{2k},h^L,\{\gamma_i^{\pm,L}\})$ and $(\{S_{i}^R\}_{i=0}^{2k},h^R,\{\gamma_i^{\pm,R}\})$ be sorted abstract foliated open books for  the triples $(M^L,\xi^L,\fol^L)$, and  $(M^R,\xi^R,\fol^R)$, respectively. Let $\HH^L = (\Sigma^L, \betas^L,\alphas^L, \zz^L)$ and $\HH^R = (\Sigma^R, \betas^R,\alphas^R, \zz^R)$ be Heegaard diagrams adapted to these foliated open books. Since $\fol^L$ is the reverse of  $\fol^R$, we can identify  $\zz^L$ with $-\zz^R$.\footnote{For ease of exposition, we assume that the intervals $I_{\pm}^L$ on $A_{\epsilon}^L$ and $I_{\mp}^R$ on $A_{\epsilon}^R$ are chosen so they match up after gluing, which can always be obtained up to isotopy.}  We can therefore glue the two diagrams along  $\zz^L = -\zz^R$ to obtain a sutured diagram $\HH' = \HH^L\cup \HH^R$. More precisely,  since $E_\pm^L=E_\mp^R$ and $H_\pm^L=H_\mp^R$, the endpoints of the two arcs $\beta_i^{a,L}$ and $\beta_i^{a,R}$ may be isotoped to match up on $A_{\epsilon}^L=-A_{\epsilon}^R$ for each $i\in H_\pm^L=H_\mp^R$.  This yields  the curves $\beta_i^a=\beta_i^{a,L}\cup \beta_i^{a,R}$ on the surface $\Sigma'=\Sigma^L\cup_{\mathcal{Z}_L=-\mathcal{Z}_R} \Sigma^R$. See Figure \ref{fig:gluedHD-new-2}.  
 
\begin{figure}[h!]
\begin{center}
\labellist
  	\pinlabel $\textcolor{red}{\tilde a_{j_1}^{R}}$ at 230 373
	\pinlabel $\rotatebox{70}{\textcolor{red}{\dots}}$ at 228 355
	\pinlabel $\textcolor{red}{\tilde a_{j_t}^{R}}$ at 220 334
	\pinlabel $\rotatebox{70}{\dots}$ at 69 323
	\pinlabel $\rotatebox{70}{\dots}$ at 208 376
	\pinlabel $S_{\epsilon}^L$ at -15 351
	\pinlabel $-S_{0}^L$ at -15 246
	\pinlabel $A_{\epsilon}^L$ at 130 381
	\pinlabel $A_{\epsilon}^R$ at 149 381
	\pinlabel $\textcolor{red}{-a^L}$ at 51 434
	\pinlabel $\textcolor{red}{a^L}$ at 9 378
	\pinlabel $\textcolor{blue}{b^L}$ at 37 378
	\pinlabel $\textcolor{blue}{-h^L\circ \iota(b^L)}$ at -14 451
	\pinlabel $\textcolor{blue}{-h^L(\beta_{j_1}^{-,L})}$ at 56 460
	\pinlabel $\rotatebox{50}{\textcolor{blue}{\dots}}$ at 89 472
	\pinlabel $\textcolor{blue}{-h^L(\beta_{j_t}^{-,L})}$ at 107 481
	\pinlabel $\textcolor{blue}{\beta_{i_1}^{+,L}}$ at 30 336
	\pinlabel $\rotatebox{70}{\textcolor{blue}{\dots}}$ at 14 316
	\pinlabel $\textcolor{blue}{\beta_{i_s}^{+,L}}$ at 15 291
	\pinlabel $\textcolor{red}{\tilde a_{i_1}^{L}}$ at 55 357
	\pinlabel $\rotatebox{70}{\textcolor{red}{\dots}}$ at 48 336
	\pinlabel $\textcolor{red}{\tilde a_{i_s}^{L}}$ at 38 321
	\pinlabel $\textcolor{blue}{\beta_{j_t}^{+,R}}$ at 265 399
	\pinlabel $\rotatebox{70}{\textcolor{blue}{\dots}}$ at 254 376
	\pinlabel $\textcolor{blue}{\beta_{j_1}^{+,R}}$ at 247 353
	\pinlabel $\textcolor{red}{-a^R}$ at 228 261
	\pinlabel $\textcolor{red}{a^R}$ at 270 306
	\pinlabel $\textcolor{blue}{b^R}$ at 243 306
	\pinlabel $\textcolor{blue}{-h^L\circ \iota(b^R)}$ at 292 246
	\pinlabel $e_-$ at 130 421
	\pinlabel $e_+$ at 130 271
	\pinlabel $-A_0^L$ at 125 211
	\pinlabel $-A_0^R$ at 155 211
	\pinlabel $e_+$ at 150 421
	\pinlabel $e_-$ at 150 271
	\pinlabel $S_{\epsilon}^R$ at 290 351
	\pinlabel $-S_{0}^R$ at 291 226
	\pinlabel $e_+$ at 129 38
	\pinlabel $e_-$ at 129 109
	\pinlabel $\dots$ at 140 15
	\pinlabel $\dots$ at 140 130
	\pinlabel $\delta_{i_1}$ at 81 13
	\pinlabel $\delta_{i_s}$ at 197 13
	\pinlabel $\delta_{j_1}$ at 81 137
	\pinlabel $\delta_{j_t}$ at 197 137
\endlabellist
\includegraphics[scale=.95]{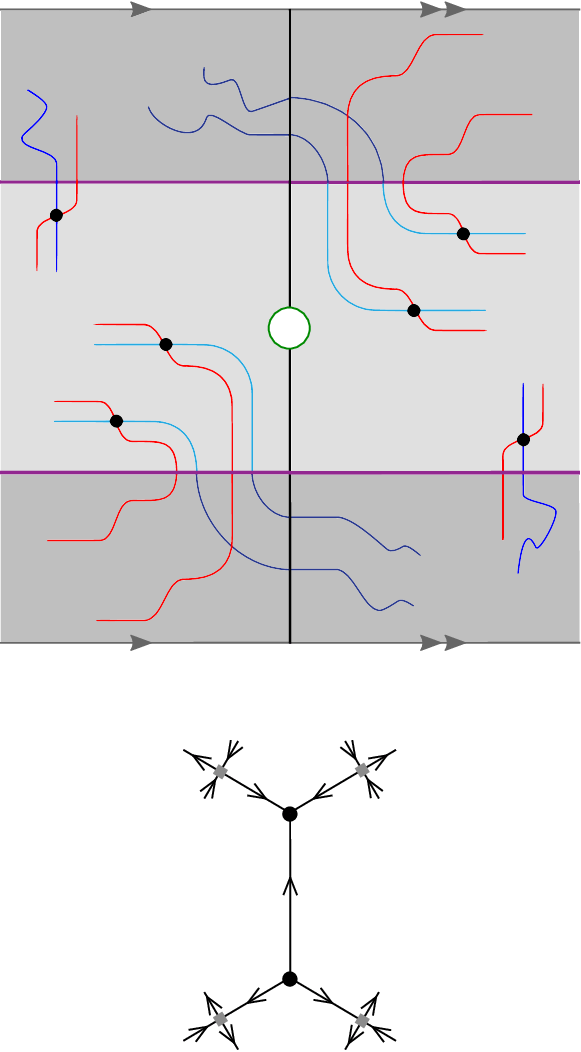}
\end{center}
\caption{Top: Two bordered sutured Heegaard diagrams with compatible boundary, glued together. We have $i_1<\cdots<i_s$ and $\{i_1, \ldots, i_s\}\in H_+^L = H_-^R$; similarly, $j_1<\cdots<j_t$ and $\{j_1, \ldots, j_t\}\in H_-^L = H_+^R$. Recall that a $\beta$-arc $\beta_i^{a, \bullet}$ corresponding to a hyperbolic point in $H_{\pm}^{\bullet}$ is also denoted by $\beta_i^{\pm, \bullet}$, to help read the diagram. Cutting arcs $b_j^L$ and $b_j^R$, as well as their perturbations $a_j^L$ and $a_j^R$, are labelled without subscripts, again for ease for reading. Bottom: The corresponding foliation on the left boundary. Only an interval of $A_0$ and the relevant nearby separatrices are shown.} 
\label{fig:gluedHD-new-2}
\end{figure}
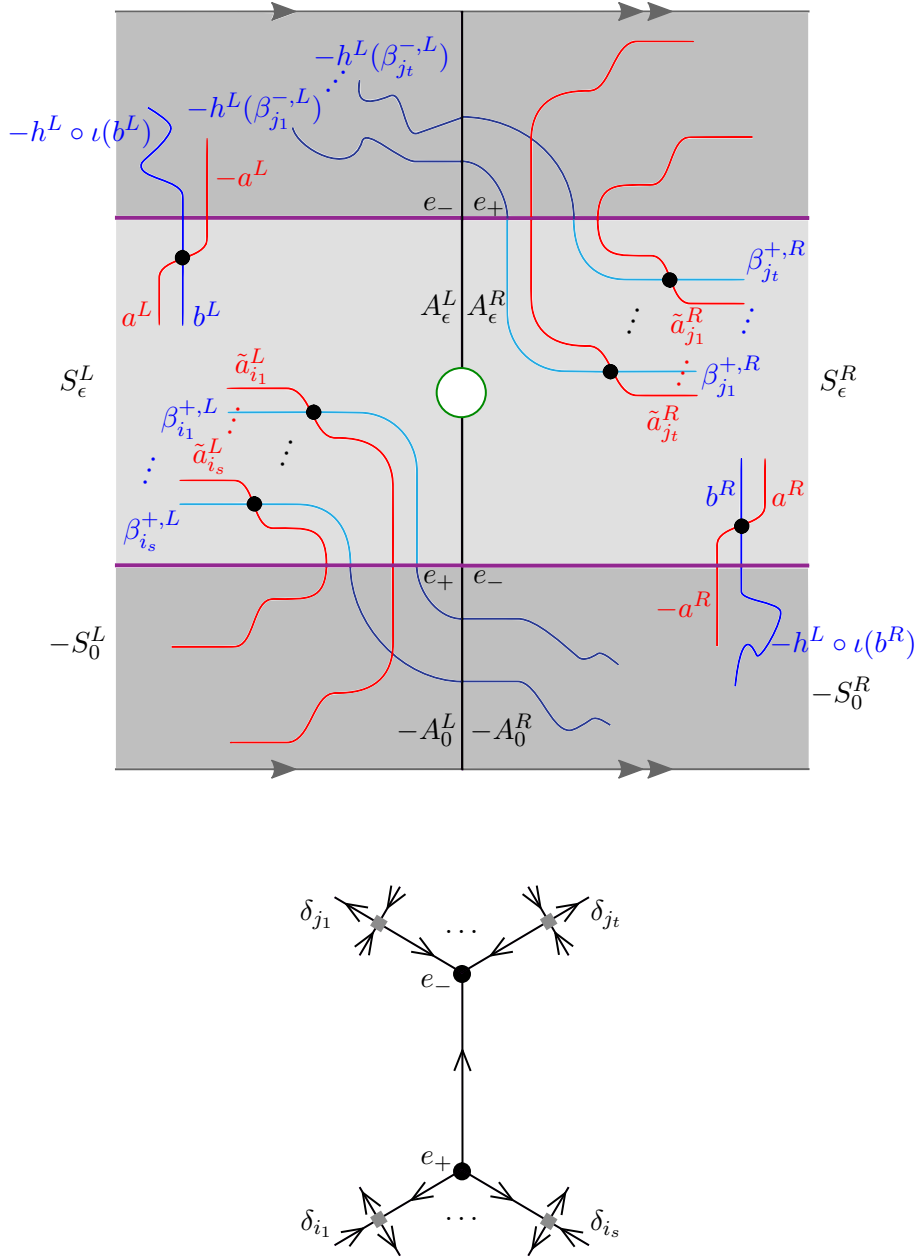

  As a next step we cap all boundary component of $\Sigma'$ with disks and place a basepoint in the middle of each disk. We  thus obtain a multipointed Heegaard diagram $\HH$ for $M$.  We could equivalently have obtained $\HH$ from $\HH^L$ and $\HH^R$ by gluing $\Sigma^L$ to $\Sigma^R$ along the entire $A_\epsilon^L$ and $A_\epsilon^R$ and placing a basepoint in the middle of each component of  $A_\epsilon^L=-A_\epsilon^R$.
Now we are ready to state the precise gluing statement:

\begin{theorem}[Gluing] \label{thm:gluing} As above, let $(M, \xi)$ denote the manifold constructed by gluing foliated open books along their compatible boundaries, and let $\HH$ be the multipointed Heegaard diagram formed by capping the discs on the resulting  sutured Heegaard diagram $\HH'$.  Then  $\HH$ is adapted to an open book $(S,h)$ for $(M, \xi)$ and the box tensor product $\x_A\boxtimes \x_D$  agrees with the generator $\mathbf{x}$ representing the multipointed contact invariant for $(M,\xi)$.
\end{theorem}  

Since our goal is to show that $\HH$ is a multipointed Heegaard diagram compatible with an open book for $(M, \xi)$, we will continue to keep track of the union $B  = B^L\cup B^R\subset \Sigma=\Sigma^L\cup_{A^L=- A^R} \Sigma^R$ of the two original bindings. Observe that $B$ splits $\Sigma$ into two ``pages", namely  $S' \coloneqq S_{\epsilon}^L\cup_{A_\epsilon^L=-A^\epsilon_R} S_{\epsilon}^R$, and $S'' \coloneqq S_0^L\cup_{A_0^L=-A^0_0} S_0^R$.

We will break the argument relating $\HH$ to an open book for $(M, \xi)$ into three propositions.
Consider the collections of arcs obtained by restricting the $\beta$ and $\alpha$ curves to $S'$ and denote these by ${\boldsymbol b}' = \{b_i'\}$ and ${\boldsymbol a}' = \{a_i'\}$, respectively. Denote the remaining parts of the $\beta$  and $\alpha$ curves by ${\boldsymbol b}''=\{b_i''\}$ and ${\boldsymbol a}''=\{a_i''\}$. 

\begin{proposition}\label{claim:discs}
The set of arcs ${\boldsymbol b}'$ cuts  $S'$ into $\frac{|E^L|}2=\frac{|E^R|}2$  disks, each of which contains exactly one basepoint. 
\end{proposition}

\begin{proof}
Following the notation in Section~\ref{ssec:bordereddiagram}, we will write $b_i^{+,\bullet}$ for $\beta_i^{a, \bullet}\cap S_{\epsilon}^{\bullet}$ whenever $i\in H_+^{\bullet}$. Thus, the collection ${\boldsymbol b}'$ is exactly the union of the collections ${\boldsymbol b}^{+,L} = \{{b_i}^{+,L}\}_{i\in H_+^L}$, ${\boldsymbol b}^L = \{b_i^L\}$, ${\boldsymbol b}^{+,R} = \{{b_i}^{+,R}\}_{i\in H_+^R}$, and ${\boldsymbol b}^R=\{b_i^R\}$.

Observe that 
\[S'\setminus {\boldsymbol b}' = (S_{\epsilon}^L\setminus \big({\boldsymbol b}^L\cup {\boldsymbol b}^{+,L})\big)\cup_{A_{\epsilon}^L =- A_{\epsilon}^R}(S_{\epsilon}^R\setminus \big({\boldsymbol b}^R\cup {\boldsymbol b}^{+,R})\big),\]
so to prove the proposition, we need to understand how the pieces in the right- and left-hand sides glue together. 

Recall from Section~\ref{ssec:bordereddiagram} that each connected component of $S_{\epsilon}^{\bullet}\setminus ({\boldsymbol b}^{\bullet}\cup {\boldsymbol b}^{+,\bullet})$ is a disk with exactly one interval of $A_{\epsilon}^{\bullet}$ on its boundary; the rest of the disk's boundary consists of copies of arcs in ${\boldsymbol b}^{\bullet}\cup {\boldsymbol b}^{+,\bullet}$ and intervals on the binding $B^{\bullet}$. The intersection with $A_{\epsilon}^{\bullet}$ must contain a basepoint, and thus we have exactly one basepoint in each component of $S_{\epsilon}^{\bullet}\setminus ({\boldsymbol b}^{\bullet}\cup {\boldsymbol b}^{+,\bullet})$.
Thus,  restricting the gluing of $\HH^L$ and $\HH^R$ to these disks results in a pairing --- each disk in $S_{\epsilon}^L\setminus ({\boldsymbol b}^L\cup {\boldsymbol b}^{+,L})$
is glued to a disk in $S_{\epsilon}^R\setminus ({\boldsymbol b}^R\cup {\boldsymbol b}^{+,R})$ along an interval of $A_{\epsilon}^{L}=-A_\varepsilon^R$.
Thus, $S'\setminus {\boldsymbol b}'$ consists of $|E_+^L| = \frac{|E^L|}{2}$ disks, each containing exactly one basepoint.
\end{proof}

Next we observe the following:

\begin{proposition}\label{prop:perturb}
The arcs in ${\boldsymbol b}'$ are the standard perturbations (i.e., perturbations ``to the right", as in \cite{hkm09}) of the arcs in ${\boldsymbol a}'$.
\end{proposition}
\begin{proof}
This follows immediately from the construction in Section~\ref{ssec:bordereddiagram}, as each arc $\tilde a_i^{\bullet}$ is a perturbation to the left of the arc $b_i^{+,\bullet}$, and each arc $a_i^{\bullet}$ is a perturbation to the left of $b_i^{\bullet}$.
\end{proof}

Note that this implies that we can replace ${\boldsymbol b}'$ with ${\boldsymbol a}'$ in the statement of Proposition~\ref{claim:discs}.

Recall from Section~\ref{ssec:gluing} the definition of the monodromy associated to the open book $(S,h)$ built by gluing a pair of compatible foliated open books: the abstract map $h:S\rightarrow S$ is defined by flowing $S_0$ through the manifold in the positive $t$-direction to $S_{2k}$, a map denoted by $\iota$, and then applying the left and right monodromies to the appropriate subsurfaces.

\begin{proposition}\label{prop:curves}
With the above definitions, $a_i'' = a_i'$ (via the trivial identification of the $\epsilon$ and $0$ pages). Furthermore, up to isotopy fixing the boundary, $h\colon S'\to S''$ maps the arcs $b_i'$ onto $b_i''$.
\end{proposition}

\begin{proof}
The first statement follows immediately from the above discussion. To prove the second assertion, consider the identifications 
\[R_{\epsilon}^L\cong R_+(M^L)\cong R_-(M^R)\cong R_{2k}^R \qquad\text{ and } \qquad R_{\epsilon}^R\cong R_+(M^R)\cong R_-(M^L)\cong R_{2k}^L\]
in the  definition of $\iota$.

Each of the surfaces is defined as a  neighborhood of a graph, and we note that away from a collar neighborhood of $A_\epsilon^L=-A_\epsilon^R$ and $A_{2k}^R=-A_{2k}^L$ the identifications may be chosen to induce isotopies \[\gamma_i^{+,L}\cong  {\delta_i^L\cong \delta_i^R}\cong \gamma_i^{-,R} \qquad\text{ and } \qquad \gamma_i^{+,R}\cong  {\delta_i^R\cong \delta_i^L}\cong \gamma_i^{-,L}.\] (Recall that $\delta_i$ are the (un)stable separatrices of the foliation connecting $h_i$ to the elliptic points.)

Away from $A_\epsilon^L=-A_\epsilon^R$, the $\boldsymbol{b}'$ agree with $\gamma^{+,\bullet}$ curves, which are carried by $\iota$ to $\gamma^{-,\bullet}$ curves before being sent by $h^R$ and $h^L$  to the ${\boldsymbol b}^{''}$ curves on $S''$.   As the collar neighborhood of $-A_0^L \cup A_\epsilon^L=A_0^R \cup -A_\epsilon^R$ is a union of disks, the statement follows.

\end{proof}

\begin{proof}[Proof of Theorem~\ref{thm:gluing}]
The above propositions demonstrate that the  multipointed Heegaard diagram $\HH$   constructed above corresponds to the open book $(S,h)$ obtained by gluing the sorted foliated open books $(\{S_i^L\},h^L)$ and $(\{S_i^R\},h^R)$. Indeed, first, $S$ is diffeomorphic to $S'$ and $S''$, and 
by Proposition \ref{claim:discs} the page $S'$ is cut by ${\boldsymbol a}'$ into $|\Gamma|$ disks, each containing a basepoint. Furthermore, by Proposition \ref{prop:perturb}, the arcs in ${\boldsymbol b}'$ are obtained by the usual perturbation of the ${\boldsymbol a}'$ arcs. Lastly, by Proposition \ref{prop:curves}, the arcs of ${\boldsymbol a}''$ are  the image of the arcs of ${\boldsymbol a}'$ under the identification of $S'$ with $S''$, while 
the arcs in ${\boldsymbol b}''$ are  the image of the arcs of ${\boldsymbol b}'$ under a diffeomorphism in the mapping class of $h$. There is only one generator in the page $S'$, which  by definition represents the multipointed contact invariant of $\xi$, while by construction it is also the box tensor of the bordered contact invariants for $\xi^L$ and $\xi^R$. This completes the proof. 
\end{proof}


\section{Relationship to the HKM contact element}
\label{sec:relation}
Let $\zz=(Z,a,m)$ be an arc diagram and $F=F(\zz)$ be the corresponding surface. The union of $Z$ and the parametrizing arcs  $\amalg_{i=1}^{2k}e_i$ on $F$ is an embedded graph, denoted by $G(\zz)\subset F$. For any subset $r$ of parametrizing arcs, Zarev defines a bordered sutured manifold $\mathcal{W}_r$, called a \emph{cap}, in  \cite[Definitions 2.5 and 6.1]{bs-JG}. We recall the definition for a $\beta$-type arc diagram $\zz$.

 First, $\mathcal{W}_r$ consists of the three-manifold $W=F\times [0,1]/\sim$ where $(x,t)\sim(x,t')$ whenever $x\in \bdy F$. Thus, $\bdy W=(-F\times \{0\})\cup (F\times \{1\})$, and the bordered and sutured parts of $\bdy W$ are $-F\times\{0\}$ and $F\times \{1\}$, respectively. Second, each component of $Z$ is a push-off of an arc in $\bdy F$ into the interior of $F$; denote the union of these arcs in $\partial F$ by $S_-$. Let $R$ be the union of the disk regions enclosed by $Z$ and $S_-$. Then, the sutures $\bsGamma_r\subset F\times \{1\}$ are defined as $\bdy\left( R\cup N(r)\right)\setminus S_-$ where $N(r)=\amalg_{e_i\in r}N(e_i)$. Furthermore, $R_{-}(\bsGamma_r)=R\cup N(r)$. Finally, $-F\times \{0\}$ is parametrized by $-\mathcal{Z}$.

Zarev constructs a bordered sutured Heegaard diagram $\HD^r=(\Sigma^r,\alphas^r,\betas^r,-\zz)$ for any cap $\mathcal{W}_r$ as follows. Let $D=(-Z)\times [0,1]$. The endpoints of any parametrizing arc $e_i$ is a pair of matched points in $a\subset -Z$. To construct $\Sigma^r$, first for every $e_i\notin r$, do $0$-surgery on $D$ along the surface framed $0$-sphere obtained by slightly pushing $(\bdy e_i)\times\{0\}$ into $\Int(D)$. Then, for any $e_i\in r$ attach a  band to $D$  along a small neighborhood in $-Z\times \{1\}$ of the $0$-sphere $(\bdy e_i)\times \{1\}$  so the resulting surface remains oriented. 
This diagram has no $\beta$-circles.  However, associated  to each $e_i$ there is a $\beta$ arc $\beta_i^r\subset \Sigma^r$ which connects the points $(\bdy e_i)\times \{0\}$. If $e_i\in r$, $\beta_i^r$ is the union of $(\bdy e_i)\times [0,1]$ and the core of the corresponding band. If $e_i\notin r$, $\beta_i^r$ goes over the tube on the boundary of the corresponding $1$-handle once. Write $\betas^r=\{\beta_1^r,\ldots,\beta_{2k}^r\}$. Next, for any $e_i\notin r$ let $\alpha_i^r\subset \Sigma^r$ be the dual circle to $\beta_i^r$, i.e., the belt sphere of the associated $1$-handle. This circle intersects $\beta_i^r$ in a single point denoted by $x_i^r$ and is disjoint from $\beta_{j}^r$ for any $j\neq i$. Then, let $\alphas^r=\{\alpha_i^r\ |\ e_i\notin r\}$. Finally, $G(-\zz)$ is embedded in $\Sigma^r$ such that $-Z$ is identified with $(-Z)\times \{0\}$ and any edge $e_i$ is identified with $\beta_i^r$. Clearly, this diagram has a unique generator $\x^r=\{x_i^r\ |\ e_i\notin r\}$. See Figure~\ref{fig:CapHD}.

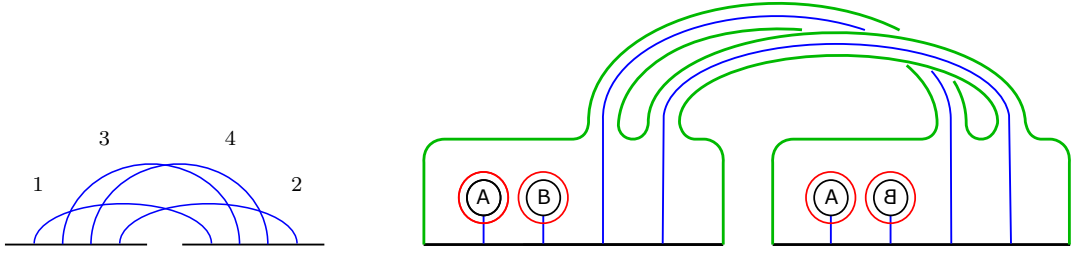
\begin{figure}[h]
\begin{center}
\labellist
  	\pinlabel $\textcolor{black}{\scriptstyle 1}$ at 50 150
	\pinlabel $\textcolor{black}{\scriptstyle 2}$ at 440 150
  	\pinlabel $\textcolor{black}{\scriptstyle 3}$ at 150 220
	\pinlabel $\textcolor{black}{\scriptstyle 4}$ at 340 220
\endlabellist
\includegraphics[scale=0.25]{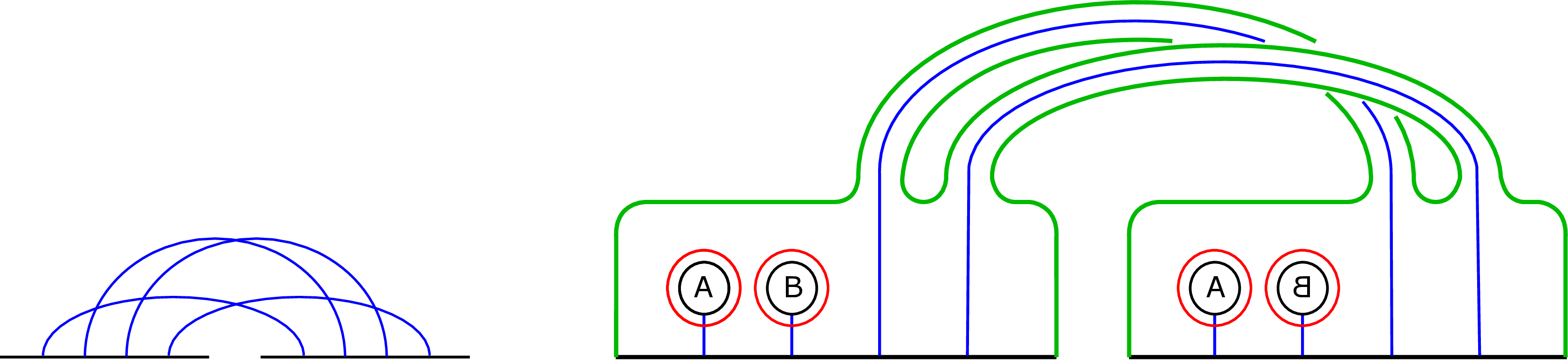}
\end{center}

\caption{Left: The arc diagram for the torus example in Figure \ref{fig:torus-bs-NEW}; note that the arc $e_i$ corresponding to $\delta_i$ is labelled with $i$. Right: The bordered sutured diagram $\HD^r$ for $r=\{e_2,e_4\}$.}
\label{fig:CapHD}
\end{figure}

We now consider a particular cap that arises naturally in the context of the bordered sutured manifold $(M,\bsGamma,\mathcal{Z})$ associated to a foliated contact three-manifold  $(M,\xi,\mathcal{F})$.  As in Section~\ref{ssec:suturedbordered}, let $\mathcal{Z}=(Z,a,m)$ be the arc diagram  which parametrizes the bordered part $F=F(\mathcal{Z})$ in $\bdy M$ and let $\Gamma
(\fol)$ denote the dividing set associated to the foliation $\mathcal{F}$. Let $k$ be the number of elements in $H_{\pm}$,  so $|a|=4k$. Each matched pair of points in $a$ corresponds to a hyperbolic point  and to a parameterizing arc $e_i\subset\delta_i$, where $i\in [2k]$ is the index of the point. From now on, fix $r=\amalg_{i\in H_{+}}e_i$ and define the associated cap $\mathcal{W}:=\mathcal{W}_{r}$ with suture $\bsGamma_{r}$.

Correspondingly, we define the sutured manifold $(M, \widetilde{\Gamma}):= (M,\bsGamma,\mathcal{Z})\cup_{F}\mathcal{W}$. We will show here that the sutures $\widetilde{\Gamma}$ may in fact be identified with  $-\Gamma(\fol)$. Recall that $\widetilde{\Gamma}$ separates $\partial M$ into regions $R_\pm(\widetilde{\Gamma})$.
It is straightforward from the definition that 
\[R_{-}(\widetilde{\Gamma})=R_{-}(\bsGamma)\cup R_-(\bsGamma_r)=R_-(\bsGamma)\bigcup(R\cup N(r)).\]
Observe that $R_-(\bsGamma)\bigcup(R\cup N(r))$ is a neighborhood of the stable separatrices of points in $H_+$. Thus, we can identify it with $R_+(\fol)$, and as a consequence, $\widetilde{\Gamma}=-\bdy R_+(\fol)$ and is isotopic to $-\Gamma(\fol)$.  To avoid a proliferation of notation, we will henceforth denote the sutured manifold obtained by gluing the cap as $(M,\Gamma):=(M, -\Gamma(\fol))$.

Consequently, if we define $\overline{\mathcal{W}}=(-W,\bsGamma,-\overline{\zz})$, then \[(-M,-\Gamma(\fol))=(-M,\bsGamma,\overline{\zz})\cup_{\overline{F}} \overline{\mathcal{W}}.\]
The  gluing formula gives the following:
\[\sfc(-M,-\Gamma(\fol))\simeq \bsahat(-M,\bsGamma,\overline{\zz})\boxtimes_{\mathcal{A}(\overline{\zz})}\bsdhat(\overline{\mathcal{W}}).\]

Let  $\iota_r\in \mathcal{A}(\overline{\zz})$ be the idempotent corresponding to the horizontal arcs for the points in $a$ that lie on the boundary of $r$. It follows from the special structure of the Heegaard diagram for $\mathcal{W}$  that $\bsdhat(\overline{\mathcal{W}})$ is an elementary type $D$ module generated by $\iota_{r}$; see \cite[Proposition 6.2]{bs-JG}.  Thus
\[\sfh(-M,-\Gamma(\fol))\cong H_*\left(\bsahat(-M,\bsGamma,\overline{\zz})\right)\cdot \iota_r.\]
See \cite[Theorem 6.5]{bs-JG}.

We now prove Theorem~\ref{thm:hkm} for $\iota_+$ given by $\iota_r$. 

\begin{proof}[Proof of Theorem~\ref{thm:hkm}]
Suppose $(\{S_i\}_{i=0}^{2k}, h, \{\gamma_i^{\pm}\})$ is a sorted foliated open book compatible with the triple $(M,\xi,\mathcal{F})$. Let $\HH=(\Sigma,\alphas,\betas=\betas^a\cup\betas^c,\zz)$ be the bordered sutured Heegaard diagram constructed in Section \ref{ssec:bordereddiagram} for $(M,\bsGamma, \zz)$. We begin with a brief sketch of the argument before presenting it in detail.  First, we will construct a sutured Heegaard diagram for $(M,-\Gamma(\fol))$ by gluing $\HH$ to the diagram $\HH^{r}$, where $\HH^{r}$ denotes the diagram for the cap $\mathcal{W}$ described above.  Then, we show that after possibly performing some handleslides and $2k$-times destabilizing this diagram, one recovers the sutured Heegaard diagram corresponding to the partial open book $(S'=S_{0}, P'=P_{0},h'=h|_{P_{2k}}\circ\iota)$ compatible with the contact sutured manifold $(M,\Gamma(\fol),\xi)$, following the conventions in \cite[Section 2]{hkm09}. Note that this is a diagram for $(M,-\Gamma(\fol))$.  Letting $\x^r=\{x_i^r\ |\ i\in H_-\}$ denote the generator of $\HH^r$, we additionally prove that the map induced  by the Heegaard moves sends the generator $\x\otimes \x^r$ in $\sfc(\overline{\HH}\cup \overline{\HH^r})$ to the generator  which represents the contact invariant $\mathrm{EH}(M,\Gamma(\fol), \xi)\in \sfh(-M,-\Gamma(\fol))$.

After gluing $\HH$ and $\HH^r$ along $Z$, for every $i\in [2k]$, the arc $\beta_i^{r}$ in $\Sigma^r$ will be paired with the arc $\beta_i^{a}$ in $\Sigma$, producing a circle $\widetilde{\beta}_i$. For any $i\in H_+$, there exists a circle $\widetilde{\alpha}_i$ in $\alphas$, and for any 
$i\in H_-$, we define $\widetilde{\alpha}_i=\alpha_i^{r}$. We write ${\alphas}^r=\{\widetilde{\alpha}_i\ |\ i\in H_-\}$. Then 
\[\HH'=\HH\cup\HH^{r}=\left(\Sigma'=\Sigma\cup_{Z}\Sigma^{r},\alphas'=\alphas\cup{\alphas}^r, \betas'=\betas^c\cup\{\widetilde{\beta}_i\}_{i\in [2k]}\right) \]
is a Heegaard diagram for $(M,-\Gamma(\fol))$. 
In this diagram, we see that $\widetilde{\alpha}_i\cap\widetilde{\beta}_i=\{x_i^{+}\}$ for $i\in H_+$, while $\widetilde{\alpha}_i\cap\widetilde{\beta}_i=\{x_i^{r}\}$ for $i\in H_-$. Furthermore, for $i\in H_-$, the $\alpha$-circle $\widetilde{\alpha}_i$ is disjoint from all circles in $\betas'$ aside from $\widetilde{\beta}_i$. On the other hand, since $\HD$ is defined using an abstract sorted foliated open book decomposition, Condition (2) of Definition \ref{def:sorted} implies that for any $i,j\in H_+$, the arc $\beta_i^a\subset\Sigma$ is disjoint from the circle $\widetilde{\alpha}_j\subset\Sigma$ if and only if $j>i$. Thus, for $i\in H_+$, the $\beta$-circle $\widetilde{\beta}_i$ is disjoint from the subset
\[\{\alpha_1,\ldots,\alpha_{2g_0+n_0+|A_0|-k-2}\}\cup\{\widetilde{\alpha}_j\ |\ j\in H_+\ \text{and}\ j>i\}\cup{\alphas}^r\subset\alphas'.\]

Suppose $i\in H_-$. Since  $\widetilde{\alpha}_i$ intersects $\widetilde{\beta}_i$ once, the curves in $\alphas$  can be made disjoint from $\widetilde{\beta}_i$ by sliding each curve in $\alphas$ that intersects $\widetilde{\beta}_i$ over $\widetilde{\alpha}_i$. Then we may destabilize $\HH'$ using the pair $\widetilde{\alpha}_i$ and $\widetilde{\beta}_i$. Repeating this procedure for each $i\in H_-$ yields a Heegaard diagram $\HD''=(\Sigma'',\alphas'',\betas'')$ for $(M,-\Gamma(\fol))$, where  $\alphas''=\alphas$,  $\betas''=\betas^c\cup\{\widetilde{\beta}_i\}_{i\in H_+}$, and $\Sigma''$ is obtained from $\Sigma$ by attaching bands along $\bdy \beta_i^a$ for each $i\in H_+$.

Write $H_-=\{j_1<j_2<\cdots<j_k\}$. Let $\HH'_l=(\Sigma'_l,\alphas'_l,\betas'_l)$ be the Heegaard diagram obtained from $\HH'$ by applying the aforementioned Heegaard moves for $j_1,\ldots,j_{l-1}$. Denote by
\[f^-_{l}:\sfc(\ol{\HH'}_{l})=\sfc(\Sigma'_{l},\betas'_{l},\alphas'_{l})\to\sfc(\ol{\HH'}_{l+1})=\sfc(\Sigma'_{l+1},\betas'_{l+1},\alphas'_{l+1})\]
the composition of chain maps corresponding to the $l$th sequence of moves.

Let $f^-=f^-_{k}\circ\cdots\circ f^-_1$.
Note that the intersection points for $\HH$ and $\HH''$ are in one-to-one correspondence, and $\x\in\sfc(\ol{\HH''})$ denotes the generator corresponding to the intersection point $\x$ for $\HH$.
Let $f^h_l$ be the composition of chain maps associated with the sequence handleslides, and let $f_l^{ds}$ be the chain map associated with the  subsequent destabilization. 
By an argument analogous to that in the proof of Proposition~\ref{prop:stab-invariance}, we have that  for an appropriate choice of the complex structure $f^-_l=f_l^{ds}\circ f_l^h(\x\otimes \x^{r}_l)=\x\otimes\x_{l+1}^{r}$, and so $f^-(\x\otimes\x^{r})=\x$. Here, $\x_l^{r}=\{x_{j_i}^{r}\ |\ l\le i\le k\}$.

Next, we simplify the diagram $\HH''$ using the pairs $\widetilde{\alpha}_i$ and $\widetilde{\beta}_i$ for all $i\in H_+$. Suppose $H_+=\{i_1<i_2<\cdots<i_k\}$. Note that $\widetilde{\beta}_{i_1}\cap\widetilde{\alpha}_{i_1}=x^+_{i_1}$ while $\widetilde{\beta}_{i_1}$ is disjoint from the rest of the $\alpha$-circles. As before,  first slide every $\beta$-circle that intersects $\widetilde{\alpha}_{i_1}$ over $\widetilde{\beta}_{i_1}$ to remove the intersection points and then destabilize the diagram using  
$\widetilde{\alpha}_{i_1}$ and $\widetilde{\beta}_{i_1}$. In the new diagram, $\widetilde{\beta}_{i_2}$ is disjoint from all $\alpha$-circles except  $\widetilde{\alpha}_{i_2}$, and $\widetilde{\beta}_{i_1}\cap\widetilde{\alpha}_{i_1}=x^+_{i_2}$. So we can repeat this simplification for $i_2$, and in fact continue until $i=i_k$. At the end, we get a surface homeomorphic to $\Sigma\setminus \amalg _{i\in H_+}N(\beta_i^a)$. By definition, $\Sigma=S_{\epsilon}\cup_B(-S_0)$, where $S_{\epsilon}$ is a copy of  $S_0$. Therefore, 
\[\Sigma\setminus \amalg_{i\in H_+}N(\beta_i^a)\cong (S_{0}\setminus\amalg_{i\in H_+}N(\gamma_i^+))\cup_B(-S_0)=P_{0}\cup_{B}(-S_0)\]
and this surface is equipped with 
\[\{\alpha_1,\ldots,\alpha_{2g_0+n_0+|A_0|-k-2} \}\quad \text{and} \quad\{\beta_1,\ldots,\beta_{2g_0+n_0+|A_0|-k-2}\}.\]
This is exactly the sutured Heegaard diagram defined in \cite{hkm09} associated to the partial open book $(S',P',h')$. 

Let $\HH''_l$ be the Heegaard diagram obtained from $\HH''$ after applying the above Heegaard moves for $i_1,\ldots,i_{l-1}$ and denote the chain map induced by the $l$th sequence of moves by
\[f_l^+:\sfc(\ol{\HH''_l})\to\sfc(\ol{\HH''_{l+1}}).\]

Let $f^+=f_k^+\circ\cdots\circ f_1^+$. Once again, by an argument analogous to that in the proof of Proposition~\ref{prop:stab-invariance} we have that for an appropriate choice of  complex structure, $f^+(\x)$ is the generator representing the contact invariant $\mathrm{EH}(M,\Gamma(\fol),\xi)$, i.e., $\{x_1,\ldots,x_{2g_0+n_0+|A_0|-k-2}\}$.

Hence, $f^+_*\circ f^-_*(c_A(\xi)\cdot\iota_r)=f^+_*\circ f^-_*([\x\otimes\x^r])=\mathrm{EH}(M,\Gamma(\fol),\xi)$. 
\end{proof}


\section{Examples}
\label{sec:examples}

\begin{example}\label{ex:ball1} 
For our first example, consider the foliated open books obtained as follows. Begin with the open book for $(S^3, \xi_{\mathrm{std}})$ whose pages are disks. Embedding a sphere in this open book so that it intersects the binding in four points, as in Figure~\ref{fig:H1ball},   produces two tight balls. We first consider the ball whose initial page is a rectangle. 

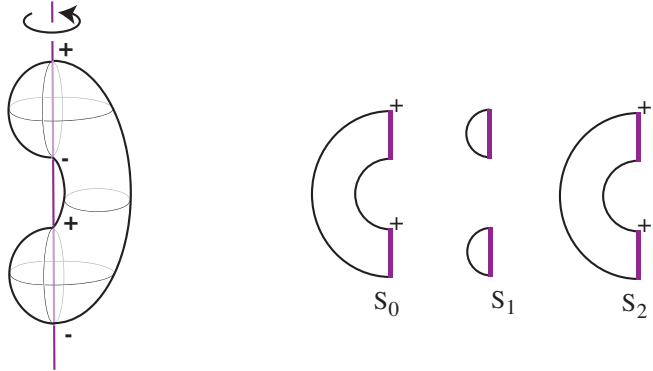
\begin{figure}[h]
	\labellist
	\pinlabel {$e_+$} [bl] at 21 143
	\pinlabel {$e_-$} [l] at 21 100
	\pinlabel {$e_+$} [l] at 21 66
	\pinlabel {$e_-$} [tl] at 21 23
	\pinlabel {$e_+$} [ ] at 182 126
	\pinlabel {$e_+$} [ ] at 182 70
	\pinlabel {$e_+$} [ ] at 300 126
	\pinlabel {$e_+$} [ ] at 300 70
	\pinlabel {$S_0$} [ ] at 167 23
	\pinlabel {$S_1$} [ ] at 225 23
	\pinlabel {$S_2$} [ ] at 284 23
	\endlabellist
	\begin{center}
		\includegraphics[scale=0.8]{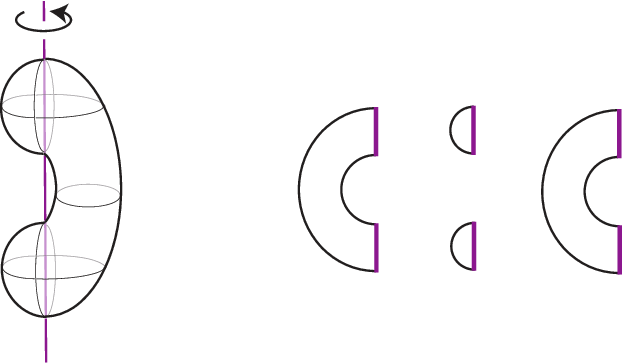}
		\caption{The  sphere shown on the left separates $S^3$ into two tight balls and induces a foliated open book for each.   The elliptic points are labelled with the signs associated to the ball whose first page is connected, as shown on the right.}\label{fig:H1ball}
	\end{center}
\end{figure}

Observe that this ball can also be obtained by cutting the solid torus of Example~\ref{ex:torus} along a meridional disk disjoint from the binding, which has the effect of removing a pair of hyperbolic points from the induced foliated open book.  In this example, the first hyperbolic point is positive and the second negative, so the foliated open book is automatically sorted.  The associated Heegaard diagram  $\HD_1$ is shown  in Figure~\ref{fig:HD-ball}.
 
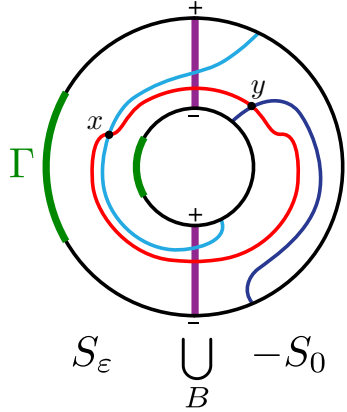
\begin{figure}[h]
	\labellist
	\pinlabel {\Large $S_\varepsilon$} [B] at 52 5
	\pinlabel {\Large $-S_0$} [B] at 98 5
	\pinlabel {$B$} [ ] at 71 -6
	\pinlabel {$e_+$} [b] at 74 177
	\pinlabel {$e_-$} [t] at 74 137
	\pinlabel {$e_+$} [b] at 74 79
	\pinlabel {$e_-$} [t] at 74 38
	\pinlabel {$x$} [ ] at 28 130
	\pinlabel {$y$} [ ] at 101 144
	\endlabellist
	\begin{center}
	\includegraphics[scale=0.8]{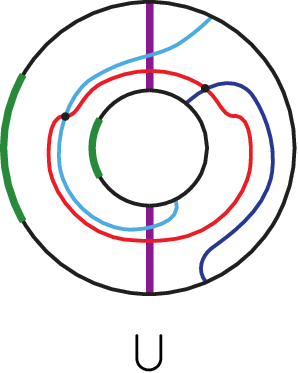}
	\end{center}
\caption{The bordered sutured Heegaard diagram adapted to the sorted foliated open book of Example~\ref{ex:ball1}.}\label{fig:HD-ball}
\end{figure}

 Let $x$ be the unique intersection point on the $S_{\epsilon}$ part of the diagram, and let $y$ be the other intersection point. Let $\rho_1$ and $\rho_2$ be the algebra elements corresponding to the Reeb chords on the inside and outside components of $\bdy\overline{\HD_1}$, as seen on Figure~\ref{fig:HD-ball}.  The type $A$ structure $\bsahat(\overline{\HD_1})$ is generated by $x$ and $y$, and has structure maps
\begin{align*}
m_2(y, \rho_1) & = x\\
m_2(y, \rho_2) & = x.
\end{align*}
The contact class $c_A(B^3,\xi_1, \fol_1)$ is the homotopy equivalence class of $x$. 
\end{example}

\begin{example}\label{ex:ball2}
As a second example we consider the ball in $(S^3, \xi_{\mathrm{std}})$ which is the complement of the previous example. 
\begin{figure}[h]
	\labellist
	\pinlabel {$S_0$} [ ] at 41 -7
	\pinlabel {$S_1$} [ ] at 132 -7
	\pinlabel {$S_2$} [ ] at 230 -7
	\endlabellist
	\begin{center}
		\includegraphics[scale=0.8]{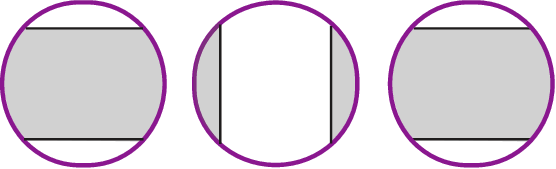}
		\end{center}
		\caption{The shaded regions represent the pages of the solid ball from Example~\ref{ex:ball1}, while their unshaded complements are the pages of a foliated open book for the ball of Example~\ref{ex:ball2}.}\label{fig:compball1}
\end{figure}

Observe that $S_0$ is disconnected, and the first negative hyperbolic point is followed by a second positive hyperbolic point.  The associated foliated open book is not sorted, so we stabilize it before constructing the associated bordered sutured Heegaard diagram $\HD_2$. The result of the stabilization is shown in  Figure~\ref{fig:torus2stab}.

\begin{figure}[h]
	\labellist
	\pinlabel {$S_0$} [ ] at 16 31
	\pinlabel {$S_1$} [ ] at 131 12
	\pinlabel {$S_1$} [ ] at 111 138
	\pinlabel {$\tau$} [ ] at 140 49
	\pinlabel {$\tau^{-1}$} [ ] at 100 98
	\pinlabel {$\gamma_2^-$} [ ] at 149 78
	\pinlabel {$\gamma_3^+$} [ ] at 144 137
	\pinlabel {$S_2$} [ ] at 214 31
	\endlabellist
	\begin{center}
		\includegraphics[scale=0.8]{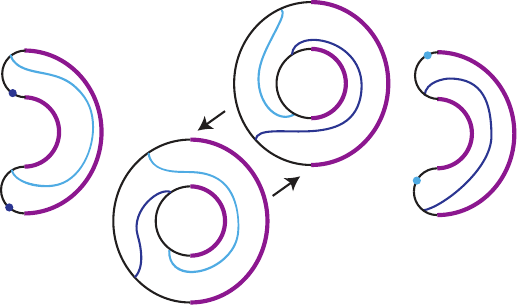}
		\caption{A sorted foliated open book for the ball of Example~\ref{ex:ball2}.}\label{fig:torus2stab}
	\end{center}
\end{figure}
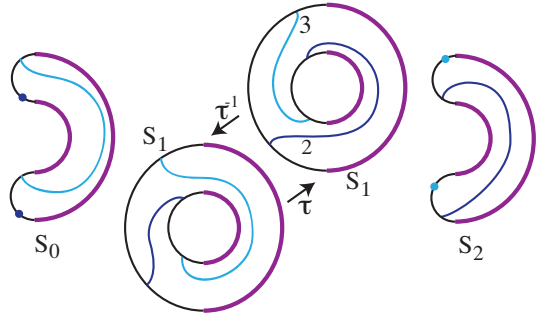

By keeping track of the four elliptic points in the embedded open book decomposition of $(S^3, \xi_{\mathrm{std}})$ and their images on the two Heegaard diagrams,  we see that the corresponding Heegaard diagram $\HD$ for $(S^3, \xi_{\mathrm{std}})$ is obtained by gluing $\HD_1$ to $\HD_2$ by identifying the outside boundary component of one diagram to the inside boundary component of the other, and vice versa. Thus, to have consistent labels (between $\HD_1$ and $\HD_2$) of the algebra generators for the given foliation, we must label the generator for the inside boundary component of $\overline{\HD_2}$ by $\rho_2$ and the other one by $\rho_1$. Let $s$ be the unique intersection point on the $S_{\epsilon}$ part of $\overline{\HD_2}$, and let $t$ be the other intersection point. The type $D$ structure $\bsdhat(\overline{\HD_2})$ is generated by $s$ and $t$, and has structure map
\[\delta^1(t)  = (\rho_1+\rho_2)\otimes s.\]
The contact class $c_D(B^3,\xi_2, \fol_2)$ is the homotopy equivalence class of $s$. 
\end{example}

\begin{example}\label{ex:ball-glue}
We take the box tensor product of the type $A$ structure from Example~\ref{ex:ball1} and the type $D$ structure from Example~\ref{ex:ball2}, to compute the contact class for $(S^3, \xi_{\mathrm{std}})$. Considering the idempotents, $\bsahat(\overline{\HH_1})\sqtens\bsdhat(\overline{\HH_2})$ is generated by $y\otimes t$ and $x\otimes s$.The differential $\bdy=\bdy^{\sqtens}$ is trivial, so the homology $H_{\ast}(\bsahat(\overline{\HH_1})\sqtens\bsdhat(\overline{\HH_2}))$ is generated by $[y\otimes t]$ and $[x\otimes s]$. Therefore, $[x\otimes s]$ is the contact class of the glued three-sphere $(S^3, \xi_{\mathrm{std}})$. 
\end{example}

\begin{example}\label{ex:torus1}

Finally, we return to the  the running example used throughout this paper: the solid torus first introduced in Example~\ref{ex:torus}.  Recall that Example~\ref{ex:HeegaardDiagram} constructed the bordered sutured Heegaard diagram adapted to a  sorted foliated open book for this manifold.  Here we compute the associated type $A$ structure.

For simplicity we label the intersection points in Heegaard diagram $\HD$ by $x_1$, $x_2$, $x_3$, $x_4$, $y_1$, and $y_2$ as in Figure \ref{fig:HD5_diagram_labelled}. In this diagram, we have six generators\[\left\{x_1y_1, x_1y_2, x_2y_2,x_3y_1,x_3y_2,x_4y_2\right\}\] where  $x_1y_1$ is the generator  $x_1^+x_3^+$ defined in Section~\ref{ssec:bordereddiagram}.  Then, we simplify our depiction of the Heegaard diagram $\HD$ by first removing the domains adjacent to $\bdy\Sigma\setminus Z$ and then cutting along the intersection points $x_1$, $y_1$ and $y_2$. Moreover, $\overline{\HD}$ is obtained from $\HD$ by switching $\alphas$ and $\betas$, so we draw $\betas$ arcs in red, and $\alphas$ circles in blue.  See Figure \ref{fig:HDTorus-A}.

\begin{figure}[h]
	\labellist
	\pinlabel {$S_\varepsilon$} [ ] at 15 0
	\pinlabel {$-S_0$} [ ] at 308 0
	\pinlabel {$B$} [] at 163 61
	\pinlabel {$e_+$} [b] at 75 140
	\pinlabel {$e_-$} [t] at 75 101
	\pinlabel {$e_+$} [b] at 75 43
	\pinlabel {$e_-$} [t] at 75 3
	\pinlabel {$e_+$} [b] at 258 140
	\pinlabel {$e_-$} [t] at 258 101
	\pinlabel {$e_+$} [b] at 258 43
	\pinlabel {$e_-$} [t] at 258 3
	\pinlabel {\tiny $x_1$} [ ] at 27 93
	\pinlabel {\tiny $y_1$} [ ] at 121 86
	\pinlabel {\tiny $y_2$} [ ] at 214 80
	\pinlabel {\tiny $x_2$} [ ] at 295 55
	\pinlabel {\tiny $x_3$} [ ] at 284 93
	\pinlabel {\tiny $x_4$} [ ] at 273 113
	\endlabellist
	\begin{center}
	\includegraphics[scale=0.8]{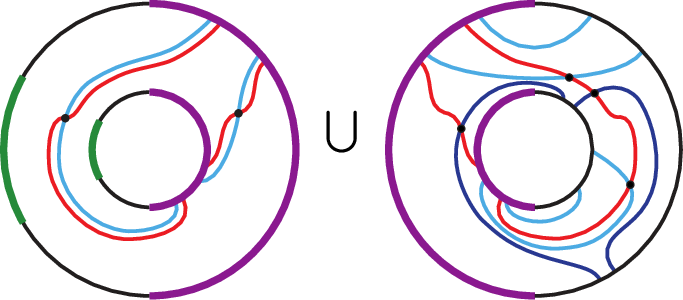}
	\end{center}
	\caption{The Heegaard diagram $\HH$ from Example~\ref{ex:HeegaardDiagram}, with intersection points labelled so that $x_1y_1$ is the contact class.}\label{fig:HD5_diagram_labelled}
\end{figure}
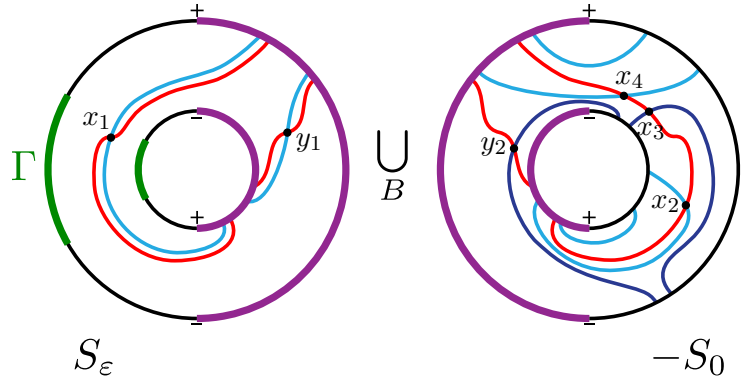

\begin{figure}[h]
\begin{center}
\labellist
         \pinlabel $\textcolor{blue}{\beta_2^-}$ at 162  8
         \pinlabel $\textcolor{gray}{\tiny 1}$ at 157  25
        \pinlabel $\textcolor{blue}{\beta_4^-}$ at 162 40
        \pinlabel $\textcolor{gray}{\tiny 2}$ at 157  55
	\pinlabel $\textcolor{blue}{\beta_3^+}$ at 162 71
	\pinlabel $\textcolor{gray}{\tiny 3}$ at 157  87
	\pinlabel $\textcolor{blue}{\beta_1^+}$ at 162 106	
	\pinlabel $\textcolor{blue}{\beta_2^-}$ at 162 175
	\pinlabel $\textcolor{gray}{\tiny 4}$ at 157  192
	\pinlabel $\textcolor{blue}{\beta_4^-}$ at 162 207
	\pinlabel $\textcolor{gray}{\tiny 5}$ at 157  223
	\pinlabel $\textcolor{blue}{\beta_3^+}$ at 162 241
	\pinlabel $\textcolor{gray}{\tiny 6}$ at 157 256
	\pinlabel $\textcolor{blue}{\beta_1^+}$ at 162 273
	\pinlabel $y_2$ at -8 10
	\pinlabel $y_1$ at -8 246
	\pinlabel $x_2$ at 83 186	
	\pinlabel $x_3$ at 83 217	
	\pinlabel $x_4$ at 83 248	
	\pinlabel $x_1$ at 83 278
	\pinlabel $x_1$ at 104 113
	\pinlabel $y_1$ at 95 153
	\pinlabel $y_2$ at 130 153
\endlabellist
\includegraphics[scale=0.8]{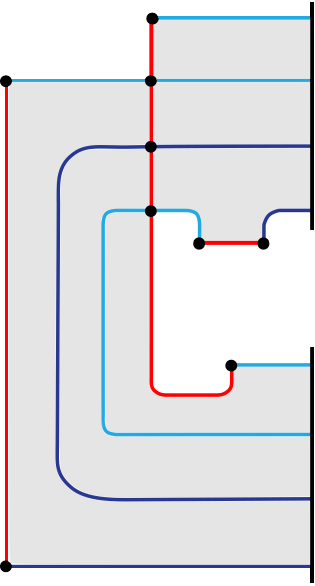}
\caption{The portion of the Heegaard diagram $\HD$ from Figure~\ref{fig:HD5_diagram_labelled} relevant to the computation of $\bsahat(\overline{\HD})$. The boundary is parametrized by an arc diagram which consists of two arcs with eight points on them. The intervals between points are labeled with $\{1,\ldots,6\}$ in gray.}
\label{fig:HDTorus-A}
\end{center}
\end{figure} 
   
Every region is adjacent to the boundary, so $m_1=0$. We computed $m_i$ for $i>1$, by finding the local coefficients of curves and computing their Maslov indices. Then, it is easy to see that the curves with Maslov index one contribute to the maps listed below. Moreover, $m_i=0$ for $i>3$. Note that instead of \[m_2(x_3y_2,I_{\{2,4\}}a(\{1,5\}))=x_3y_1\]
we write simply \[m_2(x_3y_2,\{1,5\})=x_3y_1,\]
as the appropriate completion of a set of Reeb chords is clear from the context. 
The nontrivial maps are as follows. 

\begin{align*}
&m_2(x_3y_2, 5)=x_4y_2\\
&m_2(x_3y_2,\{1,5\})=x_3y_1\\
&m_2(x_3y_2,\{2,4\})=x_3y_1\\
&m_2(x_3y_2,2)=x_2y_2\\
&m_2(x_3y_2,23)=x_1y_2\\
&m_2(x_3y_2,56)=x_1y_2\\
&m_2(x_3y_2,\{123,5\})=x_1y_1\\
&m_2(x_3y_2,\{2,456\})=x_1y_1\\
&m_2(x_4y_2,123)=x_1y_1\\
&m_2(x_4y_2,1)=x_3y_1\\
&m_2(x_4y_2,6)=x_1y_2\\
&m_2(x_2y_2,4)=x_3y_1\\
&m_2(x_2y_2,456)=x_1y_1\\
&m_2(x_2y_2,3)=x_1y_2\\
&m_2(x_3y_1,23)=x_1y_1\\
&m_2(x_3y_1,56)=x_1y_1\\
&m_3(x_4y_2,\{1,6\},5)=x_1y_1\\
&m_3(x_4y_2,1,\{5,6\})=x_1y_1\\
&m_3(x_4y_2,\{4,6\},2)=x_1y_1\\
&m_3(x_4y_2,4,\{2,6\})=x_1y_1\\
&m_3(x_1y_2,1,5)=x_1y_1\\
&m_3(x_1y_2,4,2)=x_1y_1\\
&m_3(x_2y_2,4,\{2,3\})=x_1y_1\\
&m_3(x_2y_2,\{3,4\},2)=x_1y_1\\
&m_3(x_2y_2,1,\{3,5\})=x_1y_1\\
&m_3(x_2y_2,\{1,3\},5)=x_1y_1\\
&m_3(x_3y_2,\{1,23\},5)=x_1y_1\\
&m_3(x_3y_2,\{4,56\},2)=x_1y_1\\
&m_3(x_3y_2,\{4,23\},2)=x_1y_1\\
&m_3(x_3y_2,\{1,56\},5)=x_1y_1\\
&m_3(x_3y_2,\{1,2\},\{3,5\})=x_1y_1\\
&m_3(x_3y_2,\{4,5\},\{2,6\})=x_1y_1\\
&m_3(x_3y_2,\{2,4\},\{2,3\})=x_1y_1\\
&m_3(x_3y_2,\{1,5\},\{5,6\})=x_1y_1
\end{align*}

\end{example}

\begin{example}\label{ex:torus2}
Our next example is the solid torus in $(S^3, \xi_{\mathrm{std}})$ which is the complement of the previous example. 
 As in Example~\ref{ex:ball2}, the embedded foliated open book arising from the disk open book for $(S^3, \xi_{\mathrm{std}})$ is not sorted.  A sequence of two stabilizations results in a sorted foliated open book,  the first and last page of which are shown in  Figure~\ref{fig:torus2stab}.

\begin{figure}[h]
	\labellist
	\pinlabel {$S_0$} [ ] at 67 3
	\pinlabel {$S_4$} [ ] at 146 3
	\pinlabel {$\gamma_1^+$} [ ] at 11 54
	\pinlabel {$\gamma_2^+$} [ ] at 53 67
	\pinlabel {$\gamma_3^-$} [ ] at 172 10
	\pinlabel {$\gamma_4^-$} [ ] at 200 25
	\endlabellist
	\begin{center}
		\includegraphics[scale=1]{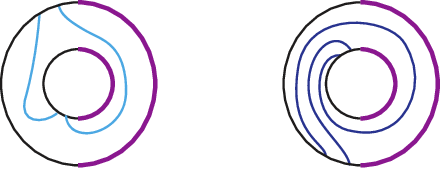}
		\caption{The first and last page of a sorted foliated open book for the solid torus of Example~\ref{ex:torus2}.}\label{fig:torus2stab}
	\end{center}
\end{figure}
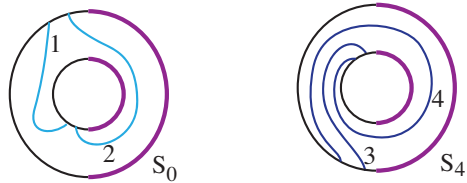

The construction from Section~\ref{ssec:bordereddiagram} yields the Heegaard diagram $\HD'$ shown in Figure~\ref{fig:HD-torus-D}.

 \begin{figure}[h]
 	\labellist
 	\pinlabel {$S_\varepsilon$} [ ] at 15 0
 	\pinlabel {$-S_0$} [ ] at 308 0
 	\pinlabel {$B$} [] at 163 61
 	\pinlabel {$e_+$} [b] at 75 140
 	\pinlabel {$e_-$} [t] at 75 101
 	\pinlabel {$e_+$} [b] at 75 43
 	\pinlabel {$e_-$} [t] at 75 3
 	\pinlabel {$e_+$} [b] at 258 140
 	\pinlabel {$e_-$} [t] at 258 101
 	\pinlabel {$e_+$} [b] at 258 43
 	\pinlabel {$e_-$} [t] at 258 3
 	\pinlabel {\tiny $x_1$} [ ] at 27 93
 	\pinlabel {\tiny $y_1$} [ ] at 121 86
 	\pinlabel {\tiny $y_2$} [ ] at 208 82
 	\pinlabel {\tiny $x_2$} [ ] at 235 35
 	\pinlabel {\tiny $x_3$} [ ] at 288 33
 	\pinlabel {\tiny $x_4$} [ ] at 309 74
 	\pinlabel {\tiny $x_5$} [ ] at 298 108
 	\pinlabel {\tiny $x_6$} [ ] at 276 122
 	\endlabellist
	\begin{center}
		\includegraphics[scale=0.8]{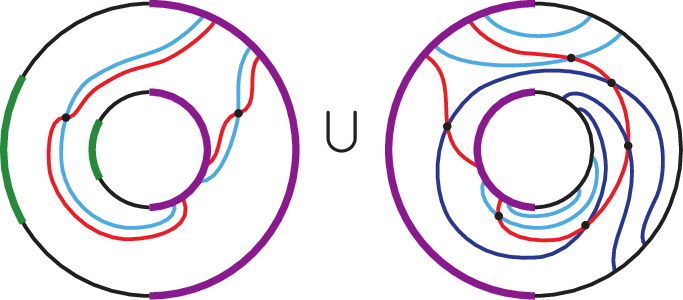}
	\end{center}
		\caption{The bordered sutured Heegaard diagram $\HD'$ adapted to the sorted foliated open book depicted in Figure~\ref{fig:torus2stab}.}\label{fig:HD-torus-D}
\end{figure}
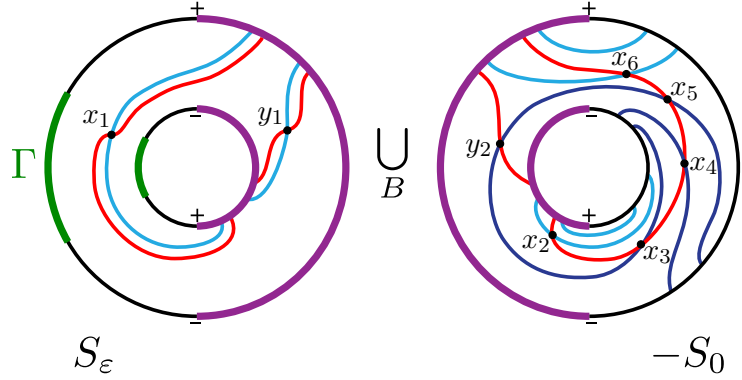

 We label the intersection points in Heegaard diagram $\HD'$ by $x_1'$, $x_2'$, $x_3'$, $x_4'$,  $x_5'$, $x_6'$, $y_1'$, and $y_2'$ as in Figure \ref{fig:HD-torus-D}. We have eight generators\[\left\{x_1'y_1', x_1'y_2', x_2'y_2',x_3'y_1',x_4'y_1',x_4'y_2', x_5'y_1', x_6'y_2'\right\}\] where  $x_1'y_1'$ is the generator  $x_2^+x_4^+$ defined in Section~\ref{ssec:bordereddiagram}. By removing the domains adjacent to  $\bdy\Sigma\setminus Z$ and then cutting along the points $x_1', x_4', y_1'$ and $y_2'$ we get the planar depiction of the Heegaard diagram $\HD'$ shown in Figure~\ref{fig:HDTorus-DO}. 
By matching elliptic points again, as in Example~\ref{ex:ball2}, we obtain labels $\{1, \ldots, 6\}$ on the intervals of the arc diagram for $\HD'$ consistent with those for the diagram $\HD$ from Example~\ref{ex:torus1}.
 
  \begin{figure}[h]
	\begin{center}
		\labellist
		     \pinlabel $\textcolor{blue}{\beta_2^+}$ at -8  10
		     \pinlabel $\textcolor{blue}{\beta_4^+}$ at -8  44
		     \pinlabel $\textcolor{blue}{\beta_3^-}$ at -8  78
		     \pinlabel $\textcolor{blue}{\beta_1^-}$ at -8  112
   		     \pinlabel $\textcolor{blue}{\beta_2^+}$ at -8  148
		     \pinlabel $\textcolor{blue}{\beta_4^+}$ at -8  180
		     \pinlabel $\textcolor{blue}{\beta_3^-}$ at -8  218
		      \pinlabel $\textcolor{blue}{\beta_1^-}$ at -8  250
		      \pinlabel $x_1'$ at 57  5
		      \pinlabel $x_2'$ at 57  35
		      \pinlabel $x_3'$ at 57  88
		      \pinlabel $x_4'$ at 57  115
		      \pinlabel $x_1'$ at 57  145
		      \pinlabel $x_6'$ at 57  173
		      \pinlabel $x_5'$ at 57  228
		      \pinlabel $x_4'$ at 57  255
		      \pinlabel $y_1'$ at 88  35
		      \pinlabel $y_2'$ at 88  88
		      \pinlabel $y_1'$ at 88  173
		      \pinlabel $y_2'$ at 88 228
		      \pinlabel $\textcolor{gray}{\tiny 1}$ at -8  27
		      \pinlabel $\textcolor{gray}{\tiny 2}$ at -8  61
		      \pinlabel $\textcolor{gray}{\tiny 3}$ at -8  95
		      \pinlabel $\textcolor{gray}{\tiny 4}$ at -8  164
		      \pinlabel $\textcolor{gray}{\tiny 5}$ at -8  198
		      \pinlabel $\textcolor{gray}{\tiny 6}$ at -8  232
	\endlabellist
	\includegraphics[scale=0.8]{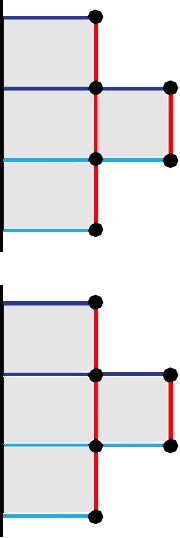}
		\end{center} 
		\caption{The portion of the Heegaard diagram $\HD'$ from Figure~\ref{fig:HD-torus-D} relevant to the computation of $\bsdhat(\overline{\HD'})$.}\label{fig:HDTorus-DO}
\end{figure}

This time we have a nice diagram, and the type $D$ structure is straightforward to compute. The type $D$ structure maps are listed below.
\begin{align*}
\delta^1(x_1'y_1') &= 0\\
\delta^1(x_1'y_2') &= 2\otimes x_1'y_1' + 5\otimes x_1'y_1'\\
\delta^1(x_2'y_2') &= I\otimes x_3'y_1' + 1\otimes x_1'y_2'\\
\delta^1(x_3'y_1') &= 12\otimes x_1'y_1' \\
\delta^1(x_4'y_1') &= 123\otimes x_1'y_1' + 456\otimes x_1'y_1' + 3\otimes x_3'y_1' + 6\otimes x_5'y_1'\\
\delta^1(x_4'y_2') &= 123\otimes x_1'y_2' + 456\otimes x_1'y_2' + 23\otimes x_2'y_2' + 2\otimes x_4'y_1' + 5\otimes x_4'y_1' + 56\otimes x_6'y_2'\\
\delta^1(x_5'y_1') &= 45\otimes x_1'y_1' \\
\delta^1(x_6'y_2') &= 4\otimes x_1'y_2' + I\otimes x_5'y_1' 
\end{align*}
\end{example}

\begin{example}
We take the box tensor product of the type $A$ structure from Example~\ref{ex:torus1} and the type $D$ structure from Example~\ref{ex:torus2}, to compute the contact class for $(S^3, \xi_{\mathrm{std}})$. 

Considering the idempotents, $\bsahat(\overline{\HH})\sqtens\bsdhat(\overline{\HH'})$ has $8$ generators: 
\[\left\{\begin{split}
&x_1y_1\otimes x_1'y_1'\ ,\ x_1y_2\otimes x_2'y_2'\ ,\ x_1y_2\otimes x_3'y_1'\ ,\ x_1y_2\otimes x_5'y_1'\\
&x_1y_2\otimes x_6'y_2'\ ,\ x_2y_2\otimes x_4'y_1'\ ,\ x_3y_2\otimes x_4'y_2'\ ,\ x_4y_2\otimes x_4'y_1'
\end{split}\right\}.\]

Moreover,  the differential $\bdy=\bdy^{\sqtens}$ is given as follows: 
\begin{align*}
\bdy\left(x_3y_2\otimes x_4'y_2'\right)&=x_2y_2\otimes x_4'y_1'+x_4y_2\otimes x_4'y_1'+x_1y_2\otimes x_2'y_2'+x_1y_2\otimes x_6'y_2'\\
\bdy\left(x_2y_2\otimes x_4'y_1' \right)&=x_1y_2\otimes x_3'y_1'+x_1y_1\otimes x_1'y_1'\\
\bdy\left(x_4y_2\otimes x_4'y_1'\right)&=x_1y_2\otimes x_5'y_1'+x_1y_1\otimes x_1'y_1'\\
\bdy\left(x_1y_2\otimes x_2'y_2'\right)&=x_1y_2\otimes x_3'y_1'+x_1y_1\otimes x_1'y_1'\\
\bdy\left(x_1y_2\otimes x_6'y_2'\right)&=x_1y_2\otimes x_5'y_1'+x_1y_1\otimes x_1'y_1'
\end{align*}
It is easy to see that the homology is generated by $[x_1y_1\otimes x_1'y_1']$ and $[x_1y_2\otimes x_2'y_2'+x_2y_2\otimes x_4'y_1']$. Therefore, $[x_1y_1\otimes x_1'y_1']$ is the contact class $(S^3, \xi_{\mathrm{std}})$. Observe that this agrees with the result in Example~\ref{ex:ball-glue}.

\end{example}

\bibliographystyle{alpha}

\bibliography{master}

\end{document}